\documentclass[12pt,epsfig,amsfonts]{article}

\usepackage{amsthm}
\usepackage{amsmath}
\usepackage{amssymb}
\usepackage[margin=1in]{geometry}
\usepackage{enumerate}
\usepackage{hyperref}

\usepackage{mathrsfs}

\theoremstyle{theorem}
\newtheorem{thm}{Theorem}[section]
\newtheorem*{thm*}{Theorem 1}
\newtheorem{cor}[thm]{Corollary}
\newtheorem*{cor*}{Corollary 1}
\newtheorem{lem}[thm]{Lemma}
\newtheorem{prop}[thm]{Proposition}

\theoremstyle{definition}

\newtheorem{defn}[thm]{Definition}
\newtheorem{rmk}[thm]{Remark}

\newtheorem{cla}[thm]{Claim}

\newcommand{\N}{\mathbb{N}}

\newcommand{\R}{\mathbb{R}}

\newcommand{\Z}{\mathbb{Z}}

\newcommand{\Fc}{\mathcal{F}}

\newcommand{\Gc}{\mathcal{G}}

\renewcommand{\a}{\alpha}
\renewcommand{\b}{\beta}

\renewcommand{\d}{\delta}
\newcommand{\e}{\epsilon}
\renewcommand{\l}{\lambda}
\newcommand{\graph}{\operatorname{graph}}
\newcommand{\lip}{\operatorname{Lip}}
\newcommand{\Id}{\text{Id}}

\newcommand{\Bc}{\mathcal{B}}

\renewcommand{\tilde}{\widetilde}

\newcommand{\tf}{\tilde{f}}
\newcommand{\ds}{/ \! /}

\newcommand{\Sc}{\mathcal{S}}

\newcommand{\Oc}{\mathcal{O}}
\newcommand{\Wc}{\mathcal{W}}

\newcommand{\Bs}{\mathscr{B}}
\newcommand{\As}{\mathscr{A}}

\newcommand{\oN}{\overline{N}}

\newcommand{\la}{\lambda}

\newcommand{\td}{\tilde{\delta}}

\newcommand{\tri}[1]{{\left\vert\kern-0.25ex\left\vert\kern-0.25ex\left\vert #1 
    \right\vert\kern-0.25ex\right\vert\kern-0.25ex\right\vert}}

\begin{document}

\title{\bf Entropy, volume growth and SRB measures \\
   for Banach space mappings}

\author{Alex Blumenthal\thanks{Courant Institute of Mathematical Sciences, New York University, New York, USA.  Email: alex.blumenthal@gmail.com.}
\and Lai-Sang Young\thanks{Courant Institute of Mathematical Sciences, New York University, New York, USA.  Email: lsy@cims.nyu.edu. This research was supported in part by NSF Grant DMS-1363161.}
}

\maketitle

\begin{abstract} We consider $C^2$ Fr\'echet differentiable mappings of Banach spaces
leaving invariant compactly supported Borel probability measures, and study the relation
between entropy and volume growth for a natural notion of volume defined on finite dimensional subspaces. % of the Banach spaces in question. 
SRB measures are characterized as exactly those measures
for which entropy is equal to volume growth on unstable manifolds, equivalently the sum of
positive Lyapunov exponents of the map. In addition to numerous difficulties incurred by 
our infinite-dimensional setting, a crucial aspect to the proof is the technical point that
the volume elements induced on unstable manifolds are regular enough to permit
distortion control of iterated determinant functions.
The results here generalize previously known results for diffeomorphisms of finite dimensional Riemannian manifolds, and are applicable to dynamical systems defined by large classes of dissipative parabolic PDEs.
\end{abstract}

\tableofcontents

\vskip .6in

This paper is part of a program to expand the scope of smooth ergodic theory,
with a view towards making it applicable to PDEs as well as ODEs. 
Here, we extend to Banach space mappings 
an important result known in finite dimensions, namely the
characterization of SRB measures as  invariant measures for which
 entropy attains its upper bound given by the rate of unstable volume growth. 
 The class of mappings to which our results apply includes 
(but is not limited to) time-t maps 
of semiflows defined by periodically forced nonlinear dissipative parabolic PDEs.

Orbits tend to attractors in dissipative dynamical systems. It is often assumed
in the physics literature that asymptotic behaviors of ``typical" orbits are captured 
by certain special invariant measures called SRB measures \cite{eckmannRuelle}, 
in the sense that their time averages  tend to space averages taken with respect
to these measures. Mathematically, proving the existence of SRB measures 
poses nontrivial challenges, but for finite dimensional systems it has been shown
that when they exist, ergodic SRB measures have the properties above, confirming
that they are, to dissipative systems, what Liouville measures are to Hamiltonian systems. 
See \cite{youngMathTheoryLE} for a review. Without a doubt, 
SRB measures (named after Sinai, Ruelle and Bowen, who discovered them 
for uniformly hyperbolic attractors) are among the most important ideas in 
finite dimensional theory. This paper extends to Banach space mappings,
including time-$t$ maps of semiflows generated by certain kinds of PDEs, 
the characterization of SRB
measures in terms of two much studied dynamical invariants, {\it metric entropy} 
and {\it Lyapunov exponents}. 

Entropy measures the growth in
randomness in the sense of information theory as a transformation is iterated, while 
Lyapunov exponents measure geometric instability: they give the rates at which nearby orbits diverge. Though {\it a priori} quite different, these two ways to capture dynamical complexity 
are in fact closely related: For a differentiable mapping of a finite dimensional Riemannian
manifold preserving a compactly supported Borel probability measure, it was shown 
by Ruelle \cite{ruelle} that entropy is always dominated by the sum of positive Lyapunov exponents (counted with multiplicity). In the case of volume preserving diffeomorphisms, Pesin
showed that the two quantities above are in fact equal \cite{pes}; another proof was later given
by Ma\~n\'e \cite{mane}. The ultimate results in this direction are contained in the 
combined works of Ledrappier {\it et. al.} \cite{ledstrel, led, ledyou1, ledyou2}, 
which identified the SRB property of the invariant measure as both 
necessary and sufficient for the entropy formula to hold, and related the gap
 in this formula to the dimension of the invariant measure in general.

The results of  \cite{ledstrel, led, ledyou1} reinforce the view 
that the ``effective dimension" of a dynamical system
is equal to its number of positive Lyapunov exponents, in that all of the
dynamical complexity of a system is captured in its expanding directions. 
%In this view, SRB measures, which have densities on unstable manifolds, play 
%the role of Liouville measure for Hamiltonian systems. 
Even though it is impossible to mathematically carry out a dimension reduction  
procedure along the lines described, these ideas are conceptually valid.
SRB measures and their entropy formula characterization are therefore
especially relevant for systems with many degrees of freedom and
relatively low effective dimensions, such as time-$t$ maps of 
semiflows defined by nonlinear dissipative parabolic PDEs.
For background information on dynamical systems generated by parabolic
PDEs, see e.g. \cite{henry}.

%
%. In principle, then, it should be possible to
%generalize this property of SRB measures to systems with infinitely many degrees of freedom,
%so long as the effective dimension (the number of positive Lyapunov exponents) 
%remains finite; this is the case, e.g., for time-$t$ maps of 
%semiflows defined by nonlinear dissipative parabolic PDEs.
%
%Among the applications we have in mind
%are time-t maps of semiflows defined by nonlinear dissipative parabolic PDEs-
%systems defined by such PDEs are important examples of infinite dimensional 
%systems with relatively low effective dimensions.

\bigskip \noindent
{\it Technical issues associated with ergodic theory on Banach spaces}

\smallskip
The main results of this paper generalize \cite{ledstrel} and \cite{led}
to Banach space mappings $f$ preserving a compactly supported Borel 
probability measure $\mu$ with finitely many positive Lyapunov exponents.
We prove that under the condition of no zero Lyapunov exponents, 
$\mu$ is an SRB measure if and only if the entropy of $f$ is equal to 
the sum of its positive Lyapunov exponents. 

A number of results in nonuniform hyperbolic theory have been
extended to Hilbert space mappings \cite{thieu,  lianyou1, lianyou2, linshu1, linshu2, lianyou3,
luwangyoung}, and some have further been extended to mappings of Banach space, 
e.g. Ruelle's inequality for deterministic and random maps \cite{thieu, linshu2} and 
the absolute continuity of the stable foliation for systems with invariant cones
\cite{lianyou3}. Observe, however, that ideas surrounding the entropy formula and 
SRB measures have not been studied for systems on Banach spaces.
%aside from systems admitting finite-dimensional center manifolds \cite{luwangyoung}. 
A hurdle might be that in one way or another, these ideas are related to
volume growth on finite dimensional (unstable) manifolds, and in Banach spaces
there is no intrinsic notion of $k$-dimensional volume for $k>1$. While one may be
able to make do with a Lebesgue measure class on unstable manifolds
(Haar measure is certainly well defined), we believe a systematic understanding 
of volume growth is conducive to understanding SRB measures and the relation 
between entropy and Lyapunov exponents. 

It is simple enough to put a notion of volume on a fixed finite dimensional normed 
vector space, and one can do that -- one subspace at a time -- for all finite dimensional 
subspaces of a Banach space. But for such a notion to be useful in smooth ergodic theory, 
{\it regularity} of this volume function as subspaces are varied is essential. 
It is well known that norms do not necessarily vary smoothly with vectors on Banach spaces; 
volumes and determinants are not likely to fare better. Hence it is important that
 the volumes we introduce are regular enough to support distortion  estimates 
 on unstable manifolds, as such bounds
are key to many important results in hyperbolic theory. We will show
that they have the regularity we need, but it is not clear that finite dimensional results
involving higher regularity of determinants, e.g. \cite{ruelle2}, will carry over to
Banach spaces.
%
%
%Indeed they are, as we will show, 
%but we cannot guarantee
%much more: for example, unlike the case on Riemannian manifolds,
% it does {\it not} follow that higher regularity
%of the mapping will lead to higher regularity of conditional densities on unstable manifolds
%for SRB measures; see \cite{ruelle2}.

In addition to the absence of an intrinsic notion of volume, another difficulty
we face has to do with {\it noninvertibility} of the map $f$, which is not onto and has
arbitrarily strong contraction in some directions. Even where $f^{-1}$ is
defined, one cannot expect it to have nice properties. This leads to regularity issues
for objects such as $E^u$-spaces the definitions of which involve backward iterations.
In response to these difficulties, throughout this paper we have tried to identify 
differences between diffeomorphisms and
maps that are not invertible, and differences between finite and infinite dimensions. 
We have taken special care in 
the treatment of volume growth, recognizing that Banach spaces do not always admit 
a notion of volume as nice as that on Hilbert spaces or 
on finite dimensional Riemannian manifolds.

\bigskip
The organization of this paper is as follows: The main results are stated in
Section 1. Section 2 contains a discussion of volumes and determinants on
finite dimensional subspaces of Banach spaces; part of this material is included for
the convenience of the reader, and other parts (e.g. regularity of determinants) are new.
We hope this basic material will be useful beyond the present paper. 
Section 3 contains a small addendum to the Multiplicative Ergodic
Theorem, following up on volume growth ideas in relation to Lyapunov exponents. Sections 4 and 5 contain preparations for the proofs of our main results, such as Lyapunov charts, distortion estimates, etc. Additional technical issues and
the proofs of the main results are carried out in Section 6. 
%Once we have properly set the stage and dealt with some technical issues arising from our infinite-dimensional setting, the proofs in Section 6 are, in fact, quite similar to those in
%finite dimensions. Throughout this paper, we have tried to identify 
%differences between diffeomorphisms and
%maps that are not invertible, and differences between finite and infinite dimensions. 
%We have taken special care  in 
%the treatment of volume growth, recognizing that Banach spaces do not always admit 
%a notion of volume as nice as that on Hilbert spaces or 
%on finite dimensional Riemannian manifolds.

%%%%%%%%%%%%%%%%%%%%%%%%%%%%%%%%%%%%%%%%%%%%%
%%%%%%%%%%%%%%%%%%%%%%%%%%%%%%%%%%%%%%%%%%%%%
\section{Statement of Results}

%Let $(\Bc, |\cdot|)$ be a Banach space. After some preliminaries in which
%we introduce a notion of volume on finite dimensional subspaces of Banach
%spaces, we turn to the main topic of this paper, nonuniform hyperbolic theory for 
%Banach space mappings. We begin with some basic facts of this theory, 
%proved under 
%conditions (H1)--(H3) below.

Let $(\Bc, |\cdot|)$ be a Banach space. After some preliminary work fixing a notion
of volume on finite-dimensional subspaces of $\Bc$ (Section 2), 
we turn to the main topic of this paper, 
nonuniform hyperbolic theory for Banach space mappings. We begin with some basic facts of this theory, 
proved under conditions (H1)--(H3) below.

\bigskip \noindent
{\bf Setting for basic nonuniform hyperbolic theory.} We consider $(f, \mu)$,
where $f : \Bc \to \Bc$ is a mapping and $\mu$ is an $f$-invariant Borel probability measure. The following properties are assumed: 
\begin{itemize} 
\item[(H1)] (i) $f$ is $C^{2}$ Fr\'echet differentiable and injective; 

\vspace{-3 pt}
(ii) the derivative of $f$ at $x \in \Bc$, denoted $df_x$, is also injective.

\vspace{-3 pt}
\item[(H2)] (i) $f$ leaves invariant a compact set $\As \subset \Bc$, with $f(\As)=\As$;

\vspace{-4 pt}
(ii) $\mu$ is supported on $\As$. 

\vspace{-3 pt}
\item[(H3)] We assume 
$$l_\alpha(x):=\lim_{n\to\infty}\frac1n \log |df^n_x|_\alpha \ < \ 0 
\quad \mbox{ for } \ \mu-\text{a.e. }x\ . $$
Here $|df^n_x|_\alpha$ is the Kuratowski measure of noncompactness of 
the set $df^n_x(B)$, where
$B$ is the unit ball in $\Bc$. 
\end{itemize}

Condition (H3) is discussed in more detail in Sect. 3.1. It is a relaxation of 
the condition that $df_x$ is the sum of a compact operator and a contraction 
for each $x \in \As$ (see Remark \ref{rmk:impliesH3}), and it implies 
that positive and zero Lyapunov exponents of $(f,\mu)$ 
are well defined and have finite multiplicity. 

\bigskip
\noindent
{\bf Two other relevant assumptions.} \vspace{-3 pt}
\begin{itemize}
\item[(H4)] $(f,\mu)$ has no zero Lyapunov exponents.

\vspace{-3 pt}
\item[(H5)] the set $\As$ in (H2) has finite box-counting dimension.
\end{itemize}

\smallskip
We remark that (H5) is automatically satisfied if $f$ satisfies (H1) and (H2)(i),
 and $df_x$ is the sum of a compact operator and a contraction for each $x \in \As$;
see \cite{maneDimension}.
 
\medskip
For diffeomorphisms of Riemannian manifolds, one generally requires in the definition
of SRB measures 
 that the conditional measures of $\mu$ on unstable manifolds
be absolutely continuous with respect to the Riemannian measures induced on 
these manifolds. In Banach spaces, the notion of Riemannian volume is absent, 
but there is the following well defined {\it Lebesgue measure class} on any finite dimensional submanifold $W$: For $x \in W$, we let $\Bc_x$ denote
the tangent space to $\Bc$ at $x$, and choose a closed subspace $F$ so that $\Bc_x = E \oplus F$ where $E$ is the subspace
tangent to $W$ at $x$. Then on a small neighborhood $U$ of $x$ in $W$, 
the ``Lebesgue measure class" is the one that when projected to $E$ along
$F$ gives the Haar measure class on $E$. We state below a provisional definition
of SRB measures; see Sect. 6.1 for a formal definition.

\smallskip
\begin{defn} \label{defn:introSRB}
We say $\mu$ is an {\it SRB measure} if (i) it has a positive Lyapunov exponent 
$\mu$-a.e. and (ii) the conditional measures of $\mu$ on unstable manifolds are 
in the ``Lebesgue measure class" induced on these manifolds.
\end{defn}

%\smallskip
%For diffeomorphisms of Riemannian manifolds, one generally requires in the definition
%of SRB measures 
% that the conditional measures of $\mu$ on unstable manifolds
%be absolutely continuous with respect to the Riemannian measures induced on 
%these manifolds. In Banach spaces, the notion of Riemannian volume is absent, 
%but there is the following well defined {\it measure class} on any finite dimensional submanifold $W$: For $x \in W$, we let $\Bc_x$ denote
%the tangent space to $\Bc$ at $x$, and choose a closed subspace $F$ so that $\Bc_x = E \oplus F$ where $E$ is the subspace
%tangent to $W$ at $x$. Then on a small neighborhood $U$ of $x$ in $W$, 
%the ``Lebesgue measure class" is the one that when projected to $E$ along
%$F$ gives the Haar measure class on $E$. More detail will be given in the text.

Let $(f,\mu)$ be as above. We let $h_\mu(f)$ denote the entropy
with respect to $\mu$, and let $\lambda_1(x)> \lambda_2(x) > \cdots$, with multiplicities
$m_1(x), m_2(x), \dots$, denote the distinct Lyapunov exponents of $(f,\mu)$ at $x$.
Write $a^+ = \max\{a, 0\}$. Our main results are the following:

\bigskip \noindent
{\bf Theorem 1.} {\it Suppose $(f,\mu)$ satisfies (H1)--(H4) above, and assume that $\mu$ is an SRB measure. Then
\begin{equation} \label{entropyformula}
h_\mu(f) = \int \sum_i m_i(x) \lambda_i^+(x)  \ d\mu\ .
\end{equation}
}

\smallskip \noindent
{\bf Theorem 2.} {\it Suppose $(f,\mu)$ satisfies (H1)--(H5). If $\lambda_1>0$ $\mu$-a.e.
and the entropy formula (\ref{entropyformula}) holds, then $\mu$ is an SRB measure.}

\bigskip
The results in Theorems 1 and 2 were proved in \cite{ledstrel},\cite{led} in a finite dimensional context,
more precisely for diffeomorphisms of compact Riemannian manifolds, and extended to Hilbert spaces in \cite{linshu1}. In all likelihood, the no zero exponents assumption (H4)  is not necessary, but in the presence of zero Lyapunov exponents, the proofs are more
elaborate and we have elected to treat that case elsewhere.

One way to understand Theorem 1 is to view the sum of positive
Lyapunov exponents as representing  {\it volume growth} on unstable manifolds, 
so that the right side of (\ref{entropyformula}) tells us how
volumes on unstable manifolds are transformed, while entropy describes the
transformation of the conditional measures of $\mu$. To express
these ideas in a systematic way, we need a coherent notion of {\it volume} on unstable
manifolds, not just a measure class (which was sufficient for purposes of defining
SRB measures). We know of no previous studies of volumes on finite dimensional
subspaces of Banach spaces that serve our purposes, the closest approach to these ideas being the `volume function' in \cite{lianlu} (see also \cite{quasConcise}), which does not
arise from a genuine volume on subspaces in the usual sense.
Thus we include in Section 2 a short introduction to these ideas. 

With a coherent notion of induced volumes on finite dimensional subspaces in hand,
we have a well defined notion of finite dimensional {\it determinant} for $df_x$, 
and a result to the following effect:

\medskip \noindent
{\bf Corollary 3.} {\it The conditional densities of an SRB measure on unstable manifolds are Lipschitz
and have the form
$$
\frac{\rho(x)}{\rho(y)} = \prod_{i = 1}^{\infty} \frac{ \det(df_{f^{-i}y}|_{T_{f^{-i} y} W})}{ \det(df_{f^{-i}x}|_{T_{f^{-i} x} W})}
$$
for all $x,y$ on the same local unstable manifold $W$.}

\medskip

A precise statement of this result requires some preparation and is given in Sect. 6.5.

%%%%%%%%%%%%%%%%%%%%%%%%%%%%%%%%%%%%%%%%%%%%%%%%%%%%%%%%%%%%%%%%%%%%%
%%%%%%%%%%%%%%%%%%%%%%%%%%%%%%%%%%%%%%%%%%%%%%%%%%%%%%%%%%%%%%%%%%%%%
%%%%%%%%%%%%%%%%%%%%%%%%%%%%%%%%%%%%%%%%%%%%%%%%%%%%%%%%%%%%%%%%%%%%%
\section{Volumes on Finite Dimensional Subspaces of a Banach Space}
%%%%%%%%%%%%%%%%%%%%%%%%%%%%%%%%%%%%%%%%%%%%%%%%%%%%%%%%%%%%%%%%%%%%%
%%%%%%%%%%%%%%%%%%%%%%%%%%%%%%%%%%%%%%%%%%%%%%%%%%%%%%%%%%%%%%%%%%%%%
%%%%%%%%%%%%%%%%%%%%%%%%%%%%%%%%%%%%%%%%%%%%%%%%%%%%%%%%%%%%%%%%%%%%%

Whereas in a Hilbert space, a finite dimensional subspace is naturally an inner product space with an obvious choice of volume element, there is no such `obvious' choice in a Banach space. The objective of this section is to introduce a coherent notion of volume on finite-dimensional subspaces of a Banach space, and to establish some basic properties.
Definitions and basic facts of induced volumes and determinants are given in 
Sects. 2.1 and 2.2. Their regularity, which are the main results of this section,
are proved in Sect. 2.3. As noted in the Introduction, regularity of the determinant
is relevant for controlling the distortion of iterated densities on unstable manifolds.

We assume throughout that $(\Bc, |\cdot|)$ is a Banach space.

%%%%%%%%%%%%%%%%%%%%%%%%%%%%%%%%%%%%%%%%%%%%%%%%%%%%%%%%%%%%%%%%%%
\subsection{Relevant Banach space geometry (mostly review)}
%%%%%%%%%%%%%%%%%%%%%%%%%%%%%%%%%%%%%%%%%%%%%%%%%%%%%%%%%%%%%%%%%%

We gather in this subsection some known facts that are relevant
to smooth ergodic theory on Banach spaces, casting them in a light suitable for our purposes.

%%%%%%%%%%%%%%%%%%%%%%%%%%%%%%%%%
\subsubsection{Induced volumes}
%%%%%%%%%%%%%%%%%%%%%%%%%%%%%%%%%

Following the idea of the Busemann-Hausdorff volume in Finsler geometry \cite{buse}, \cite{rund}, we make the following definition.

\begin{defn}
Let $E \subset \Bc$ be a finite-dimensional subspace. We define the \emph{induced volume }$m_E$ on $E$ to be the unique Haar measure on $E$ for which
$$
m_E \{u \in E \mid |u| \leq 1\} = \omega_k
$$
where $k = \dim E$ and $\omega_k$ is the volume of the Euclidean unit ball in $\R^k$.
\end{defn}

The following are some basic properties of $m_E$ entailed by this definition.

\begin{lem} \label{lem:indMeasProps}
Let $E \subset \Bc$ be a $k$-dimensional subspace. Then $m_E$ satisfies the following.
\begin{enumerate}
\vspace{-4 pt}
\item For any $v \in E$ and any Borel measurable set $S \subset E$, we have $m_E(v + S) = m_E (S)$.
\vspace{-4 pt}
\item If $m'$ is any other $\sigma$-finite non-zero measure on $E$ satisfying item 1, then 
$m'$ and $m_E$ are equivalent with $\ \frac{d m'}{dm_E} \equiv c \ \ m_E$-a.e. for a constant $c>0$.
\vspace{-4 pt}
\item For any $a > 0$ and any Borel measurable set $S \subset E$, we have $m_E (a S) = a^k m_E(S)$.
\end{enumerate}
\end{lem}

%%%%%%%%%%%%%%%%%%%%%%%%%%%%%%%%%
\subsubsection{Complementation and `angles'}
%%%%%%%%%%%%%%%%%%%%%%%%%%%%%%%%%

%We say that a subspace $E$ is \emph{complemented} if there exists a bounded linear operator 
%$\pi : \Bc \to E$ for which $\pi|_{E} = \id |_E$. 
%We call $\ker \pi := F$ a {\em (topological) complement} to $E$, and
%call $\Bc = E \oplus F $ a {\it topological splitting}. 

Let $\Gc(\Bc)$ denote the Grassmanian of closed subspaces of $\Bc$.  The topology on $\Gc(\Bc)$ is the metric topology defined by the Hausdorff distance $d_H$ between unit spheres: for nontrivial subspaces $E, E' \in \Gc(\Bc)$,
$$
d_H(E, E') = \max \{ \sup\{d(e, S_{E'}) : e \in S_E \}, \sup\{d(e', S_E) : e' \in S_{E'}\}\}
$$
where $S_E =\{v \in E \,|\, |v|=1\}$. 

%\begin{rmk} \label{rmk:aperture}
A more convenient definition, known as the \emph{aperture} or \emph{gap} (\cite{kato}, see also \cite{akh}), is
$$
\d_a(E, E') = \max\{ \sup\{d(e, E') : e \in S_E \}, \sup\{d(e', E) : e'  \in S_{E'} \} \}\ .
$$
On Hilbert spaces, $\d_a$ is a metric, and coincides with the operator norm of the difference between orthogonal projections. On Banach spaces, $\d_a$ is not a metric, but $d_H$ and 
$\d_a$ are related by the inequality $\d_a(E, E') \leq d_H(E, E') \leq 2 \d_a(E, E')$ \cite{kato}. 
We will
work with $d_H$ or $\d_a$, whichever one is more convenient. 
%\end{rmk}

We say $E \in \Gc(\Bc)$ is  \emph{complemented} if there exists $F \in \Gc(\Bc)$
such that $\Bc = E \oplus F$, and call $F$ a {\it complement} of $E$. Observe that if
$E, F \in \Gc(\Bc)$ are complements, then $\pi = \pi_{E \ds F}: \Bc \to E$,
the projection to $E$ along $F$ defined by $\pi(e+f)=e$ for $e \in E, f \in F$, is automatically
bounded as an operator by the closed graph theorem. We note further that
$$ 
|\pi_{E \ds F}|^{-1} = \a(E, F)   \qquad \mbox{ where } \qquad
\a(E, F) = \inf\{|e - f | : e \in E, |e| = 1, f \in F\}\ .
$$
Though $\a(\cdot, \cdot)$ is not symmetric, it satisfies $\a(E, F) \leq 2 \a(F, E)$ 
whenever $E, F$ are complements. These quantities have the geometric connotation of `angle' between $E$ and $F$.

The next lemma gives conditions under which complementation persists.

\begin{lem}\label{lem:openCond} 
Let $\Bc = E \oplus F $ for $E, F \in \Gc(\Bc)$. If $E' \in \Gc(\Bc)$ is such that 
$d_H(E, E') < |\pi_{E \ds F}|^{-1}$, then $\Bc = E' \oplus F $.
\end{lem}

A proof of this result is given in the appendix.

Not every $E \in \Gc(\Bc)$ admits a complement, but all finite dimensional subspaces do,
and while there are no orthogonal complements to speak of,
the following result provides a substitute that is adequate for our purposes.

\begin{lem}[\cite{woj}] \label{lem:compExist} 
Every subspace $E \subset \Bc$ of finite dimension $k$ has a complement $F \in \Gc(\Bc)$ 
with the property that $\a(E, F) \geq \frac{1}{\sqrt{k}}$, equivalently, $|\pi_{E \ds F}| \leq \sqrt{k}$.
\end{lem}

The next lemma contains some estimates that are used repeatedly in Sect. 2.3.
Proofs are given in the Appendix.

\begin{lem} \label{lem:graNormEst}
Let $E, E' \subset \Bc$ be subspaces with finite dimension $k$, and let $F\in \Gc(\Bc)$ be 
a complement to $E$ with $|\pi_{E \ds F}| \leq \sqrt{k}$. Suppose 
$d_H(E, E') \leq \frac{1}{2 \sqrt{k}}$. Then:
\begin{itemize}
\vspace{-4 pt}
\item[(a)] $\Bc = E' \oplus F$, with
\vspace{-4 pt}
$$
|\pi_{E' \ds F}|  \leq 2 \sqrt{k} \qquad \mbox{ and }
\qquad |\pi_{F \ds E'}|_E| \leq 4 \sqrt{k}\ d_H(E, E')\ ;
$$
\item[(b)] for any $e \in E$ with $|e|=1$,
\vspace{-4 pt}
$$
1 - 4 \sqrt{k} \ d_H(E, E') \leq |\pi_{E' \ds F} e| \leq 1 + 4 \sqrt{k} \ d_H(E, E') \ .
$$
\end{itemize}
\end{lem}

%%%%%%%%%%%%%%%%%%%%%%%%%%%%%%%%%
\subsubsection{Comparison with norms arising from inner products} \label{sec:johnThm}
%%%%%%%%%%%%%%%%%%%%%%%%%%%%%%%%%

One can leverage known results on inner product spaces by comparing 
$|\cdot|$ to norms that arise from inner products. The key to this direction of thinking
 is John's Theorem \cite{bollo}, which states that a convex body in $\R^n$ 
is contained in a unique volume-minimizing ellipsoid. Since ellipsoids and inner products are equivalent, this result can be stated in terms of inner products. We take the liberty
to state a version of John's theorem that fits with the way it will be used in this paper.

\begin{thm}[John's Theorem] \label{thm:johnEll} 

Let $E \subset \Bc$ be a subspace of finite dimension $k$. Then there is an inner product $(\cdot, \cdot)$ on $E$ and a norm $\|\cdot\|$ arising from it such that for all $v \in E$, 
$$
\|v\| \leq |v| \leq \sqrt{k} \|v\|\ .
$$
\end{thm}

The following is a direct consequence of John's Theorem and Lemma \ref{lem:indMeasProps}.

\begin{cor} \label{lem:orthoBasis} Let $E \subset \Bc$ be a subspace of dimension $k$.
We let $(\cdot,\cdot)$ and $\|\cdot\|$ be given by Theorem \ref{thm:johnEll}.
Scaling $(\cdot,\cdot)$ by a suitable constant, one can obtain a new inner product 
$(\cdot,\cdot)_E$ and norm $\|\cdot\|_E$ on $E$ with the property that if $\hat m_E$ is the
induced volume on $E$ with respect to $\|\cdot\|_E$, then
$$
m_E = \hat m_E \qquad \mbox{ and } \qquad 
\frac{1}{\sqrt{k}} \|\cdot\|_E \leq |\cdot| \leq \sqrt{k} \|\cdot\|_E\ .
$$
\end{cor}

\begin{proof} Let $\tilde{B}_E = \{v \in E \,|\, \|v\| \le 1\}$, and let $C$ be such that
$m_E \tilde{B}_E = C \omega_k$. Scale $(\cdot, \cdot)$ so that 
$\|\cdot\|_E := C^{1/k} \|\cdot\|$. We leave the rest as an exercise.
\end{proof}

%, and
%let $\tilde{B}_E = \{v \in E \,|\, \|v\| \le 1\}$. Rescale $(\cdot,\cdot)$ to $(\cdot,\cdot)_E$ 
%so that the resulting norm $\|\cdot\|_E := C^{1/k} \|\cdot\|$
%where $C$ is given by $m_E \tilde{B}_E = C \omega_k$. 
%Let $\hat m_E$ be the induced measure with respect to $\|\cdot\|_E$.
%Then
%$$
%m_E = \hat m_E \qquad \mbox{ and } \qquad 
%\frac{1}{\sqrt{k}} \|\cdot\|_E \leq |\cdot| \leq \sqrt{k} \|\cdot\|_E\ .
%$$
%\end{cor}

\begin{rmk} For  purposes of this paper
what matters in the results in Theorem \ref{thm:johnEll} and Corollary \ref{lem:orthoBasis}
 is not the bound $\sqrt{k}$ but the fact that there is a bound that depends only 
 on the dimension of the subspace in question. Indeed, any means of constructing an inner product on $E \subset \Bc$ would do, so long as it gives rise to a norm uniformly equivalent
 to the original one, with constants depending only on the dimension of $E$.
% The sole advantage of working with John's Theorem is that it is the optimal (in the sense of volume minimizing) inner product approximating the unit ball of a normed vector space \cite{bollo}.
\end{rmk}

Corollary \ref{lem:orthoBasis} is used many times in the discussion to follow. It
enables us to deduce quickly many results for normed vector spaces
by appealing to their counterparts on inner product spaces. For example, 
for $\{v_1, \dots, v_k\} \subset \Bc$, let $P[v_1, \cdots, v_k]$ denote the 
parallelepiped defined by the vectors $\{v_i\}$, i.e.,
$$
P[v_1, \cdots, v_k] = \{a_1 v_1 + \cdots + a_k v_k : 0 \leq a_i \leq 1 \}\ .
$$
Then given $\{v_i\} \subset E$ and $\lambda_i \in \mathbb R$, we have
relations such as
$$
m_E(P[\lambda_1 v_1, \cdots, \lambda_k v_k])  = 
\left(\prod_{i=1}^k |\lambda_i| \right) m_E(P[v_1, \cdots, v_k])
$$
because this is true for $\hat m_E$, and
\begin{equation} \label{pgram}
m_E(P[v_1, \cdots, v_k]) = \hat m_E(P[v_1, \cdots, v_k]) \le \prod_{i=1}^k \|v_i\|_E
\le k^{\frac{k}{2}} \prod_{i=1}^k |v_i|\ .
\end{equation}

%%%%%%%%%%%%%%%%%%%%%%%%%%%%%%%%%%%%%%%%%%%%%%%%%%%%%%%%%%%%%%%%%%
\subsection{The determinant and its properties}
%%%%%%%%%%%%%%%%%%%%%%%%%%%%%%%%%%%%%%%%%%%%%%%%%%%%%%%%%%%%%%%%%%

Associated with the induced volumes defined in Sect. 2.1, we have, for each linear map 
$A: \Bc \to \Bc$, a notion of {\it determinant} on finite dimensional subspaces
which describes how these measures are 
transformed by $A$.

\begin{defn}
Let $A : \Bc \to \Bc$ be a bounded linear operator and $E \subset \Bc$  a finite-dimensional subspace. Then
$$
\det(A | E) := \begin{cases} \frac{m_{A (E)} (A (B_E))}{m_E(B_E)} & \dim A(E) = \dim E \, , \\ 0 & \text{else,} \end{cases}
$$
where $B_E = \{v \in E : |v| \leq 1\}$.
\end{defn}

It follows from this definition and from Lemma \ref{lem:indMeasProps} that $\det(\cdot)$
has the basic properties of the usual determinant, such as:

\begin{lem} \label{lem:detProps}
Let $E, F, G$ be subspaces of $\Bc$ of the same finite dimension, and let $A, B : \Bc \to \Bc$ be bounded linear maps for which $A E \subset F, B F \subset G$. Then:
\begin{enumerate}
\vspace{-4 pt}
\item $m_F (A (S)) =  \det(A | E) \cdot m_E(S)$ for every Borel set $S \subset E$;
\vspace{-4 pt}
\item $\det (B A|E) = \det (B|F) \cdot \det (A|E)$.
\end{enumerate}
\end{lem}

The following are further illustration of how one can leverage results for inner product 
spaces via John's Theorem. The proofs are left as (easy) exercises.

\begin{lem} \label{lem:johnSVD}
Let $A : \Bc \to \Bc$ be a bounded linear operator, and let $V, V' \subset \Bc$ be
$k$-dimensional subspaces such that $A(V)=V'$. We equip $V$ and $V'$ with 
the inner products $(\cdot, \cdot)_V$ and $(\cdot, \cdot)_{V'}$ in Corollary \ref{lem:orthoBasis}.
\begin{enumerate}
\item If $\{v_1, \dots, v_k\} \subset V$ are orthonormal with respect to 
$(\cdot, \cdot)_V$, then it follows from (\ref{pgram}) that
$$
\det(A | V) = \frac{m_{V'} A(P[v_1, \dots, v_k])}{ m_V P[v_1, \dots, v_k]}
\le k^{\frac{k}{2}} \prod_{i=1}^k |Av_i|  \ .
$$
\item If $\{v_1, \dots, v_k\}$ is an orthonormal basis of $V$ corresponding to the singular 
value decomposition of $A|_V:V \to V'$ , then
$$
k^{-\frac k 2} \prod_{i = 1}^k |A v_i| \leq \det(A|V) \leq k^{\frac k 2} \prod_{i = 1}^k |A v_i| \ .
$$
\end{enumerate}
\end{lem}

The following is how $\det(\cdot)$ behaves with respect to splittings.

\begin{lem} \label{lem:detSplit}
For any $k \geq 1$ there is a constant $C_k \geq 1$ with the following property. Suppose that $V,V' \subset \Bc$ have dimension $k$ and $A : V \to V'$ is invertible. Let $V = E \oplus F, V' = E' \oplus F'$ be splittings for which $A E = E', A F = F'$. Then,
$$
\frac{\a(E', F')^q}{C_k } \leq \frac{\det(A | V)}{\det(A | E) \det(A | F)} \leq \frac{C_k}{ \a(E, F)^q}
$$
where $q = \dim E$.
\end{lem}

\begin{proof}
We let $(\cdot, \cdot)_V$ and $(\cdot, \cdot)_{V'}$ be as above, and let
$\hat \det(A| V)$ denote the determinant with respect to these inner products.
Let us take for granted the (standard) result for inner product spaces which says that
\begin{align} \label{eq:inProdDetSplit}
\frac{1}{\|\pi_{E' \ds F'}\|_V^q} \leq \frac{\hat \det(A | V)}{\hat \det(A | E) \hat \det (A | F)} \leq \|\pi_{E \ds F}\|_{V'}^q 
\end{align}
where $q = \dim E$. As noted earlier, $\det(A|V) = \hat \det(A|V)$. As for $\det(A | E)$,
though $(\hat m_V)|E$ is not necessarily equal to $m_E$, they differ by a multiplicative
constant depending only on $k$, so $\det(A | E)$ and  $\hat \det(A | E)$ differ in the
same way, as do $\det(A | F)$ and  $\hat \det(A | F)$. Finally, as
$$
\frac{1}{k} |\pi_{E \ds F}| \leq \|\pi_{E \ds F}\|_V \leq k |\pi_{E \ds F}| 
$$
and similarly for $\|\pi_{E' \ds F'}\|_{V'}$, the proof is complete upon 
relating $|\pi_{E \ds F}|$ to $\a(E,F)$.
\end{proof}

%%%%%%%%%%%%%%%%%%%%%%%%%%%%%%%%%
\subsection{Regularity of induced volumes and determinants}
%%%%%%%%%%%%%%%%%%%%%%%%%%%%%%%%%

We motivate the results in this subsection as follows: Let $W$ and $W'$ be embedded
$k$-dimensional submanifolds, and consider a $C^1$ map $f: \Bc \to \Bc$ that maps $W$ diffeomorphically onto $W'$. For each $x \in W$, let $T_x W$ denote the tangent space to $W$
at $x$. From Sect. 2.1, we have an induced volume $m_{T_x W}$ on each $T_x W$. 
Under very mild regularity assumptions on $m_E$, these volumes on the tangent 
spaces of $W$ induce a $\sigma$-finite Borel 
measure $\nu_W$ on $W$. Analogous definitions hold for $W'$. It follows from 
Sect. 2.2 that if 
$f_*(\nu_W)$ is the pushforward of the measure $\nu_W$ by $f$,
then at each $y \in W'$, 
$$
\frac{d f_*(\nu_W)}{d \nu_{W'}} (y) \ = \ \frac{1}{\det(df_{f^{-1}y}|T_{f^{-1} y} W)} \ .
$$
For reasons to become clear in the pages to follow, it is important to control these densities. 
This translates
into regularity properties of the function $x \mapsto  \det(df_x| T_x W)$. We will tackle
these questions below in a slightly more general context in preparation for
the distortion estimates in Sect. 5.3.

%%%%%%%%%%%%%%%%%%%%%%%%%%%%%%%%%
\subsubsection{Regularity of induced volumes}
%%%%%%%%%%%%%%%%%%%%%%%%%%%%%%%%%

For linearly independent vectors $v_1, \cdots, v_k \in \Bc$, 
let $\langle v_1, \dots, v_k \rangle$
denote the subspace spanned by $\{v_1, \dots, v_k\}$, and recall that
$P[v_1, \cdots, v_k]$ denotes the parallelepiped defined by the vectors $\{v_i\}$. To simplify notation,
we write $m_{\langle \{v_i\} \rangle} = m_{\langle v_1, \dots, v_k \rangle}$.
We remark from the outset that
there is no reason, in general, to expect the dependence of $m_{\langle \{v_i\} \rangle} P[v_1, \cdots, v_k]$ on $v_1, \cdots, v_k$ to be any better than Lipschitz, as $v \mapsto |v|$ is not differentiable in general Banach spaces. An instance of this already occurs 
for $\R^2 = \{(x,y) : x, y \in \R\}$ endowed with the norm
$|(x,y)|= \max\{|x|, |y|\}$, for which the breakdown of differentiability occurs along 
the diagonal lines $\{(x, x)\}$ and $\{(x, -x)\}$.

To control how far a basis $\{v_i\}$ deviates from `orthogonality', we introduce 
the quantity
$$
N[v_1, \cdots, v_k] = \sum_{i = 1}^k |\pi_{\langle v_i \rangle \ds \langle v_j : j \neq i \rangle} | \ ,
$$
where $\pi_{\langle v_i \rangle \ds \langle v_j : j \neq i \rangle} : \langle v_1, \dots, v_k\rangle
\to \langle v_i \rangle$ is the projection operator defined earlier.

\begin{prop} \label{prop:measReg}
For any $k \geq 1$ and $\oN > k$, there exist $L = L(\oN,k) > 0$ and $\d = \d(\oN, k) \geq 0$ such that the following holds. If $\{v_i\}, \{w_i\}$ are two sets of $k$ linearly independent unit vectors in $\Bc$ for which $\max_{i \leq k} |v_i - w_i| \leq \d$ and $N[v_1, \cdots, v_k], N[w_1, \cdots, w_k]  \leq \oN$, then
$$
\left| \log  \frac{m_{\langle \{v_i\} \rangle} P[v_1, \cdots, v_k] }{ m_{\langle \{w_i\} \rangle}  P[w_1, \cdots, w_k]}  \right| \leq L \sum_{i = 1}^k |v_i - w_i|\ .
$$
\end{prop}

First we prove the following lemma:

\begin{lem}\label{lem:diffSubsp}
For any $k \geq 1$, there exist $\d_1 > 0$ and $L_1 > 0$ (depending only on $k$) such that 
the following hold. Assume 

(i) $E, E' \subset \Bc$ are two $k$-dimensional subspaces with $d_H(E, E') \leq \d_1$, 

(ii) $F$ is a complement to $E$ with  $|\pi_{E \ds F}| \leq \sqrt{k}$
(exists by Lemma \ref{lem:compExist}), and

(iii) $\{v_i\}_{i = 1}^k$ is a basis of unit vectors of $E$.

\noindent
Assume $\d_1$ is small enough that  $\Bc = E' \oplus F$. Let
$v_i' := (\pi_{E' \ds F} v_i) / |\pi_{E' \ds F} v_i|$. Then
\begin{enumerate}
\item $N[v_1', \cdots, v_k'] \leq 2 k N[v_1, \cdots, v_k]$;
\item $$ \left| \log  \frac{m_E P[v_1, \cdots, v_k]}{m_{E'} P[v_1', \cdots, v_k'] } \right| \leq L_1 d_H(E, E')\ .$$
\end{enumerate}
\end{lem}

\begin{proof} Assuming $\d_1 \le \frac{1}{2 \sqrt{k}}$, Lemma \ref{lem:graNormEst} guarantees that $\Bc = E' \oplus F$ with $|\pi_{E' \ds F}| \leq 2 \sqrt{k}$. 
Let $v_i' \in E'$ be as in the statement. Then
\begin{equation} \label{reg-vol}
\frac{m_E  P[v_1, \cdots, v_k]}{m_{E'} P[v_1', \cdots, v_k']} \ = \ \frac{m_E P[v_1, \cdots, v_k]}{m_{E'} \big( \pi_{E' \ds F} P[v_1, \cdots v_k] \big) } \cdot \frac{m_{E'} \big(\pi_{E' \ds F} P[v_1, \cdots, v_k] \big) }{m_{E'} P[v_1', \cdots, v_k']}\ .
\end{equation}
By Lemma \ref{lem:indMeasProps}, the first quotient on the right side is equal to
$$
\frac{m_E B_E}{m_{E'} (\pi_{E' \ds F} B_E)}\ .
$$
In light of Lemmas \ref{lem:indMeasProps} and \ref{lem:graNormEst}, we have that
$$
(1 - 4 \sqrt{k} \ d_H(E,E') )^k \cdot m_{E'}B_{E'} \le m_{E'}(\pi_{E' \ds F} B_E) 
\le (1 + 4 \sqrt{k} \ d_H(E, E'))^k \cdot m_{E'} B_{E'}\ ,
$$
So long as $\d_1 \leq \frac{1}{8 \sqrt{k}}$, and recalling that $|\log(1 +z)| \leq 2 |z|$ for $z \in [-1/2,1/2]$, it follows that
$$
\left| \log \frac{m_E P[v_1, \cdots, v_k]}{m_{E'} \big( \pi_{E' \ds F} P[v_1, \cdots v_k] \big) } \right| \leq 8 k \sqrt{k}\ d_H(E, E')\ .
$$
For the second quotient on the right side of (\ref{reg-vol}), observe that for each $i$, $\pi_{E' \ds F} v_i = |\pi_{E' \ds F} v_i| \ v'_i$. The same reasoning as in Sect. \ref{sec:johnThm} then gives 
\begin{align*}
m_{E'} P[\pi_{E' \ds F} v_1, \cdots, \pi_{E' \ds F} v_k] = \left( \prod_{i = 1}^m |\pi_{E' \ds F} v_i| \right) m_{E'} P[v_1', \cdots, v_k']\ .
\end{align*}
Using Lemma \ref{lem:graNormEst} again to estimate the quantity in parenthesis, we obtain in a similar fashion that
$$
\left| \log  \frac{m_{E'} \big( \pi_{E' \ds F} P[v_1, \cdots, v_k] \big) }{m_{E'} P[v_1', \cdots v_k'] } \right| \leq 8 k \sqrt{k} \ d_H(E, E')\ .
$$
So, Item 2 in Lemma \ref{lem:diffSubsp} holds with $L_1 = 16 k \sqrt{k}$ and any $\d_1 \leq \frac{1}{8 \sqrt{k}}$.

For Item 1 in the lemma, observe that when $\pi_i$ is the projection onto $v_i$ parallel to the rest of the basis and $\pi_i'$ is the analogous for $\{v_i'\}$, we have 
$$
\pi'_i = \pi_{E' \ds F} \circ \pi_i \circ \pi_{E \ds F}| E'\ ,
$$
so that
$$
|\pi'_i| \le  |\pi_{E' \ds F}| \cdot |\pi_i| \cdot | \pi_{E \ds F}| \le 2 \sqrt{k} \cdot |\pi_i| \cdot \sqrt{k}\ ,
$$
giving the desired bound. 
\end{proof}

\begin{proof}[Proof of Proposition \ref{prop:measReg}]

Let $\{v_i\}, \{w_i\}$ be as in the statement, and denote $\langle \{v_i\} \rangle = E, \langle \{w_i\} \rangle = E'$. We will estimate the quantity in question by
\begin{equation}\label{2quotients}
\frac{m_E P[v_1, \cdots, v_k] }{ m_{E'}  P[w_1, \cdots, w_k]}
= \frac{m_E P[v_1, \cdots, v_k] }{m_{E'} P[v'_1, \cdots, v'_k] } \cdot
\frac{m_{E'} P[v'_1, \cdots, v'_k] } { m_{E'}  P[w_1, \cdots, w_k]}\ ,
\end{equation}
where $\{v'_i\}$ is as in Lemma \ref{lem:diffSubsp}. To apply Lemma \ref{lem:diffSubsp} to
 the first quotient on the right side, we must show 
 $d_H(E, E') \leq \text{const} \cdot \sum_i |v_i-w_i|$: 
For $v \in E$ with $ |v| = 1$, we write $v= \sum_i a_i v_i$, and let $w=\sum_i a_i w_i$. Then
\begin{eqnarray*}
d(v, E') \ \le \ |v-w| \ \le \ \sum_i |a_i| |v_i-w_i| & \le & \left(\sum_i |a_i| \right) \cdot \max_i |v_i-w_i|\\
& \le & N[v_1,\dots, v_k] \cdot \max_i |v_i-w_i|\ .
\end{eqnarray*}
Clearly, the role of $E$ and $E'$ can be interchanged in the above. Recalling that
$d_H \le 2 \d_a$ where $\d_a$ is as in Sect. 2.1.2, we have
\begin{eqnarray*} 
d_H(E, E') & \le & 2 \max \{N[v_1,\dots, v_k] , N[w_1, \dots, w_k] \} \cdot \max_i |v_i-w_i|\\
& \le & 2 \oN \max_i |v_i-w_i| \ .
\end{eqnarray*}
So, as long as $2  \oN \max_i |v_i - w_i| \le \d_1$, where $\d_1$ is as in 
Lemma \ref{lem:diffSubsp}, this lemma gives 
$$
\left| \log \frac{m_{E'}  P[v_1', \cdots, v_k']}{m_{E} P[v_1, \cdots, v_k]}  \right| \leq  L_1 d_H(E, E') \leq 2  \oN L_1 \sum_{i = 1}^k |w_i - v_i|\ .
$$

As for the second quotient on the right side of (\ref{2quotients}),
 since all vectors lie in $E'$, it is easy to see, by putting the
 inner product $(\cdot, \cdot)_{E'}$ on $E'$ and using the regularity of $\log \circ \det$ on $E'$,  
that there is a constant $L'_1$ (depending on $\oN$) such that
$$
\left| \log \frac{m_{E'} P[v'_1, \cdots, v'_k]}{m_{E'} P[w_1, \cdots, w_k]} \right| 
\leq L'_1 \sum_{i = 1}^k \|v'_i - w_i \|_{E'}\ .
$$
We need to bound $\|v'_i - w_i \|_{E'}$ by a quantity involving
$\sum_i |v_i-w_i|$. Now $\|v'_i - w_i \|_{E'} \le \sqrt{k} | v_i'-w_i|$ and
$| v_i'-w_i| \le |v_i' - v_i| + |v_i - w_i|$. It remains to observe that
\begin{eqnarray*}
|v_i - v'_i| & \le & |v_i - \pi_{E' \ds F} v_i | + |\pi_{E' \ds F} v_i - v_i'|\\
& = & |\pi_{F \ds E'} v_i| + ||\pi_{E' \ds F} v_i|-1|\\
& \le & 8\sqrt{k} \ d_H(E, E') \qquad \mbox{ by Lemma \ref{lem:graNormEst}}\ .
\end{eqnarray*}
This together with the bound on $d_H(E, E')$ above completes the proof.
\end{proof}

%%%%%%%%%%%%%%%%%%%%%%%%%%%%%%%%%
\subsubsection{Regularity of the determinant}
%%%%%%%%%%%%%%%%%%%%%%%%%%%%%%%%%

The following is the main result of this section.

\begin{prop}\label{prop:detReg}
For any $k \geq 1$ and any $M > 1$ there exist $L_2, \d_2 > 0$ with the following properties. If $A_1, A_2 : \Bc \to \Bc$ are bounded linear operators and $E_1, E_2 \subset \Bc$ are $k$-dimensional subspaces for which
\begin{gather*}
|A_j|,~ |(A_j|_{E_j})^{-1}| \leq M \quad j=1,  2 \, ,\\
|A_1 - A_2|, ~d_H(E_1, E_2) \leq \d_2 \, ,
\end{gather*}
then we have the estimate
\begin{equation} \label{regularity}
\left| \log \frac{\det(A_1 | E_1)}{\det (A_2 | E_2)} \right| \leq L_2 (|A_1 - A_2| + d_H(E_1, E_2))\ .
\end{equation}
\end{prop}

\begin{proof} Putting the inner products from Corollary \ref{lem:orthoBasis} on $E_1$ and $A_1E_1$, we let
$\{v_i\}, \{w_i\}$ be bases for $E_1$ and $ A_1 E_1$ respectively consisting of orthogonal
vectors corresponding to a singular value decomposition of $A_1|_{E_1}$, normalized so as to 
have $|v_i|=|w_i|=1$, and ordered so that $w_i = |A_1 v_i|^{-1} A_1 v_i$.
Taking $\d_2 \leq \d_1$ with $\d_1$ as in Lemma \ref{lem:diffSubsp} and fixing a complement $F$ to $E_1$ with $|\pi_{E_1 \ds F}| \leq \sqrt k$, we define
$$v_i' = \frac{\pi_{E_2 \ds F} v_i }{ |\pi_{E_2 \ds F} v_i|} \ , \qquad \mbox{ and }
\qquad w_i' = \frac{A_2 v_i' }{ |A_2 v_i'|}\ .
$$
First we argue that (with respect to the norms $|\cdot|$) all four of the $N[\cdots]$ quantities so defined are bounded by 
some $\oN$ 
depending only on $k$ and $M$: clearly, $N[v_1, \cdots, v_k], N[w_1, \cdots, w_k] \leq k^2$, 
and $N[v_1', \cdots, v_k'] \leq 2 k N[v_1, \cdots, v_k]$ by Lemma \ref{lem:diffSubsp}. 
To bound $N[w_1', \cdots, w_k']$, write $\pi_i'$ for the parallel projection onto $v_i'$ and $\sigma_i'$ the parallel projection onto $w_i'$,
and observe that $\sigma_i' \circ A_2 = A_2 \circ \pi_i'$,
which yields the bound $|\sigma_i'| \leq |A_2| \cdot |(A_2|_{E_2})^{-1}| \cdot |\pi_i'|$,  
so that $N[w_1', \cdots, w_k'] \leq M^2 N[v_1', \cdots, v_k']$. This bounds all four $N[\cdots]$ quantities by $\oN = 2 k^3 M^2$.

We will estimate the left side of (\ref{regularity}) as follows:
\begin{align*}
\frac{\det (A_1 | E_1) }{\det(A_2 | E_2)} &= \frac{m_{A_1 E_1} P[A_1 v_1, \cdots, A_1 v_k]}{m_{E_1} P[v_1, \cdots, v_k]} \cdot 
 \left( \frac{m_{A_2 E_2} P[A_2 v_1', \cdots, A_2 v_k']}{m_{E_2} P[v_1', \cdots, v_k']} \right)^{-1} \\
 & = \left( \prod_{i =1}^k \frac{|A_1 v_i|}{|A_2 v_i' |}\right) \cdot \frac{m_{A_1 E_1} P[w_1, \cdots, w_k]}{m_{A_2 E_2} P[w_1', \cdots, w_k']} \cdot 
 \frac{m_{E_2} P[v_1', \cdots, v_k']}{m_{E_1} P[v_1, \cdots, v_k]}
\end{align*}
where the extraction of the parenthetical term is as discussed in Section \ref{sec:johnThm}.

We estimate the three factors above separately. For the first, a simple computation gives
$$
\left|  \frac{|A_1 v_i|}{|A_2 v_i'|} -1 \right| = \left| \frac{|A_2 v_i'|-|A_1 v_i|}{|A_2 v_i'|}\right|
\le M |A_1 v_i - A_2 v_i'| \leq M |A_1 - A_2| + M^2 |v_i - v_i'|\ .
$$

For the second and third terms, we will show that for $\d_2$ small enough, $\max_i |v_i - v_i'|, |w_i - w_i'| \leq \d(\oN, k)$ with $\d(\oN, k)$ as in Proposition \ref{prop:measReg}, so that we obtain
$$
\bigg| \log \frac{m_{E_2} P[v_1', \cdots, v_k']}{m_{E_1} P[v_1, \cdots, v_k]} \bigg|  \ \le \ 
L \sum_{i = 1}^k |v_i - v_i'|\ , \qquad
\bigg| \log \frac{m_{A_1 E_1} P[w_1, \cdots, w_k]}{m_{A_2 E_2} P[w_1', \cdots, w_k']} \bigg|  \ \le \ 
L \sum_{i = 1}^k |w_i - w_i'|\ ,
$$
where $L = L(\oN, k)$. 

It now remains to control $|v_i - v_i'|$ and $|w_i - w_i'|$ in terms of $|A_1 - A_2|$ and 
$d_H(E_1, E_2)$. A bound on $|v_i - v_i'| $ was given in the proof of 
Proposition \ref{prop:measReg}, and it is straightforward to estimate
\begin{align*}
|w_i - w_i'| & \leq \frac{1}{|A_1 v_i|} \bigg( |A_1 v_i - A_2 v_i'|  + \big| |A_1 v_i| - |A_2 v_i'| \big| \bigg) \\
&  \le 2 M  |A_1 v_i - A_2 v_i'| \leq 2 M |A_1 - A_2| + 2 M^2 |v_i - v_i'| \ .
\end{align*}
The proof is complete.
\end{proof}

\begin{rmk}\label{rmk:constantsDependence}
Later, when we apply Proposition \ref{prop:detReg} to distortion estimates, we will need to use the dependence of the constants $\d_2, L_2$ on the parameters $k, M$. Keeping track of the constraints on the constants $\d_2, L_2$ made throughout Section 2, one can show that there exists a constant $C_k \geq 1$, depending only on the dimension $k \in \N$, such that we may take $\d_2 = (C_k M^{10 k })^{-1}$ and $L_2 = C_k M^{10 k }$ in the conclusion to Proposition \ref{prop:detReg}.
\end{rmk}

As we have shown, in spite of the lack of differentiability present in finite-dimensional and Hilbert spaces, the notion of determinant we have introduced in this section is at least locally Lipschitz in the sense of Proposition \ref{prop:detReg}. This regularity is used in a crucial way in Section \ref{sec:distEst}, when we apply Proposition \ref{prop:detReg} to distortion estimates.

%%%%%%%%%%%%%%%%%%%%%%%%%%%%%%%%%%%%%%%%%%%%%%%%%%%%%%%%%%%%%%%%%%%%%
%%%%%%%%%%%%%%%%%%%%%%%%%%%%%%%%%%%%%%%%%%%%%%%%%%%%%%%%%%%%%%%%%%%%%
%%%%%%%%%%%%%%%%%%%%%%%%%%%%%%%%%%%%%%%%%%%%%%%%%%%%%%%%%%%%%%%%%%%%%
\section{Addendum to the Multiplicative Ergodic Theorem}
%%%%%%%%%%%%%%%%%%%%%%%%%%%%%%%%%%%%%%%%%%%%%%%%%%%%%%%%%%%%%%%%%%%%%
%%%%%%%%%%%%%%%%%%%%%%%%%%%%%%%%%%%%%%%%%%%%%%%%%%%%%%%%%%%%%%%%%%%%%
%%%%%%%%%%%%%%%%%%%%%%%%%%%%%%%%%%%%%%%%%%%%%%%%%%%%%%%%%%%%%%%%%%%%%

The Multiplicative Ergodic Theorem (MET)
has by now been proved a number of times. Limiting our discussion to infinite dimensions,
it was proved in \cite{Rmet} for Hilbert space cocycles, and in 
\cite{Mmet}, \cite{thieu}, \cite{lianlu} for Banach space cocycles; see also \cite{QFmet}, \cite{QGTmet}, \cite{quasConcise} and \cite{blumenthal}. 
In Sect. 3.1, we recall a version of the MET that is adequate for our purposes, and in Sect. 3.2, we add some 
interpretation in terms of volume growth, following up on the ideas in the previous section. In Sect. 3.3, we discuss continuity properties of certain subspaces.

When proving Theorems 1 and 2, standard techniques will allow us to reduce to 
working only with ergodic measures, so to keep the exposition simple we will state and work with the MET assuming that the underlying dynamical system is ergodic.

%%%%%%%%%%%%%%%%%%%%%%%%%%%%%%%%%%%%%%%%%%%%%%%%%%%%%%%%%%%%%%%%%%
\subsection{A version of the MET for Banach space cocyles}
%%%%%%%%%%%%%%%%%%%%%%%%%%%%%%%%%%%%%%%%%%%%%%%%%%%%%%%%%%%%%%%%%%

We recall below a precise statement of the MET on Banach spaces following Thieullen \cite{thieu}, in a (slightly simplified) setting that is adequate for our purposes.

\bigskip \noindent
{\it Standing hypotheses and notation for Section 3:} 
Let $X$ be a compact metric space, and let $f : X \to X$ be a homeomorphism
preserving an ergodic Borel probability measure $\mu$ on $X$. 
We consider a continuous map $T : X \to {\bf B}(\Bc)$ where
${\bf B}(\Bc)$ denotes the space of bounded linear operators on $\Bc$,
the topology on ${\bf B}(\Bc)$ being the   operator norm topology.
We will sometimes refer to the triple $(f,\mu;T)$ as a {\it cocycle}, and write 
$T_x^n = T_{f^{n-1} x} \circ \cdots \circ T_x$.

\begin{defn}
Let $C \subset \Bc$ be any bounded set. The {\it Kuratowski measure of 
noncompactness} of $C$ is defined by
$$
\a(C) = \sup\{r > 0 : \text{ there is a finite cover of } C \text{ by balls of radius }r \}\ .
$$
\end{defn}
For $A \in {\bf B}(\Bc)$, we denote $|A|_{\a} = \a(A(B))$, where $B$ is the closed ball of radius $1$ in $\Bc$. The assignment $|\cdot|_{\a}$ is a submultiplicative seminorm for which $|A|_\a \leq |A|$ for any $A \in {\bf B}(\Bc)$ (in particular, $A \mapsto |A|_\a$ is continuous as a map on ${\bf B}(\Bc)$ with the operator norm). This and other properties of $|\cdot|_{\a}$ can be found in \cite{nuss}. Since 
$x \mapsto |T^n_x|_{\a}$ is continuous for any $n \geq 1$, it follows from subadditivity that the limit
$$
l_{\a} = \lim_{n \to \infty} \frac{1}{n} \log |T^n_x|_{\a} \ \ge \ -\infty
$$
exists and is constant $\mu$-almost surely; moreover, it  coincides $\mu$-a.s. 
with $\inf_{n \geq 1} \frac 1 n \log |T^n_x|_\a$.

\begin{rmk}\label{rmk:impliesH3}
For $c > 0$, the condition that $l_{\a}< \log c$ $\mu$-a.e. is implied by the following:
Let $\mathcal L_c(\Bc) = \{A \in {\bf B}(\Bc) : A = C + K, \text{ where } K \text{ is compact and } |C| < c\}$. If $T_x \in \mathcal L_c(\Bc)$ for all $x \in X$, then $\sup_{x \in X} |T_x|_{\a} < c$, hence $l_\a < \log c$ $\mu$-almost surely by the continuity of $x \mapsto |T_x|_\a$ and the compactness of $X$.
\end{rmk}

\begin{thm}[Multiplicative ergodic theorem \cite{thieu}]  \label{thm:MET} In addition
to the Standing Hypotheses above we assume that $T_x$ is injective for every $x \in X$.
Then, for any $\l_{\a} > l_{\a}$, there is a measurable, $f$-invariant set 
$\Gamma \subset X$ 
with $\mu(\Gamma)=1$ and at most finitely many real numbers 
$$\l_1 > \l_2 > \cdots > \l_r $$ 
with $\l_r > \l_{\a}$ for which the following properties hold. 
For any $x \in \Gamma$, there is a splitting 
$$\Bc = E_1(x) \oplus E_2(x) \oplus \cdots \oplus E_r(x) \oplus F(x)$$
such that
\begin{itemize}
\item[(a)] for each $i=1,2,\dots, r$, $\dim E_i(x) = m_i$ is finite and constant 
$\mu$-a.s., $T_x E_i(x) = E_i(f x)$, and for any $v \in E_i(x) \setminus \{0\}$, we have
$$
\l_i = \lim_{n \to \infty} \frac{1}{n} \log |T^n_x v| = - \lim_{n \to \infty} \frac{1}{n} \log  |( T^n_{f^{-n} x})^{-1} v |\ ;
$$
\item[(b)] the distribution $F$ is closed and finite-codimensional, satisfies $T_x F(x) \subset F(f x)$ and
$$
\la_\a \geq \limsup_{n \to \infty} \frac{1}{n} \log |T^n_x |_{F(x)}|  \ ;
$$
\item[(c)] the mappings $x \mapsto E_i(x), x \mapsto F(x)$ are $\mu$-continuous (see Definition \ref{defn:muCont} below), and
\item[(d)] writing $\pi_i(x)$ for the projection of $\Bc$ onto $E_i(x)$ via the splitting at $x$, 
we have
$$
\lim_{n \to \pm \infty} \frac{1}{n} \log |\pi_i(f^n x)| = 0 \quad a. s.
$$
\end{itemize}
\end{thm}

\begin{defn}\label{defn:muCont}
Let $(X, \mu)$ be as in the beginning of Section 3.1, and let $Z$ be a metric space. We say that a map $\Phi : X \to Z$ is $\mu$-continuous if there is an increasing sequence $ \bar K_n, n \in \N,$ of compact subsets of $X$, satisfying $\mu (\cup_n \bar K_n) = 1$, for which $\Phi|_{ \bar K_n}$ is continuous for each $n \in \N$.
\end{defn}

\begin{rmk}
When $Z$ is separable, Lusin's Theorem implies directly that $\mu$-continuity is equivalent with Borel measurability, i.e., the inverse image of a Borel subset of $Z$ is a Borel subset of $X$ \cite{castaing1977convex}. This equivalence continues to hold for arbitrary metric spaces $Z$ as a consequence of a deep result of Fremlin; see Theorem 4.1 in \cite{kupka}.

%(see, e.g., \S 0 in Chapter III of \cite{castaing1977convex})

%One may worry that this fails to hold when $Z$ is not separable (as we do not assume our Banach space $\Bc$ to be separable).
%
%We do not assume that our Banach space $\Bc$ is separable, and so must regularly deal with mappings $\Phi$ as in Definition \ref{defn:muCont} for which the range $Z$ is nonseparable. 
%
%A complication arises from the fact that the standard `inverse image' definition of measurability no longer implies that the mapping is the almost-sure limit of simple functions. One resolution of this is to adopt the latter notion (known as \emph{strong measurability}) as the definition of measurability; because the domain $(X, \Fc, \mu)$ is a probability space, strong measurability is equivalent $\mu$-continuity by a version of Lusin's theorem. Following \cite{mane} and \cite{thieu}, we adopt $\mu$-continuity as our notion of measurability. 
\end{rmk}

\noindent 
We may assume going forward that there exist Borel $K_n \subset X$, $n \in \N$, such that 
\begin{itemize}
\vspace{-4 pt}
\item[--] $\Gamma = \cup_n K_n$ is an $f$-invariant set with $\mu(\Gamma)=1$, and
\vspace{-4 pt}
\item[--] the mappings $x \mapsto F(x)$ and $x \mapsto E_i(x)$, $1 \leq i \leq r$, are 
continuous on the closure of each $K_n$.
\end{itemize}
To see this, let $\{ \bar K^{(i)}_n\}$ and $\{ \bar K^{(F)}_n\}$ be the compact sets
given by the $\mu$-continuity of $x \mapsto E_i(x)$ and $x \mapsto  F(x)$
respectively, and let $ \bar K_n = \cap_{i=1}^r  \bar K^{(i)}_n \cap \bar K^{(F)}_n$. 
It is easy to check that $\mu(\cup_n  \bar K_n)=1$. Trimming 
away sets of measure $0$, we obtain an invariant set as claimed.

\begin{lem}\label{lem:measurable} The $\mu$-continuity of
$x \mapsto E(x)$ for $E=E_i$, any $i$, or $E = F$, implies that the following functions are
Borel measurable: 
\begin{itemize}
\vspace{-4 pt}
\item[(i)] $x \mapsto |T_x|_{E(x)}|$, 
\vspace{-4 pt}
\item[(ii)] $x \mapsto m(T_x|_{E(x)})$ where $m(A|_V) = 
\min\{ |A v| : v \in V, |v| = 1\}$ is the minimum norm,
\vspace{-4 pt}
\item[(iii)] $x \mapsto \det(T_x | E(x))$ for $E=E_i$.
\end{itemize}
\end{lem}

\begin{proof} 
Items (i) and (ii) follow from the $\mu$-continuity of $x \mapsto E(x)$ together with the
continuity of $(A, V) \mapsto |A|_V|$ and $(A, V) \mapsto m(A|_V)$ as maps on
${\bf B}(\Bc) \times \Gc(\Bc)$. Now assume $E = E_i$ with dim$(E_i)=m$, and let $\Gc_m(\Bc)$ denote the Grassmannian of $m$-dimensional subspaces. Item (iii) follows from the $\mu$-continuity 
of $x \mapsto E(x)$ together with the continuity of $(A, V) \mapsto \det(A |V)$ viewed
as a map on ${\bf B}_{inj}(\Bc) \times \Gc_m(\Bc)$, ${\bf B}_{inj}(\Bc) \subset {\bf B}(\Bc)$ being the subset of injective linear operators; see Proposition \ref{prop:detReg}.
\end{proof}

%%%%%%%%%%%%%%%%%%%%%%%%%%%%%%%%%
\subsection{Interpretation as volume growth and corollaries}

We now verify for the notion of volume introduced in Section 2 that Lyapunov 
exponents are infinitesimal volume growth rates. 
The setting and notation are as in Theorem \ref{thm:MET}. 

\begin{prop} \label{prop:volGrowthMET} For any collection of
indices
%
% 
%Apply Theorem \ref{thm:MET} to the cocycle $T$ for any $\l > l_{\a}$, and assume that $r > 0$. 
%
%Then, for any fixed $1 \leq i \leq r$, $x \mapsto \det (T_x | E_i(x))$ is Borel measurable, and for almost every $x \in \Gamma_T$,
%$$
%\lim_{n \to \infty} \frac{1}{n} \log \det(T_x^n | E_i(x)) = m_i \l_i
%$$
%For any indices 
$1 \leq i_1 < i_2 < \cdots < i_k \leq r$ the map $x \mapsto \det(T_x | \oplus_{l = 1}^k E_{i_l}(x))$ is measurable, and for $\mu$-a.e. $x \in \Gamma$,
$$
\lim_{n \to \infty} \frac{1}{n} \log \det \big(T^n_x \big| \bigoplus_{l = 1}^k E_{i_l}(x) \big) = \sum_{l = 1}^k m_{i_l} \l_{i_l}\ .
$$
\end{prop}

\begin{proof} Let us first prove the result for $k=1$, writing $\l$, $E$, $m$ instead
of $\l_{i_1}$, $E_{i_1}$, $m_i$.

Define $\phi(x) = \log \det (df_x | E(x))$ for $x \in \Gamma$. By Lemma \ref{lem:measurable}, $x \mapsto \phi(x)$ is Borel measurable, and since $\phi$ is bounded from above (this follows from Lemma \ref{lem:johnSVD} and that $\sup_{x \in X} |T_x| < \infty$), the Birkhoff Ergodic Theorem tells us that there is a constant $\gamma \in \R \cup \{- \infty\}$ such that for $\mu$-almost every $x \in \Gamma$, 
$$
\gamma = \lim_{n \to \infty} \frac{1}{n} \sum_{i =0}^{n-1} \phi(f^ix) 
= \lim_{n \to \infty} \frac{1}{n} \sum_{i =1}^{n} \phi(f^{-i}x)\ .
$$
It suffices to show that $\gamma = m \l$: For each $x$, 
choosing a basis $\{v_1, \cdots, v_m\}$ for $E(x)$ orthonormal with respect to 
$(\cdot, \cdot)_{E(x)}$, we have $\det(T^n_x | E(x)) \leq m^{m/2} \prod_{i = 1}^m  |T_x^n v_i|$ (see Lemma \ref{lem:johnSVD}, Item 1). The growth rates of $|T_x^n v_i|$ are given by the MET, proving $\gamma \leq m \l$. Since one cannot
estimate easily a lower bound for $\det(T^n_x | E(x))$ starting from a fixed set of vectors
in $E(x)$, we iterate {\it backwards} instead. Fixing $\d > 0$ and unit vectors $\{v_1, \cdots, v_m\} \subset
E$ as above, we obtain for large enough $n$,
$$
\det((T^{n}_{f^{-n} x})^{-1} | E(x)) \leq m^{m/2} \prod_{i = 1}^m |(T_{f^{-n} x}^n)^{-1} v_i| \leq m^{m/2} e^{nm (- \l + \d)}\ ,
$$
which gives the desired lower bound for 
$$
\prod_{i = 1}^n e^{\phi(f^{-i} x)} = \det(T^n_{f^{-n}x} | E(f^{-n} x))\ .
$$ 

Proceeding to the general case, it suffices to give a proof for $k=2$, which contains
the main ideas. Let $1 \leq i_1 < i_2 \leq r$. The bounds in Lemma \ref{lem:detSplit}
together with the result for individual $E_i$ proven above gives
$$
\left( C_{m_{i_1} + m_{i_2}} |\pi_{i_1}(f^n x)|^{m_{i_1}} \right)^{-1}  \leq \frac{\det(T^n_x | E_{i_1}(x) \oplus E_{i_2}(x))}{\det(T^n_x | E_{i_1}(x))\det(T^n_x | E_{i_2}(x)) } \leq C_{m_{i_1} + m_{i_2}} |\pi_{i_1}(x)|^{m_{i_1}}\ .
$$
Here we have used the fact that $|\pi_{i_1}(x) |_{E_{i_1}(x) \oplus E_{i_2}(x)}| \le |\pi_{i_1}(x)|$.
The volume growth formula now follows from the single subspace case and the fact that $$\lim_{n \to \infty} \frac{1}{n} \log |\pi_i(f^n x)|  =0$$ for any $1 \leq i \leq r$.
\end{proof}

Let $1 \leq i_1 < i_2 < \cdots < i_k \leq r$. A technical fact that will be needed is the integrability of 
$\log^- m(T_x|_{\oplus_{l = 1}^k E_{i_l}(x)})$, equivalently the integrability of $\log^+ |(T_x|_{\oplus_{l = 1}^k E_{i_l}(x)})^{-1}|$,
which requires justification as these minimum norms can be arbitrarily small. We deduce it from Proposition \ref{prop:volGrowthMET}.

\begin{cor}\label{cor:integrable} For any collection of indices $1 \leq i_1 < i_2 < \cdots < i_k \leq r$, the function $\psi(x) = \log^+ |(T_x|_{\oplus_{l = 1}^k E_{i_l}(x)})^{-1}| \in L^1(\mu)$.
\end{cor}

\begin{proof} Since for any $v \in \oplus_{l = 1}^k E_{i_l}(x)$ with $|v|=1$, we have
$$
\det(T_x| \oplus_{l = 1}^k E_{i_l}(x) ) \le m^m |T_x|_{\oplus_{l = 1}^k E_{i_l}(x)}|^{m-1} |T_x v| 
$$
where $m=\dim (\oplus_{l = 1}^k E_{i_l})$, it follows that 
$$
|(T_x|_{\oplus_{l = 1}^k E_{i_l}(x)})^{-1}| = \max_{v \in \oplus_{l = 1}^k E_{i_l}(x), |v| = 1} \frac{1}{|T_x v|} \le 
\frac{C}{\det(T_x|\oplus_{l = 1}^k E_{i_l}(x))}
$$
where $C>0$ is a constant depending only on $\sup_{x \in X} |T_x|$ and $m$. Our assertion follows from the fact that
$-\log \det(T_x | \oplus_{l = 1}^k E_{i_l}(x) ) \in L^1(\mu)$, which we have just proved. 
\end{proof}

%%%%%%%%%%%%%%%%%%%%%%%%%%%%%%%%%%%%
\subsection{Continuity of certain distributions on sets with uniform estimates}

Here we discuss the continuity of the distributions $E^i := E_1 \oplus E_2 \oplus \cdots \oplus E_i$ for any $i \leq r$. While the results are analogous to those in finite dimensions, 
some of the often used arguments in finite dimensions, such as compactness
of the Grassmannian of $m$-dimensional subspaces, are not applicable in the present
setting. 

\medskip
Let $i$ be fixed throughout. We let $E^i$ be as above, and let
$\td>0$ be such that $3 \td< |\la_i - \la_{i+1}|$ if $i<r$, and
$3 \td < |\la_i - \la_{\a}|$ if $i=r$. For $L>1$, let 
$$
G^i_L:= \{x \in \Gamma :  |(T^n_{f^{-n}x}|_{E^i(f^{-n}x)})^{-1}| \le L e^{-n(\la_i - \td)} \text{ for all } n \geq 1\}\ .
$$

\begin{lem}\label{lem:contSplitting}
For any $i\le r$ and $L>1$, the map $x \mapsto E^i(x)$ is continuous 
with respect to the $d_H$-metric on $\Gc(\Bc)$ as $x$ varies over $G^i_L$. 
\end{lem}

\begin{proof} Let $E=E^i$, and define $\bar F(x) = \oplus_{j>i} E_j(x) \oplus F(x)$, so that
at each $x \in \Gamma$, we have $\Bc = E(x) \oplus \bar F(x)$. 
Fix $x \in \Gamma$, and let $x^n \in G^i_L$ be such that $x^n \to x$.
Let $\{v^n\} \subset \Bc$ be any sequence of unit vectors such that $v^n \in E(x^n)$ for each $n$, and let $v^n = w^{n, E} + w^{n, \bar F} \in E(x) \oplus \bar F(x)$ be the decomposition with respect to the splitting $E(x) \oplus \bar F(x)$. It suffices 
to show that $w^{n, \bar F} \to 0$ as $n \to \infty$; that this is sufficient for proving 
$E(x^n) \to E(x)$ follows from the fact that $E(x)$ and $E(x^n)$ have the same finite dimension 
(for instance, one could use this to show that $|\pi_{\bar F(x) \ds E(x)}|_{E(x^n)}| \to 0$ as $n \to \infty$, hence 
$\d_a(E(x^n), E(x)) \to 0$ as $n \to \infty$; see Section 2.1.2).
To derive a contradiction, we will assume, after passing to a subsequence, 
that $|w^{n, \bar F}| \geq c$ for some $c > 0$ for all $n$. 

The following notation will be used: $x_{-k}=f^{-k}x$, $x^n_{-k}=f^{-k} x^n$, and 
let $v^n_{-k}$ be the unique vector 
in $E(x^n_{-k})$ such that 
$T^k_{x^n_{-k}} v^n_{-k} = v^n$. We split 
$$v^n_{-k} = \hat w^{n, E}_{-k} + \hat w^{n, \bar F}_{-k} \in E( x_{-k}) 
\oplus \bar F( x_{-k})\ . $$ 

Let $\pi_E(x)$ denote the projection onto $E(x)$ along $\bar F(x)$.
We will show that for every $k>0$ large enough, % (depending only on $x$), 
there exists $n(k)$ such that for all $n \ge n(k)$, 
\begin{equation} \label{unifbound}
|\hat w^{n, \bar F}_{-k}| \le (L+ 2|\pi_E(x)|) \ e^{-k(\la_i -\td)} \qquad  and 
\qquad |T^k_{x_{-k}}(\hat w^{n, \bar F}_{-k})| \ge \frac{c}{2} \ .
\end{equation}
Now since $x \in \Gamma$, we may assume $x_{-k}$ visits infinitely often sets 
on which there are uniform bounds for $|T^m|_{\bar F}|, m=1,2, \dots$, so 
there exists arbitrarily large $k$ for which 
$|T^k_{x_{-k}}|_{\bar F(x_{-k})}| \ll e^{k(\la_i - 2\td)}$. 
That is clearly inconsistent with (\ref{unifbound}).

To prove (\ref{unifbound}), observe first that since $x^n \in G^i_L$, we have
$|v^n_{-k}| \le L e^{-k(\la_i -\td)}$, and notice that this bound is independent of $n$. 
Thus for each fixed $k$, 
$$|T^k_{x_{-k}}v^n_{-k} - v^n| = |(T^k_{x_{-k}}-T^k_{x^n_{-k}})(v^n_{-k})| 
\le |T^k_{x_{-k}}-T^k_{x^n_{-k}}| \cdot |v^n_{-k}| \to 0 \quad \mbox{ as } n \to \infty\ .
$$
In particular, $|T^k_{x_{-k}} \hat w^{n, E}_{-k} - w^{n, E}| \to 0$ and $|T^k_{x_{-k}} \hat w^{n, \bar F}_{-k} - w^{n, \bar F}| \to 0$ as $n \to \infty$, which implies that for all $n \ge$ some $ n(k)$,
$$
|T^k_{x_{-k}} \hat w_{-k}^{n, E}| \le 2 |\pi_E(x) v^n| \le 2 |\pi_E(x)| \qquad \mbox{ and } \qquad
|T^k_{x_{-k}} \hat w_{-k}^{n,\bar F}| \ge \frac{c}{2}\ .
$$
Finally, as $x \in \Gamma$, $|(T^k_{x_{-k}}|_{E(x_{-k})})^{-1}| < e^{-k(\la_i -\td)}$
holds for all large enough $k$. Thus
$$
|\hat w_{-k}^{n, \bar F}| \le |v^n_{-k}| + |\hat w_{-k}^{n, E}| \le (L+ 2|\pi_E(x)|)\ e^{-k(\la_i -\td)}\ ,
$$
completing the proof of (\ref{unifbound}).
\end{proof}

%\begin{rmk}\label{rmk:usedInContSplitting} In Section 6, we will need to extend 
%Lemma \ref{lem:contSplitting} to a situation where the points involved
%are not necessarily in $\As$. In anticipation of that, we point out what makes
% the proof above work: (i) $x_{-k}$ and $x^n_{-k}$ must be defined for all $n$ 
%and $k$; (ii) there has to be a $df$-invariant
%splitting $\Bc = E(x_{-k}) \oplus \tilde F(x_{-k})$ along the backwards orbit of $x$,
%with asymptotic growth rates under $df$ of vectors in $\tilde F$ strictly smaller 
%than those in $E$ (see above); (iii) there are invariant 
%splittings $E(x^n_{-k})$ along the backwards orbit of each $x^n$, with 
%uniform growth rates as in the definition of  $G^i_L$; and (iv) enough
%continuity of $f^{-n}$ to ensure that for each $k$, $x^n_{-k} \to x_{-k}$ as $n \to \infty$.
%
%
%For later use in Section 6, we record the basic ingredients used in the proof of Lemma \ref{lem:contSplitting}: (i) uniform (in $n$) control on $|T^{-k}_{x^n}|_{E(x^n)}|$, (ii) control of the Osceledets splitting $\Bc = E(x_{-k}) \oplus \tilde F(x_{-k})$ along the backwards orbit of $x$ (in particular, control of $|T^k_{x_{-k}}|_{\tilde F(x_{-k})}|$), and (iii) the continuity of $f^{-k}$ for all $k \in \N$. In particular, the `slow-growing' complement $\tilde F$ for orbits of the sequence $\{x^n\}$ was not used in the above proof.
%\end{rmk}

%(ii) the Osceledets splitting $\Bc_{x_{-k}} = E(x_{-k}) \oplus \tilde F(x_{-k})$ along the backwards orbit $\{x_{-k}\}$ of $x$, 

\vskip .2in \noindent
{\it In the rest of this paper, $(f,\mu)$ is assumed to satisfy
Hypotheses (H1)--(H3) in Section 1. For simplicity, we first treat the ergodic case,
assuming $\mu$ is ergodic from here through Sect. 5.3, removing the ergodicity
assumption only in Sect. 5.4.}

%%%%%%%%%%%%%%%%%%%%%%%%%%%%%%%%%%%%%%%
%%%%%%%%%%%%%%%%%%%%%%%%%%%%%%%%%%%%%%%%
\section{Preparation I : Lyapunov metrics}
%%%%%%%%%%%%%%%%%%%%%%%%%%%%%%%%%%%%%%%%
%%%%%%%%%%%%%%%%%%%%%%%%%%%%%%%%%%%%%%%%%

The goal of this section is to introduce new norms $|\cdot |'_x$ in 
the tangent spaces of $x \in \Gamma$ with respect to which expansions and 
contractions are reflected in a single time step. We also introduce a function $l$ that, roughly
speaking, measures
the degree to which $f$ deviates from uniform hyperbolicity.

These techniques have been used in finite dimensions and on separable Hilbert spaces
(we follow more closely \cite{ledyou1, ledyou2} and \cite{lianyou1}; see also the references in \cite{barreira}). There is, however, the following difference: 
In finite dimensions, for instance,
it is customary to fix a model space
$\mathbb R^{\dim E^u} \times \mathbb R^{\dim E^c} \times \mathbb R^{\dim E^s}$
with a Euclidean inner product %$\mathbb R^\tau$, 
and to identify a neighborhood of each $x$ (with its Lyapunov metric) with
a neighborhood of $0$ in the model space. We do not do this here, as there is
no obvious common model space for $E^s_x, x \in \Gamma$. Indeed
there is no standard model space 
for infinite dimensional subspaces of Banach spaces. Instead, we will work directly 
on the tangent spaces $\Bc_x$ of $x$.

For completeness, we will go through the entire construction, providing complete statements of results, but will omit proofs that require no modification.

%%%%%%%%%%%%%%%%%%%%%%%%%%%%%%%%%%%%
\subsection{Adapted norms}

Consider the cocycle $(f, \mu; df)$. By condition (H3), $l_\a <0$.
We fix an arbitrary $\lambda_\a \in (l_\a, 0)$, and apply Theorem \ref{thm:MET}
to obtain Lyapunov exponents $\l_1 > \l_2 > \cdots > \l_r$ and a splitting
of $\Bc_x$, the tangent space at $x$, into $\Bc_x = \oplus_{i=1}^r E_i(x) \oplus F(x)$ 
for every $x\in \Gamma$. 

%We may assume, without loss, that at each $x \in \Gamma$, 
%the conclusions of Corollary \ref{cor:vectUnifMET} hold.

For many purposes, it is sufficient to distinguish between
unstable, center and stable subspaces, defined to be 
$$
E^u_x  =  \bigoplus_{i: \l_i > 0} E_i(x)\ , \qquad 
E^c_x =  \bigoplus_{i: \l_i = 0} E_i(x)\ , \quad \mbox{ and } \quad
E^s_x =  \bigoplus_{i: \l_i < 0} E_i(x) \ \oplus \ F(x)\ .
$$
We will also write $E^{cu}_x= E^u_x \oplus E^c_x$, and use $\pi^u_x, \pi^c_x$ and
$\pi^s_x$
to denote the projections onto $E^u_x, E^c_x$ and $ E^s_x$ respectively according to the splitting $\Bc_x = E^u_x \oplus E^c_x \oplus E^s_x$. We do not require that all
these subspaces be nontrivial; in particular, in our main results, $E^c = \{0\}$.

We now proceed to modify the norms on tangent spaces of individual points,
with the aim of producing new norms with respect to which Lyapunov exponents 
will be reflected in a single time step. Let $\lambda^+ = \min\{\la_i: \la_i>0\}$ and
$\lambda^- = \max\{l_\a, \la_i: \la_i<0\}$.
We define $\l_0 = \min\{\lambda^+, -\lambda^- \}$, fix $\d_0 \ll \l_0$ and let 
$\l=\l_0-2\d_0$. For $n>0$ and $u \in E^{cu}_x$, let us agree to use the 
shorthand $df^{-n}_xu$ to mean $(df^n_{f^{-n}x}|_{E^{cu}_{f^{-n}x}})^{-1}u$.
We introduce for each $x \in \Gamma$ a new norm $|\cdot|'_x$ on $\Bc_x$ as follows:
\begin{align*}
\text{For } \qquad u &\in E^u_x, ~~ |u|'_x  = \sum_{n = 0}^\infty \frac{|df^{-n}_x  u|}{e^{-n \l }} \ , \\
 v& \in E^c_x, ~~ |v|'_x = \sum_{n = - \infty}^{\infty} \frac{|df^n_x v |}{e^{2 |n| \d_0}}\ , \\
w& \in E^s_x, ~~ |w|'_x = \sum_{n = 0}^{\infty} \frac{|df^n_xw |}{e^{- n \l}}\ ,
\end{align*}
and for $p = u +v +  w \in \Bc_x, u \in E^u_x, v \in E^c_x, w \in E^s_x$, we define 
\begin{equation} \label{max}
|p|_x' = \max \{|u|_x', |v|_x', |w|_x'\}\ .
\end{equation}
To estimate how far these new norms deviate from the original ones, we let
\begin{eqnarray*}
C_u(x) & =  & \sup_{n \geq 0} \frac{\sup_{v \in E^u_x, |v| = 1} |df^{-n}_x v|}
{e^{- n(\la_0 - \d_0)}}\ , \\
C_c(x) & =& \sup_{n \in \mathbb Z} \frac{\sup_{v \in E^c_x, |v| = 1} |df^{n}_x v|}
{e^{|n|\d_0}}\ , \\
C_s(x) & =  & 
\sup_{n \geq 0} \frac{\sup_{v \in E^s_x, |v| = 1} |df^n_x v|}{e^{- n(\la_0 - \d_0)}},
\end{eqnarray*}
and let 
$$
C(x) = \max\{C_u(x), C_c(x), C_s(x), |\pi^u_x|, |\pi^c_x|, |\pi^s_x|\}\ .
$$
Observe that all are finite-valued, Borel measurable functions on $\Gamma$ (see Lemma \ref{lem:measurable}).
%; they are all finite by Corollary \ref{cor:vectUnifMET}.

The following lemma summarizes the properties of the adapted norms $|\cdot|_x'$.
The proof is a simple computation and is omitted.

\begin{lem}\label{prop:pointwiseNorm}The following hold for all $x \in \Gamma$:
\begin{enumerate}
\item (One-step hyperbolicity) For any $u \in E^u_x, v \in E^c_x, w \in E^s_x$, we have
\begin{gather*}
|df_x u|_{fx}'  \ge e^{\lambda} |u|'_x \\
e^{- 2 \d_0} |v|_x' \leq |df_x  v|_{f x}' \leq e^{2 \d_0} |v|_x' \\
|df_x w|_{f x}' \leq  e^{ - \l } |w|'_x\ .
\end{gather*}
\item The norms $|\cdot|'_x$ are related to the usual norm $|\cdot|$ by
\begin{align}\label{normcomparison}
\frac13 |p| \le |p|_x' \leq \frac{3}{1 - e^{- \d_0}} C(x)^2 |p|\ .
\end{align}
\end{enumerate}
\end{lem}

\smallskip

Identifying $\Bc_x$ with $x + \Bc$ via the exponential map $\exp_x : \Bc_x \to \Bc$,
i.e., the map that sends $v \in \Bc_x$ to $x + v \in \Bc$, we view $\{(\Bc_x, |\cdot|_x'): x \in \Gamma\}$
as a collection of charts, and define the connecting maps 
\begin{align*} \tf_x : \Bc_x \to \Bc_{fx} \quad \text{ by } \quad \tf_x = \exp_{fx}^{-1} \circ f \circ \exp_x \ .\end{align*}
We will also use the notation $\tf^n_x = \tf_{f^{n-1}x} \circ \cdots \circ \tf_{f^{-1}x}
\circ \tf_x$. Since the derivative at $0$ of $\tf_x$, written as
$(d \tf_x)_0$, is the same as $df_x$, these derivatives exhibit hyperbolicity 
in one timestep with respect to the $|\cdot|'$ norms.

Our next task is to reduce
the sizes of the domains for $\tf_x$ so that on these reduced domains, $\tf_x$ is well 
approximated by the linear map $(d \tf_x)_0$. Since $f$ is assumed to be $C^2$ 
and $\As$ is compact
(see (H1), (H2) in Sect. 1), it is easy to see that there exist $M_0>0$ and $r_0>0$ such that
$|d^2 f_x|<M_0$ for all $x \in \Bc$ with dist$(x, \As)<r_0$.

Below, we use the notation $\tilde{B}_x(r) = \{p \in \Bc_x \mid |p|_x' \leq r\}$.
In all statements regarding the chart maps $\tf_x$, the norms on their domain and 
range spaces should be understood to be $|\cdot|'_x$ and $|\cdot|'_{fx}$ respectively. 
The next lemma is straightforward.

\begin{lem}\label{prop:chartProps}
Define $\tilde{l} : \Gamma \to [1,\infty)$ by
\begin{align}\label{eq:chartsize}
\tilde{l}(x) = \max\bigg\{\frac{27 M_0}{1 - e^{- \d_0}}, 1\bigg\}   \cdot C(f x)^2\ .
\end{align}
Then there exists $\d_1 > 0$ such that for any $\d \le \d_1$, the following holds for 
$\tf_x : \tilde{B}_x(\d \tilde{l}(x)^{-1}) \to \Bc_{fx}$, i.e., for $\tf_x$ restricted to
the domain $\tilde{B}_x(\d \tilde{l}(x)^{-1})$:
\begin{enumerate}
\vspace{-4pt}
\item $\lip( \tf_x - (d \tf_x)_0) \leq \d$;
\vspace{-4pt}
\item the mapping $z \mapsto (d\tilde f_x)_z$ satisfies 
$\lip \big( d \tf_x \big) \leq \tilde{l}(x)$.
\end{enumerate}
\end{lem}

%%%%%%%%%%%%%%%%%%%%%%%%%%%%%%%%%%%%%%%
\subsection{Measuring deviation from uniform hyperbolicity}

The maneuvers in Sect. 4.1 transform the nonuniformly hyperbolic map
$f$ into a family of uniformly hyperbolic local maps $\tilde f_x$, but it is at the expense of coordinate changes that can be unboundedly large as $x$ varies over $\Gamma$.
The sizes of these coordinate changes, which we may think of as measuring how far
$f$ deviates from being uniformly hyperbolic, are incorporated into
the function $\tilde l$ in Lemma \ref{prop:chartProps}, a function that contains two other pieces of
related information: chart sizes, i.e. how quickly $f$ deviates from $df_x$
as we move away from $x$, and the regularity of $df$ as seen in these coordinates.
Informally, the larger $\tilde l(x)$ at a point, the weaker the hyperbolicity at $x$.

The function $\tilde l$ is measurable and usually unbounded on $\Gamma$.
We show next that it is dominated by a function that varies slowly 
along orbits.

\begin{lem}\label{lem:slowvary}
Given any $\d_2 > 0$, there exists a function $l : \Gamma \to [1,\infty)$ (depending
only on $\tilde l$ and $\d_2$) such that for $\mu$-a.e. $x \in \Gamma$,
\begin{align}\label{eq:slowvary3}
\tilde{l}(x) \leq l(x), \text{ and }\; l(f^{\pm} x) \leq e^{\d_2} l(x).
\end{align}
\end{lem}

Once this lemma is proved, we will use $l$ instead of $\tilde l$ in all subsequent
estimates, and Lemma \ref{prop:chartProps} clearly holds for $\tilde f_x$ on
the domain of $\tilde{B}_x(\d {l}(x)^{-1})$ for any $\d \le \d_1$.
An obvious advantage of having slowly varying chart sizes is that it ensures
 that graph transforms of functions from $E^{u}$ to $E^{cs}$ are well
defined (see Sect. 5.1). Another advantage of a slowly varying $l$
is, as we will see, that it ensures that
estimates can deteriorate at most slow exponentially along orbits.

Lemma \ref{lem:slowvary} is well known in finite dimensions. In the present setting, 
there is a subtle difference in the proof caused by the fact that $df_x|_{E^{cu}}$ is not 
assumed to have a uniformly bounded inverse. 
This difference is exemplified by the task of finding a 
slowly-varying enlargement of the function $C_u$ (see Sect. 4.1). So instead of giving a full proof of Lemma \ref{lem:slowvary}, we will limit our discussion to enlarging $C_u$.

\begin{lem}\label{slowvary2}
Given any $\d_2 > 0$, there exists a measurable function $C_u' : \Gamma \to [1,\infty)$ such that for $\mu$-a.e. $x \in \Gamma$,
\begin{align*}
C_u(x) \leq C_u'(x), \text{ and }\; C_u'(f^{\pm} x) \leq e^{\d_2} C_u'(x).
\end{align*}
\end{lem}

We will use the following ergodic theory lemma.

\begin{lem}[Lemmas 8 \& 9 in \cite{walMET}] \label{lem:ergodSlowVary}
Let $(X, \Fc, \mu, f)$ be an invertible measure-preserving transformation (mpt) of a probability space. Let $\phi :  X \to \R$ be measurable and assume that either $(\phi \circ f - \phi)^+$ or $(\phi \circ f - \phi)^-$ is integrable. Then
$$
\lim_{n \to \pm \infty} \frac{1}{|n|} \phi \circ f^n \to 0 \qquad \mbox{ a.s. } 
$$
\end{lem}

\begin{proof}[Proof of Lemma \ref{slowvary2}]
Let $\psi : \As \to \R$ be a function for which $\log \psi$ satisfies the hypotheses 
 of Lemma \ref{lem:ergodSlowVary} for the mpt $(\As, \Bs, \mu, f^{-1})$, where $\Bs$ 
 is the Borel $\sigma$-algebra of subsets of $\As$. Then it will follow 
 that the function
\[
{\psi}'(x) := \sup_{n \in \Z} e^{- |n| \d_2} \psi(f^n x)
\]
is almost-surely finite valued, satisfies $\psi \leq {\psi}'$, and, as one can easily check, ${\psi}'(f^{\pm} x) \leq e^{\d_2} {\psi}'(x)$. So it suffices to check the hypotheses 
of Lemma \ref{lem:ergodSlowVary} for $\psi(x) = \log C_u(x)$.
% (which is measurable; see Remark \ref{rmk:measurable}).
Observe that either $C_u(x) = 1$, i.e., the supremum in the definition of $C_u(x)$
is attained at $n=0$, or that
\begin{eqnarray*}
C_u(x) & \le & |df^{-1}_x|_{E^u_{x}}| \ \cdot \ \sup_{n \ge 1, v \in E^u_{f^{-1}x}, |v|=1} e^{n(\lambda_0 -\d_0)}  |df^{-n+1}_{f^{-1}x} v|\\
& = & e^{\lambda_0 -\d_0} |df^{-1}_x|_{E^u_x}| \cdot C_u(f^{-1}x)\ .
\end{eqnarray*}
Hence 
$$
\log C_u(f^{-1}x) - \log C_u(x) \ge  \min\{ -\lambda_0 + \d_0 - \log |df^{-1}_x|_{E^u_x}|, 0\} \ .
$$
Thus to check that $(\log C_u(f^{-1} x) - \log C_u(x))^- \in L^1(\mu)$, it suffices to check 
$$\log^+ |df^{-1}_x|_{E^u_x}| = \log^+ |(df|_{E^u})^{-1}| \circ f^{-1}(x)
\in L^1(\mu)\ .
$$
Unlike the case of finite dimensional diffeomorphisms, this requires justification, as
$|(df|_{E^u})^{-1}|$ can be unboundedly 
large as $x$ varies over $\Gamma$. We have in fact anticipated this issue, 
and have proved in Corollary \ref{cor:integrable} that
 $\log^+ |(df|_{E^u})^{-1}| \in L^1(\mu)$. 
\end{proof}

\smallskip
The numbers $\lambda_0, \d_0$, (hence $\lambda$) and $\d_1$ introduced in 
Section 4 are fixed once and for all. We now fix $\d_2 > 0$ with $\d_2 \ll \l$,
(e.g. $\d_2 \leq \frac{1}{100 m_u} \la$ ; this will be useful later in Section 5.3),
and let $l$ be given by Lemma \ref{lem:slowvary}.
For $l_0 \ge 1$, let
\[
\Gamma_{l_0} := \{x \in {\Gamma} : l(x) \leq l_0\}\ .
\]
We will refer to sets of the form $\Gamma_{l_0}$ as {\it uniformity sets}. On
such a set, one has uniform expansion and contraction estimates, uniform 
`angles' of separation between $E^u, E^c$ and $E^s$, and uniform
bounds on the extent to which the adapted norms differ from the original norms.
The number $\d \leq \d_1$ can be chosen independently of all the
quantities above, permitting us to shrink our charts as needed.
Once $\d$ is fixed, chart sizes, nonlinearities and second derivatives 
in charts will also be uniformly bounded for
$x \in \Gamma_{l_0}$. Furthermore, 
Lemma \ref{lem:slowvary} tells us that for $x \in \Gamma_{l_0}$,
$f^n x \in \Gamma_{l_0 e^{|n|\d_2}}$, that is to say, the quantities above
have bounds that can deteriorate at most slow exponentially along orbits.

%\medskip
%The numbers $\lambda_0, \d_0$, (hence $\lambda$), $\d_1, \d_2$ and the function
%$l$ introduced in Section 4 are fixed once and for all; in particular, we demand that $\d_2 \leq \frac{1}{100 m_u} \la$. We will assume, by paring off a set of measure zero,
%that $l(f^{\pm} x) \leq e^{\d_2} l(x)$ for pointwise $x \in \Gamma$. For any $\d \le \d_1$,
%we have a system of charts $\{(\Bc(x), |\cdot|_x') , x \in \Gamma\}$ and connecting maps $\{\tilde f_x: \tilde B_x(\d l(x)^{-1}) \mapsto  \Bc(f x) , 
%x \in \Gamma \}$
%with the properties in Lemmas \ref{prop:pointwiseNorm}, \ref{prop:chartProps} and \ref{lem:slowvary}, and we will shrink
%$\d$ as needed in the constructions to follow.

%\begin{rmk}
%In Section 5, the parameters $\d_0, \la_0, \la$ are fixed and will not be changed. We will however diminish the parameters $\d_1$ (describing the size of charts) and $\d$ (the rate of exponential growth of the chart size function $l$); these parameters should be thought of as independent, and we will never enforce a relation between them.
%\end{rmk}

%%%%%%%%%%%%%%%%%%%%%%%%%%%%%%%%%%%%%%%
%%%%%%%%%%%%%%%%%%%%%%%%%%%%%%%%%%%%%%%%
\section{Preparation II : Elements of Hyperbolic Theory}
%%%%%%%%%%%%%%%%%%%%%%%%%%%%%%%%%%%%%%%%
%%%%%%%%%%%%%%%%%%%%%%%%%%%%%%%%%%%%%%%%%

Unstable manifolds and how volumes on them are transformed are 
central ideas in this paper. In Sects. 5.1 and 5.2, we record some
basic facts about continuous families (called ``stacks") of local unstable manifolds, 
in preparation for the definition of SRB measures in Sect. 6.1. In Sect. 5.3, 
we consider densities on unstable manifolds with respect to reference measures
derived from the induced volumes on finite dimensional subspaces 
introduced earlier. We provide a detailed proof
of distortion bounds as these densities are pushed forward,
confirming that the notion of volume proposed is adequate for our needs. 

From Section 3 through Sect. 5.3, we have operated under the assumption
of ergodicity, which has simplified considerably the exposition.
In Sect. 5.4, we discuss how the constructions and results given so far 
can be adapted for nonergodic measures.

%%%%%%%%%%%%%%%%%%%%%%%%%%%%%%%%%%%%%%%
\subsection{Local unstable manifolds}

Continuing to work in the Lyapunov metric,
we employ the following notation: for $\tau = u, c, s$ and $r > 0$, we write $\tilde{B}^{\tau}_x(r) = \{a \in E^{\tau}_{x} : |a|_x' \leq r\}$ and $\tilde{B}^{cs}_x(r) = \tilde{B}^c_x(r) + \tilde{B}^s_x(r)$, so that $\tilde{B}_x(r)$, which was defined earlier, is equal to
$\tilde{B}^{u}_x(r) + \tilde{B}^{cs}_x(r)$.

\begin{thm}[Unstable Manifolds Theorem in charts] \label{thm:unstabMfld}
For all $\d>0$ sufficiently small, there exists a unique family of continuous 
maps $\{g_x: \tilde{B}^u_x(\d l(x)^{-1}) \to \tilde{B}^{cs}_x(\d l(x)^{-1})\}_{x \in \Gamma}$ 
such that 
\[
g_x(0)=0 \qquad \mbox{and} \qquad 
\tilde{f}_x (\graph g_x) \supset \graph g_{fx} \quad \text{ for all } x \in \Gamma \ .\]
With respect to the $|\cdot|'_x$ norms on $\tilde{B}_x(\d l(x)^{-1})$, the family $\{g_x\}_{x \in 
\Gamma}$ has the following additional properties:
\begin{enumerate}
\vspace{-4pt}
\item $g_x$ is $C^{1 + \lip}$ Fr\'echet differentiable, with $(dg_x)_0 = 0$;
\vspace{-4pt}
\item $\lip g_x \leq \frac{1}{10 }$ and $\lip d g_x \leq C l(x)$ where $C > 0$ is 
independent of $x$;
\vspace{-4pt}
\item  if $\tilde{f}_x (u_i + g_x(u_i)) \in \tilde{B}_{fx}(\d l(f x)^{-1})$ for $u_i \in \tilde{B}_x^u(\d l(x)^{-1})$, $i = 1,2$, then 
\begin{align}\label{eq:unstableExpansion}
| \tilde{f}_x (u_1+ g_x(u_1))- \tilde{f}_x (u_2 + g_x(u_2)) |_{fx}' \geq (e^{\la} - \d) | (u_1 +  g_x(u_1))-  (u_2 +  g_x(u_2)) |_x' \ .
\end{align}
\end{enumerate}
\end{thm}

These results are well known for finite-dimensional systems. For Hilbert space maps, stable and unstable manifolds were constructed using Lyapunov charts in \cite{lianyou1};
the methods in that paper can be carried over without any substantive change to our Banach space setting, and we omit the proof. 

We fix $\d'_1 \le \d_1$ small enough that Theorem \ref{thm:unstabMfld} holds with
$\d \le \d'_1$, and write $\tilde W^u_{\d, x} = \graph (g_x)$ where $g_x$ is as above.
It is easy to see that $\tilde W^u_{\d, x} \subset \tilde W^u_{\d', x}$ for $\d < \d'$.
We let $W^u_{\d, x} = \exp_x \tilde W^u_{\d, x}$, and call $W^u_{\d, x}$ a 
{\it local unstable manifold} at $x$. It will be assumed implicitly in all future references
to local unstable manifolds that $\d \le \d'_1$ where $\d'_1$ is as above. Note that
we may shrink $\d$ without harm, and will do so a finite number of times in the proofs to come.
%call $W_{{\rm loc},x}^u := \exp_x (\graph g_x)$ the {\it local unstable manifold} at $x$. 
The {\it global unstable manifold} at $x$, defined to be 
$$W^u_x := \cup_{n \geq 0} f^n (W^u_{\d,x}),
$$ 
is an immersed submanifold in $\Bc$ (by the injectivity of $f$ and $df_x$; 
see (H1) in Section 1). 

Theorem \ref{thm:unstabMfld} is proved using graph transforms, an idea we will need
again later on. We state (without proof) the following known result. 
Let
$$
\Wc(x) = \left\{g : \tilde{B}^u_x(\d l(x)^{-1}) \to  E^{cs}_x \ \mbox{ s.t. }
   \ g(0) = 0  \text{ and }  \lip g \leq \frac{1}{10} \right\}\ ,
$$
and for $g \in \Wc(x)$, we let $\Psi_x g$ denote the {\it graph transform} of $g$
if it is defined; i.e., $ \Psi_x g : \tilde{B}^u_{fx}(\d l(fx)^{-1}) \to  
E^{cs}_{fx}$ is the map with the property that
$$
\tilde{f}_x \left( \graph g \right) \supset \graph \Psi_x g \ .
$$

\begin{lem}[Contraction of graph transforms] \label{lem:graphTprops}
The following hold for every $x \in \Gamma$:
\begin{itemize}
\vspace{-4pt}
\item[(i)] for every  $g \in \Wc(x)$, $\Psi_x g$ is defined and is $\in \Wc(fx)$;
\vspace{-4pt}
\item[(ii)] there exists a constant $c \in (0,1)$ such that for all $g_1, g_2 \in \Wc(x)$,
$$
\tri{\Psi_x g_1-\Psi_x g_2}_{fx} \le c \ \tri{g_1-g_2}_x 
$$
where
$$
\tri{ h}_z = \sup_{v \in \tilde{B}^u_z(\d l(z)^{-1}) \setminus \{0\}} \frac{|h(v)|_z'}{|v|_z'}
\qquad \mbox{for} \qquad z=x, fx\ .\footnote{In the case when $(f, \mu; df)$ does not have zero exponents, the uniform norm \begin{align*} \|h\|_{z, \infty} = \sup_{v \in \tilde{B}^u_z(\d l(z)^{-1})} |h(v)|_z' \end{align*} is
often used when stating contraction estimates for the graph transform. }
$$
\end{itemize}
\end{lem}

We will also need the following characterization of unstable manifolds, valid only in 
the absence of zero exponents.

\begin{lem} \label{char}
Assume $E^c = \{0\}$. For  $\d$ small enough, the following hold for all
$x \in \Gamma$. 
\begin{itemize} \vspace{-4pt}
\item[(a)]  $W^u_{\d, x}$ has the characterization \vspace{-4pt}
$$
W^u_{\d, x}  = \exp_x \big\{z \in \tilde{B}_x( \d l(x)^{-1}) : \forall n \in \N, \exists   
z_n \in \tilde{B}_{f^{-n} x}( \d l(f^{-n} x)^{-1}) \mbox{ s.t. } \tilde{f}_{f^{-n} x}^n z_n = z \big\} \, ;
$$
\item[(b)] for $y \in W^u_{\d, x}$, the tangent space $E^u_y :=  T_y W^u_{\d, x}$ to $W^u_{\d, x}$ at $y$ has the characterization \vspace{-4pt}
\[
E^u_y = \big\{v \in \Bc_y : df^{-n}_y v \text{ exists for all } n \geq 1, \text{ and } 
\limsup_{n \to \infty} \frac{1}{n} \log |df^{-n}_y v| \leq  - \la \big\} \, .
\]
\end{itemize} 
%
%Let $y \in W^u_{\d, x}$. Writing $E^u_y :=  T_y W^u_{\d, x}$, we have that
%\[
%E^u_y = \big\{v \in \Bc_y : df^{-n}_y v \text{ exists for all } n \geq 1, \text{ and } \lim_{n \to \infty} \frac{1}{n} \log |df^{-n}_y v| \leq  - \la \big\} \, .
%\]
%\item[(b)] $W^u_{\d, x}$ is characterized by the following property.
%\begin{align}\label{charWu}
%W^u_{\d, x}  = \exp_x \big\{z \in \tilde{B}_x( \d l(x)^{-1}) : \forall n \in \N, \exists   
%z_n \in \tilde{B}_{f^{-n} x}( \d l(f^{-n} x)^{-1}) \mbox{ s.t. } \tilde{f}_{f^{-n} x}^n z_n = z \big\} \, .
%\end{align}
%\end{itemize}
\end{lem}

\smallskip
The proof of Lemma \ref{char} involves the so-called ``backwards graph transform", which
is different in infinite dimensions because one cannot iterate backwards. We recall
the definition of this transform  $\Psi^s_x$, as it will be used a number of times.

Continuing to assume
$E^c = \{0\}$, we define
\[
\Wc^s_{\frac12}(x) = \bigg\{ h : \tilde B^s_x(\d l(x)^{-1}) \to \tilde B^u_x( \d l(x)^{-1}) \ \mbox{ s.t. }
 |h(0)|'_x \le \frac12 \d l(x)^{-1}  \text{ and } \lip h \leq \frac{1}{10} \bigg\} \, .
\]
For $h \in \Wc^s_{\frac12}(x)$, if $\ell : \tilde{B}^s_{f^{-1} x}(\d l(f^{-1} x)^{-1} ) \to \tilde{B}^u_{f^{-1}x}(\d l(f^{-1} x)^{-1})$ is a map with the property that 
\begin{align} \label{eq:defineGraphTform}
\tilde{f}_{f^{-1} x} (\graph \ell ) \subset \graph h \, ,
\end{align}
then we say $\ell$ is the graph transform of $h$ by $\tilde f^{-1}$, and write
$\ell =  \Psi^s_x h$. The result in Lemma \ref{lem:existsBackwardsTform} 
was proved in \cite{lianyou1} for 
Hilbert space maps in a context similar to ours; their proof generalizes without change
to Banach spaces. 

\begin{lem}\label{lem:existsBackwardsTform} Assume $E^c = \{0\}$. Let $x \in \Gamma$ and $h \in \Wc^s_{\frac12}(x)$. Then
\begin{itemize}
\vspace{-4pt}
\item[(i)] $\Psi^s_x h$ is well defined and is $\in \Wc^s_{\frac12}(f^{-1} x)$; 
\vspace{-4pt}
\item[(ii)] for $z_1, z_2 \in \graph \Psi^s_x h$,
$$
|\tilde f_{f^{-1}x} z_1 - \tilde f_{f^{-1}x} z_2|_x' \leq (e^{- \la} + \d) |z_1 - z_2|_{f^{-1} x}'\ .
$$
\end{itemize}
(i) and (ii) continue to hold for graph transforms by maps that are $C^1$
sufficiently near $\tilde f_{f^{-1}x}$.
\end{lem}

\begin{proof}[Proof of Lemma \ref{char}.] (a) Let $z \in \tilde{B}_x( \d l(x)^{-1})$ be such
that $z \not \in \graph g_x$, and let $\hat z \in \graph g_x$ be such that 
$\pi^u_x z = \pi^u_x \hat z$. For  $n \in \N$, we let 
$z_n \in \tilde{B}_{f^{-n} x}( \d l(f^{-n} x)^{-1})$ be such that  $ \tilde{f}_{f^{-n} x}^n z_n = z$
if such a $z_n$ exists, and let $\hat z_n \in \graph g_{f^{-n}x}$ be such that 
$\tilde f_{f^{-n}x}^n \hat z_n = \hat z$; we know $\hat z_n$ exists for all $n$. 
We will show that $z_n$ and $\hat z_n$ diverge exponentially at a rate faster than
$\d_2$, so one of them 
must leave the chart eventually.

We assume there exists $z_1 \in \tilde{B}_{f^{-1} x}( \d l(f^{-1} x)^{-1})$ such that 
$\tilde{f}_{f^{-1} x} z_1 =z$; if no such $z_1$ exists, we are done.
Let $h \in \Wc^s_{\frac12}(x)$ (defined using $\tilde{B}_{f^{-1} x}( 2\d l(f^{-1} x)^{-1})$)
be the constant map $h(v) \equiv  \pi^u(z)$, and let $h_1 = \Psi^s_x h$.
We claim that $z_1 \in \graph h_1$. If not, let $\pi^s_{f^{-1}x} z_1 = s_1 \in 
\tilde{B}^s_{f^{-1} x} (\d l(f^{-1} x))$. By standard hyperbolic estimates, 
\[
|\pi^u_x \big( z - \tilde{f}_{f^{-1} x}(h_1(s_1) + s_1) \big) |_{ x}' \geq |\pi^s_x \big( z - \tilde{f}_{f^{-1} x}(h_1(s_1) + s_1 ) \big)|_{ x}' \, ,
\]
contradicting $z, \tilde{f}_{f^{-1} x}(h_1(s_1) + s_1) \in \graph h$. This proves $z_1 \in \graph h_1$.
By Lemma \ref{lem:existsBackwardsTform} (ii), $|z_1 - \hat z_1|'_{f^{-1}x} \ge (e^{-\lambda}
+ \d)^{-1} |z - \hat z |'_x$.

Repeating the argument in the last paragraph with $h_1$ in the place of $h$, we obtain
that either there does not exist $z_2 \in \tilde{B}_{f^{-2} x}( \d l(f^{-2} x)^{-1})$ such that 
$\tilde{f}_{f^{-2} x} z_2 =z_1$, or $|z_2 - \hat z_2|'_{f^{-2}x} \ge (e^{-\lambda}
+ \d)^{-1} |z_1 - \hat z_1 |'_{f^{-1}x}$, providing the exponential divergence of
$z_n$ and $\hat z_n$ claimed.

\medskip 
Part (b) is proved similarly: Continuing to work in charts, we let $\tilde y = 
\exp_x^{-1} y$, and let $\tilde y_{-n}$ be such that $\tilde f_{f^{-n}x}(\tilde y_{-n})=
\tilde y$. We assume $\d$ is small enough that Lemma \ref{lem:existsBackwardsTform} applies to backward
graph transforms by the linear maps $d(\tilde f_{f^{-n}x})_{\tilde y_{-n}}$. 
Repeating the argument above using these graph transforms, we conclude that for
$v \in \Bc(y)$ such that $v \not \in E^u_y$, either $df^{-n}_y v$ is not defined for some
$n$, or  $ |df^{-n}_y v|$ diverges exponentially as $n \to \infty$.
\end{proof}

%%%%%%%%%%%%%%%%%%%%%%%%%%%%%%%%%%
\subsection{Unstable stacks}

Notice that we have not made any assertion in Theorem \ref{thm:unstabMfld} 
regarding the regularity 
of the assignment $x \mapsto g_x$. In finite dimensions and on separable Hilbert
spaces, one often asserts that $W_{\d,x}^u$ varies measurably with $x$.
In the spirit of the discussion at the end of  Sect. 3.1, we will prove here
the continuity of unstable leaves on certain measurable sets. That almost every point is
contained in a ``stack of unstable manifolds" will be
relevant in Section 6. 

We first define precisely what is meant by such a stack. Consider nearby points 
$x,y \in \Gamma$ with $d_H(E^u_y, E^u_x), \ d_H(E^{cs}_y, E^{cs}_x) \ll 1$.  Let 
$\phi_y : \operatorname{Dom}(\phi_y) \to E^{cs}_y$, where $\operatorname{Dom}(\phi_y) 
\subset E^u_y$; we let $\phi_y^x : \operatorname{Dom}(\phi_y^{x}) \to E^{cs}_x$ 
(with $\operatorname{Dom}(\phi_y^{x}) \subset E^u_{x}$) 
be the mapping for which
\begin{align}\label{eq:graphEquality}
\exp_{y} ( \graph \phi_y ) = \exp_{x} ( \graph \phi_y^{x} ) 
\end{align}
if such a mapping can be uniquely defined; whether or not this can be done depends on
$x, y$ and $\phi_y$. We say that $\phi_y^{x}$ is \emph{defined on} $V \subset E^u_x$ if 
$\operatorname{Dom}(\phi_y^x) \supset V$.

%For $\e, \e'>0$ and 
%$\phi_y: B^u_y(\e) \to E^{cs}_y$, we let
%$\phi_y^x: B^u_x(\e') \to E^{cs}_x$ be the mapping with
%$$
%\exp_y(\graph \phi_y) \cap \exp_x(B_x(\e')) = \exp_x(\graph (\phi_y^x))
%$$
%if such a mapping can be uniquely defined; 
%whether or not that can be done depends on $x,y,
%\phi_y, \e$ and $\e'$. 
%

In discussions that involve more than one chart, it is natural to 
use $|\cdot|$ norms rather than 
the pointwise adapted $|\cdot|'_x$ norms. We introduce the following notation:
For $\tau =u,c,s$, let $B^\tau_x(r) = \{v \in E^\tau_x: |v|\le r\}$, and  let $B_x(r) = B^u_x(r) + B^{cs}_x(r)$.
(Notice the distinction between $B_x^\tau(r)$ and $\tilde B^\tau_x(r)$.) Below, the space $C(B^u_x(r), E^{cs}_x)$ of 
continuous functions from $B^u_x(r)$ to $E^{cs}_x$ will be endowed with
 its $C^0$ norm $\|\cdot\|$ defined using the $|\cdot|$ norm on $E^{cs}_x$.

Recall the definition of $\Gamma_{l}$ at the end of Section 4 and the definition
of $K_n$ in Sect. 3.1.

\begin{lem} \label{lem:uStacks}
Let $l_0$ and $n_0 > 1$ be fixed, and fix $x_0 \in \Gamma_{l_0} \cap K_{n_0}$.
For $\e>0$ and $x \in \Gamma$, we let 
$$U(x, \e) := \Gamma_{l_0} \cap K_{n_0} \cap \{y:|x-y|<\e\}\ ,
$$ and let $g_y$
be as in Theorem \ref{thm:unstabMfld}. Then for any $\d \leq \frac14 \d_1'$, we may choose $\e_0 > 0$ sufficiently small so that the following holds:
\begin{itemize}
\vspace{-4pt}
\item[(a)] for every $y \in U(x_0, \e_0)$, $g^{x_0}_y$ is defined on $B^u_{x_0}(\d l_0^{-3})$, and
\vspace{-4pt}
\item[(b)] the mapping $\Theta: U(x_0, \e_0) \to C(B^u_{x_0}(\d l_0^{-3} ), E^{cs}_{x_0})$
defined by $\Theta(y) = g^{x_0}_y$ is continuous. 
\end{itemize}
\end{lem}

\smallskip
We will refer to sets of the form
\begin{align}\label{eq:stackForm}
\Sc = \bigcup_{y \in \bar U} \exp_{x_0}(\graph \Theta(y)) \, ,
\end{align}
where $x_0, U(x_0, \e_0)$ and $\Theta$ are as in Lemma \ref{lem:uStacks}
and $\bar U \subset U(x_0, \e_0)$ is a compact subset,
as a {\it stack of local unstable manifolds}.

\begin{proof} (a) We begin by giving sufficient conditions for $\phi^x_y$ to be
defined on $B^u_x(2 \d l_0^{-3})$ for a given $\phi_y : \operatorname{Dom}(\phi_y) \to E^{cs}_y$. Identifying $\Bc_x$ and $\Bc_y$, the tangent spaces at $x$ and $y$,
with $\Bc + x$ and $\Bc +y$ respectively, we define $\Xi_y^{x} : \operatorname{Dom}(\phi_y) \to E^u_x$ by 
\begin{align}\label{eq:graphingCoordinateMap}
\Xi_y^{x}(v) = \pi^u_x((\Id_{E^u_y} + \phi_y)(v) + y-x)\ , 
\end{align}
so that formally, at least, $\phi^x_y(w) = \pi^{cs}_x((\Id_{E^u_y} + \phi_y)((\Xi^x_y)^{-1}(w)) +
y-x)$. From this, we see that
 $ \phi^x_y$ is well defined on $\Xi^x_y(\operatorname{Dom}(\phi_y))$
if $\pi^u_x$ is invertible when restricted to the set $\graph(\phi_y)+y-x$. This is guaranteed if for all $w_1, w_2 \in \graph (\phi_y)$, one has
$|\pi^{cs}_x(w_1-w_2)| < |w_1-w_2|$, and that is implied by
\begin{equation} \label{changechart}
|\pi^{cs}_{x}|_{E^u_y}| + \operatorname{Lip}(\phi_y) \cdot |\pi^{cs}_{x}|_{E^{cs}_y}| < 1\ .
\end{equation}

We bound $\operatorname{Lip}(\phi_y)$ as follows: First we work in $\Bc_y$,
 letting $\operatorname{Lip}'_y(\phi_y)$ denote the Lipschitz constant of 
$\phi_y$ with respect to the norm $|\cdot|'_y$. Assume that $d(\phi_y)_0 =0$ 
and $\operatorname{Lip}'_y(d(\phi_y)) < Cl_0$; these properties are enjoyed
by $g_y$, the graphing map for the local unstable manifold at $y$ (Theorem 5.1). Then assuming
$\operatorname{Dom}(\phi_y) \subset \tilde B^u_y(4\d l_0^{-2})$, we have
$\operatorname{Lip}'_y(\phi_y) < 4C\d l_0^{-1}$. Passing to the $|\cdot|$ norm, we
have $\operatorname{Lip}(\phi_y) < 12 C\d$, which we may assume is $\ll 1$.

Thus with $\epsilon_0$ small enough that $d_H(E^u_x, E^u_y)$ and 
$d_H(E^{cs}_x, E^{cs}_y)$ are sufficiently small (depending only on $n_0, l_0$), 
the inequality in (\ref{changechart}) is satisfied.

Finally, let $\operatorname{Dom}(\phi_y) = B_y^u(4\d l_0^{-3}) \subset
\tilde B^u_y(4\d l_0^{-2})$. 
Shrinking $\epsilon_0$ if necessary so that $|x-y|$ is sufficiently small, we have
$\Xi_y^x(\operatorname{Dom}(\phi_y)) \supset B^u_x(2 \d l_0^{-3})$,
proving (a).

For (b), to prove the continuity of $\Theta$ at $x \in U=U(x_0, \e_0)$, we let $y^n \in U$
be a sequence with $y^n \to x$ as $n \to \infty$. That $\Theta(y^n) \to \Theta(x)$ 
on $B_x^u(\d l_0^{-3})$ will follow once we show $\|g_{y^n}^{x}-g_x^{x}\| \to 0$ on 
$B^u_x(2 \d l_0^{-3})$. To do this, we will show that given $\gamma > 0$, we have 
$\|g^{x}_{y^n} - g_x^{x}\| < \gamma$ for all $n$ sufficiently large.

For $k \in \mathbb Z^+$, write $x_{-k} = f^{-k} x$ and $y^n_{-k} = f^{-k}y$. 
Since  $x_{-k}, y^n_{-k} \in \Gamma_{l_0 e^{k\d_2}}$ (Lemma \ref{lem:slowvary}), we have, by Lemma \ref{lem:contSplitting}, $E^u_{y^n_{-k}} \to E^u_{x_{-k}}$ as $n \to \infty$. (We could not
have concluded this from the continuity of $E^u$ on $K_{n_0}$ alone because 
$df_x$ is not invertible.) Thus for a fixed $k$, $\exp_{y^n_{-k}}(B^u_{y^n_{-k}}(1))$ 
converges as a family of embedded disks to $\exp_{x_{-k}}(B^u_{x_{-k}}(1))$, so their $f^k$-images 
converge as well. Another way to express this is as follows: Let ${\bf 0}_{y^n_{-k}}: 
\tilde B^u_{y^n_{-k}}(\d'_1l(y^n_{-k})^{-2}) \to E^{cs}_{y^n_{-k}}$ be the function that is identically equal to $0$,
and let $\phi_{y^n,k} = \Psi_{y^n_{-1}} \circ \cdots \circ \Psi_{y^n_{-k}}({\bf 0}_{y^n_{-k}})$
where $\Psi$ is the graph transform. 
Likewise, define $\phi_{x,k}$. Then for each fixed $k$, $\|\phi_{y^n,k}^x - \phi_{x,k}^x\|
\to 0$ (as mappings defined on $B^u_x(2 \d l_0^{-3})$) as $n \to \infty$.

To finish, we estimate $\|g_{y^n}^{x}-g_x^{x}\|$ by
$$
\|g_{y^n}^{x}-g_x^{x}\| \le \|g_{y^n}^{x}-\phi^x_{y^n,k}\| + \|\phi_{y^n,k}^{x}-\phi^x_{x,k}\|
+ \|\phi_{x,k}^{x}-g_x^x\|\ .
$$
Using Lemma \ref{lem:graphTprops} and the uniform equivalence of the $|\cdot|$ and $|\cdot|_y'$ norms
on $\Gamma_{l_0}$, we have that the first and third terms above are $< \gamma/3$
for $k$ large enough. Fix one such $k$, and choose $n$ large enough that the middle term
is $<\gamma/3$. This gives the desired estimate.
\end{proof}

\begin{rmk}\label{rmk:versatileStacks}
Lemma \ref{lem:uStacks} guarantees that $\mu$-a.e. $x$ is contained in a stack,
but  observe that the involvement of the $\mu$-continuity set $K_{n_0}$ is solely to guarantee control on $E^{cs}$. That is to say, if $V \subset \Gamma_{l_0}$ is 
such that \eqref{changechart} holds for all $x, y \in V$ (setting $\phi_y = g_y|_{\tilde B^u_y(4 \d l_0^{-3})}$), and has sufficiently small diameter (depending only on $l_0$), then Lemma \ref{lem:uStacks} holds with $V$ in the place
of $U(x_0, \e_0)$. From here on we will extend the definition of stacks to include
 sets of the form 
$\Sc = \bigcup_{y \in V} \exp_{x_0}(\graph \Theta(y))$ for compact $V$ with
the properties above. 
\end{rmk}

To complete the geometric picture, we will show that $\Sc$ is homeomorphic to 
the product of a finite-dimensional ball with a compact set, and we will do this in the 
absence of zero Lyapunov exponents (the zero exponent case will require that we
strengthen Lemma \ref{char}).

%Later on, we will need to carry out disintegrations and other measure-theoretical constructions on unstable stacks. These constructions are most easily carried out by regarding a stack as homeomorphic to the product of a finite-dimensional ball with a compact set, as we show in the following lemma.

\begin{lem}\label{lem:homeoStack} Assume $E^c=\{0\}$, 
and let $\Sc$ be an unstable stack as defined in \eqref{eq:stackForm}. Let $\Sigma = \exp_{x_0}^{-1}\big(\Sc \big) \cap E^{s}_{x_0}$. Then $\Sc$ is homeomorphic to $\Sigma \times B^u_{x_0}(\d l_0^{-3})$ under the mapping $\Psi(\sigma, u) := \exp_{x_0}\big( u + g_{\sigma}(u) \big)$, where $g_{\sigma} = \Theta(y)$ corresponds to the unique leaf of $\Sc$ for which $\Theta(y)(0) = \sigma$.
\end{lem}

\begin{proof}
That $g_{\sigma}$ is well-defined for $\sigma \in \Sigma$ follows from the fact that distinct leaves in the unstable stack do not intersect, and that in turn is a direct consequence of Lemma 
\ref{char}.
%
%
%. Equipped with the backwards graph transform as in Lemma \ref{lem:existsBackwardsTform} (and in particular, property (iii)), the proof of this proceeds identically with that of the finite-dimensional case.

We now check that $\Psi$ is a homeomorphism. By compactness, it suffices to check that $\Psi$ is a continuous bijection. To prove continuity, we define the (continuous) map $\theta : \bar U \to \Sigma$ by $\theta(y) = \Theta(y)(0)$ and the equivalence relation $\sim$ on $\bar U$ by $x \sim y$ iff $\Theta(x) = \Theta(y)$, i.e., if $x$ and $y$ fall on the same unstable leaf in $\Sc$.
As $\Theta, \theta$ are constant on the equivalence classes of $\sim$, they descend to continuous maps $\tilde{\Theta}, \tilde{\theta}$ defined on the quotient space $\bar U / \sim$. The map $\tilde{\theta}$ is a continuous bijection, hence a homeomorphism (by compactness of $\bar U$), and so the proof is complete on noting that the mapping $\sigma \mapsto g_{\sigma}$ can be represented by the composition $\tilde{\Theta} \circ \tilde{\theta}^{-1}$.
\end{proof}

%%%%%%%%%%%%%%%%%%%%%%%%%%%%%%%%%%%%%%%%%
%%%%%%%%%%%%%%%%%%%%%%%%%%%%%%%%%%%%%%%%%
%%%%%%%%%%%%%%%%%%%%%%%%%%%%%%%%%%%%%%%%%
%%%%%%%%%%%%%%%%%%%%%%%%%%%%%%%%%%%%%%%%%

\subsection{Induced volume on submanifolds, and distortion estimates 
along unstable leaves}\label{sec:distEst}

In Section 2, we introduced $m_E$, a notion of volume induced on finite 
dimensional subspaces $E$ of $\Bc$. It is straightforward to extend this idea
to volumes on embedded (or injectively immersed) finite dimensional submanifolds:
Let $U \subset \R^d$ be an open set, $\phi: U \to \Bc$ a $C^1$ 
Fr\'echet embedding, and $W= \phi(U)$. For 
a Borel subset $V \subset W$, we define
$$
\nu_{\phi,W} (V) = \int_{\phi^{-1} V} \det(d\phi_y) \ dy
$$
where we have identified the tangent space at $y \in \R^d$ with $\R^d$ and 
$\det(d\phi_y)$ here is taken with respect to Euclidean volume on $\R^d$ and 
$m_{T_{\phi(y)} W}$ on the tangent space $T_{\phi(y)} W$ to $W$ at $\phi(y)$. 
That $\nu_{\phi,W}$ does not depend on $\phi$ is checked in the usual way: 
Let $\phi':U' \to \Bc$ be another embedding with
$\phi'(U')=W$. Then $\phi' = \phi \circ (\phi^{-1} \circ \phi')$, and since
$\phi^{-1} \circ \phi':U' \to U$ is a diffeomorphism, we have, by the multiplicativity
of determinants and the usual change
of variables formula,
$$ 
\nu_{\phi',W} (V) = \int_{(\phi')^{-1} V}  \det(d \phi_{(\phi^{-1} \circ \phi')(y')}) 
\det(d ( \phi^{-1} \circ \phi')_{y'}) \ dy'
 = \int_{\phi^{-1} V} \det(d\phi_y) \ dy\ .
$$
We will denote the induced volume on $W$ by $\nu_W$ from here on, having
shown that it is independent of embedding. The discussion above 
is easily extended to injectively immersed finite dimensional submanifolds,
such as unstable manifolds. 

For $x \in \Gamma$, let us abbreviate $\nu_{W^u_x}$ as $\nu_x$, and for $y\in
W^u_x$, we will use $E^u_y$ to denote the tangent space to $W^u_x$ at 
$y$.\footnote{This is in fact true even though $y$ is not necessarily in $\Gamma$:
since $|f^{-n}x-f^{-n}y| \to 0$ exponentially as $n \to \infty$, 
and the tangent spaces of $f^{-n}x$ and $f^{-n}y$ to $W^u_{f^{-n}x}$ converge exponentially as well, it follows that  
backward time Lyapunov exponents for $y \in W^u_x$ are well defined 
and are identical to those at $x$, with $E^u_y$ being the tangent space to 
$W^u_x$ at $y$.}
Then letting $f_* \nu_{f^{-1}x}$ denote the pushforward of $\nu_{f^{-1}x}$
from $W^u_{f^{-1}x}$ to $W^u_x$, we have, from the discussion above,
\begin{equation} \label{changevariable}
\frac{d (f_* \nu_{f^{-1} x})}{d \nu_x}(y) = \frac{1}{\det (df_{f^{-1}y} | E^u_{f^{-1} y})} 
\qquad \mbox{ for } \ y \in W^u_x\ .
\end{equation}
The distortion estimate below is crucially important for the arguments
in Section 6. Note that we allow $E^c \neq \{0\}$.

Let $\d_1'$ be the largest $\d$ for which Theorem \ref{thm:unstabMfld} holds. Let us write
$W^u_{{\rm loc},x}=W^u_{\d'_1,x}$.

\begin{prop}\label{prop:unstabDistEst}
For every $l \geq 1$, there is a constant $D_l$ such that the following holds for any $x \in \Gamma_l$.
\begin{itemize}
\item[(a)] For all $y^1, y^2 \in W^u_{{\rm loc},x}$ 
%\cap \exp_x \tilde{B}_x (\d'_1 D_{l}^{-1})$ 
and all $n \ge 1$:
\begin{equation} \label{distortion1}
\left|\log \frac{\det (df^n_{f^{-n} y^1} | E^u_{f^{-n} y^1})  }
{\det (df^n_{f^{-n}y^2} | E^u_{f^{-n} y^2})} \right| \ \le \ D_{l} \ |y^1-y^2|\ .
\end{equation}
\item[(b)] For any fixed $x' \in W_{{\rm loc}, x}^u$, the sequence of functions $y \mapsto  \log \Delta_N(x',y)$, where
\begin{align}\label{eq:finiteProducts}
\Delta_N(x',y) := \prod_{k = 1}^{N} \frac{\det(df_{f^{-k} x'} | E^u_{f^{-k} x'})}{\det(df_{f^{-k} y} | E^u_{f^{-k} y})} \qquad N=1,2,\dots,
\end{align}
defined for $y \in W^u_{{\rm loc},x}$ 
%\cap \exp_x \tilde{B}_x (\d'_1 D_{l}^{-1})
converges uniformly (at a rate depending only on $l(x)$) as $N \to \infty$ to a Lipschitz continuous function with constant $\leq D_l$ in the $|\cdot|$ norm.
\end{itemize}
\end{prop}

Though distortion estimates of the kind in Proposition \ref{prop:unstabDistEst}
are standard for finite dimensional systems, they are new for mappings
of infinite dimensional Banach spaces: These estimates have to do with the 
regularity of determinants for $df^n$ along unstable manifolds, where determinants are 
defined with respect to the induced volumes introduced in this paper. 
Below we include a complete proof, proceeding in two steps.  
In the first step, formulated as Proposition \ref{prop:chartDistEst}, we prove
a distortion estimate in {\it charts}, i.e., using Lyapunov metrics,
taking advantage of the uniform expansion 
along unstable leaves in adapted norms. In the second step, we bring 
this estimate back to the usual norm $|\cdot|$ on $\Bc$.

Fix $x \in \Gamma$; we introduce the following abbreviated notation.
In the first step we will be working exclusively with the maps $$\tilde f_{f^{-k}x} : 
\tilde B_{f^{-k}x}(\d'_1 l(f^{-k} x)^{-1}) \to \Bc_{f^{- (k-1)}x},  \qquad k \ge 1\ ,
$$ where 
the notation is as in Theorem \ref{thm:unstabMfld}, and the norm of interest 
on each $\Bc_{f^{-k}x}$ is exclusively
$|\cdot|'_{f^{-k}x}$. As the meanings will be clear from context, 
we will drop the subscripts in
$\tilde f_{f^{-k}x}$ and $|\cdot|'_{f^{-k}x}$, writing only $\tilde f$ and $|\cdot|'$.
For a finite dimensional subspace $E \subset \Bc_{f^{-k}x}$, $m_E'$ will
denote the volume on $E$ induced from $|\cdot|'$, and $\det'$ is to be understood
to be the determinant with respect to these volumes. We also write 
$g=g_x$ and $g_{-k}=g_{f^{-k}x}$, the graphing maps of 
$W^u_{{\rm loc}, f^{-k}x}$ given by 
Theorem \ref{thm:unstabMfld}.

\begin{prop}\label{prop:chartDistEst}
For any $l \geq 1$, there is a constant $D_l'$ with the following property. Let $x \in \Gamma$. Then for any $z^1, z^2 \in \graph g$ with $|z^1-z^2|' \leq \d_1' (D_{l(x)}')^{-1}$, 
$i = 1,2$, and any $n \ge 1$, we have that
\begin{align}\label{eq:chartDistForm}
\bigg| \log  { \det' (d \tf^n_{z^1_{-n} } | E^1_{-n})  \over \det' (d \tf^n_{z^2_{-n}} | E^2_{-n}) } \bigg| \leq  D_{l(x)}' |z^1 - z^2|'\ ,
\end{align}
where $z^i_n$ is the unique point in $\graph g_{-n}$ with 
$\tf^n z^i_{-n} = z^i$, and $E^i_{-n}$ is the tangent space to 
$\graph g_{-n}$ at $z^i_n$.
\end{prop}

\begin{proof}[Proof of Proposition \ref{prop:chartDistEst}] Consider to begin with
arbitrary $z^1, z^2 \in \graph g$.
Using the multiplicativity of the determinant, we decompose the argument of 
$\log$ in the LHS of \eqref{eq:chartDistForm} as
\begin{align}\label{eq:chartExpanded}
 \frac{\det'(d \tf^n_{z^1_{-n} } | E^1_{-n}) }
 { \det' ( d \tf^n_{z^2_{-n}} | E^2_{-n})} 
 = \prod_{k = 1}^n \frac{\det' ( d \tf_{z^1_{-k}} | E^1_{-k}) }
 {\det' ( d \tf_{z^2_{-k}} | E^2_{-k}) }  
\end{align}
and bound the factors on the right side of (\ref{eq:chartExpanded}) one at a time.

We will use the following slight refinement of Proposition \ref{prop:detReg} 
(see Remark \ref{rmk:constantsDependence}).

\begin{lem}\label{lem:refineDetReg}
For $m \in \N$, there is a constant $C_m > 1$ with the following property. 
Let $X, Y$ be Banach spaces, and fix $M \geq 1$. If $A_1, A_2 : X \to Y$ are bounded linear operators and $E_1, E_2 \subset X$ are subspaces with the same finite dimension $m$ such that
\begin{gather*}
|A_i|,~ |(A_i|_{E_i})^{-1}| \leq M \quad i=1,2 ,\\
|A_1 - A_2|, ~d_H(E_1, E_2) \leq \frac{1}{C_m M^{10 m}}\ ,
\end{gather*}
then
\[
\bigg| \log \frac{\det(A_1 | E_1)}{\det(A_2 | E_2)} \bigg| \leq C_m M^{10 m } ( |A_1 - A_2| + d_H(E_1, E_2) ) \ .
\]
\end{lem}

For each fixed $1 \leq k \leq n$, we apply Lemma \ref{lem:refineDetReg} to $m = m_u$, $A_i = d \tf_{z^i_{-k}}$ and $E_i = E^i_{-k}$, for an appropriate choice of $M = M_k$. To fulfill the hypotheses of the lemma, we need the following:
\begin{gather}
\label{eq:conda} | d \tf_{z^i_{-k}} |', \ \ | \big( d \tf_{z^i_{-k}} |_{E^i_{-k} } \big)^{-1} |' 
\ \leq \ M_k  \quad \text{ for } i = 1,2, \\
\label{eq:condb} |d \tf_{z^1_{-k}} - d \tf_{z^2_{-k}} |' , \ \  d_H'(E^1_{-k}, E^2_{-k})  \ 
\leq \ C_{m_u}^{-1} M_k^{-10 m_u},
\end{gather}
where $d_H'$ refers to the Hausdorff distance in the adapted norm $|\cdot|'$ on $\Bc_{f^{-k} x}$. 

First we choose $M_k$ so that \eqref{eq:conda} holds: 
$|(d \tf_{z^i_{-k}}|_{E^i_{-k}})^{-1}|' \leq 1$ poses no problem, but from the way
our adapted norms are defined in Sect. 4.1, we only have
\[
|d \tf_{z^i_{-k}}|' \leq 3 l(f^{-(k-1)} x) \cdot |d f_{\exp_{f^{-k} x}(z^i_{-k})}| \leq 3 K e^{k \d_2} l(x) \ ,
\]
where $K$ is an upper bound for $|df|$ on $\{y \in \Bc : d (y, \As) \leq r_0\}$ (see the paragraphs preceding Lemma \ref{prop:chartProps}).  So, on setting $M_k = 3 K e^{k \d_2} l(x)$, \eqref{eq:conda} is satisfied. 

Next we estimate the two terms on the left side of \eqref{eq:condb}:
\begin{align*}
|d \tf_{z^1_{-k}} - d \tf_{z^2_{-k}} |' & \leq l(f^{-k} x) |z^1_{-k} - z^2_{-k}|'  & \text{ by Lemma \ref{prop:chartProps},} \\
& \leq l(f^{-k} x) (e^{\la} - \d_1')^{-k} |z^1 - z^2|' &  \text{ by Item 3 of Theorem \ref{thm:unstabMfld},} \\
& \leq l(x) \bigg( \frac{e^{\d_2}}{e^{\la} - \d_1'} \bigg)^k \cdot |z^1 - z^2|' =: (*) & \text{ by Lemma \ref{lem:slowvary}.} 
\end{align*}
For $d_H'(E_{-k}^1, E^2_{-k})$, observe that if $z^i_{-k} = g_{-k}(u_{-k}^i)$, then
$E_{-k}^i = (\Id + (d g_{-k})_{u_{-k}^i}) E^u(f^{-k} x)$. A simple computation (see
Sect. 2.1.2) gives  
$$d_H'(E_{-k}^1, E^2_{-k}) \leq 2 \ |(d g_{-k})_{u_{-k}^1} - (d g_{-k})_{u_{-k}^2}|'\ ,
$$ 
hence
\begin{align*}
d_H'(E_{-k}^1, E_{-k}^2) & \leq 2 C l(f^{-k} x) |u_{-k}^1 - u_{-k}^2|' & \text{by Item 2 of Theorem \ref{thm:unstabMfld},} \\
& = 2 C  l(f^{-k} x) |z^1_{-k} - z^2_{-k}|' 
\ \le \ 2C \cdot (*)\ .
\end{align*}
Notice that while \eqref{eq:conda} imposes a lower bound on $M_k$, 
\eqref{eq:condb} imposes an {\it upper} bound, namely
$M_k^{10 m_u} \le (C_{m_u} \max\{2C, 1\} \cdot (*))^{-1}$. Both conditions can be
satisfied if $|z^1 - z^2|'$ is sufficiently small, such as $|z^1 - z^2|' < \d'_1 D_{l(x)}^{-1}$
where
\[
D_{l}' \geq \d_1'  C_{m_u} (3K)^{10 m_u} \cdot \max\{2 C, 1\} \cdot   l^{10 m_u + 1}\ ,
\]
assuming, as we may, that $e^{\la} - \d'_1 > e^{(10 m_u + 1) \d_2}$.

At last, we apply Lemma \ref{lem:refineDetReg} to $z^1, z^2$ with
$|z^1 - z^2|' < \d'_1 D_{l(x)}^{-1}$, obtaining
\begin{align}
\bigg| \log \frac{\det(d \tf_{z^1_{-k}} | E^1_{-k} )}{ \det(d \tf_{z^2_{-k}} | E^2_{-k} )} \bigg| &  \leq C_{m_u} (3 K e^{ k \d_2} l(x))^{10 m_u } \cdot (2 C + 1)  l(x) \bigg( \frac{e^{\d_2}}{e^{\la} - \d_1'} \bigg)^k  \cdot |z^1 - z^2|' \notag \\
& = K' \bigg( \frac{e^{(10 m_u + 1) \d_2}}{e^{\la} - \d_1'} \bigg)^k l(x)^{10 m_u + 1} |z^1 - z^2|' \label{eq:kthTermEstimate} \ .
\end{align}
%\begin{align}
%\bigg| \log \frac{\det(d \tf_{z^1_{-k}} | E^1_{-k} )}{ \det(d \tf_{z^2_{-k}} | E^2_{-k} )} \bigg| &  \leq C_m (3 K e^{ k \d_2})^{10} \cdot \max\{2 C, 1\}  l(x) \bigg( \frac{e^{\d_2}}{e^{\la} - \d_1'} \bigg)^k  \cdot |z^1 - z^2|' \notag \\
%& = K' \bigg( \frac{e^{11 \d_2}}{e^{\la} - \d_1'} \bigg)^k l(x) |z^1 - z^2|' \label{eq:kthTermEstimate} \ .
%\end{align}
Reconstituting the expression \eqref{eq:chartExpanded}, we obtain the estimate
\[
\bigg| \log \frac{\det'(d \tf^n_{z^1_{-n} } | E^1_{-n}) }{ \det' ( d \tf^n_{z^2_{-n}} | E^2_{-n} ) }\bigg| \leq K' l(x)^{10 m_u + 1} |z^1 - z^2|' \cdot \sum_{k = 1}^n \bigg( \frac{e^{(10 m_u + 1) \d_2}}{e^{\la} - \d_1'} \bigg)^k \leq K'' l(x)^{10 m_u + 1} |z^1 - z^2|' \ ,
\]
where $K''$ is independent of $x$ and $n$. By increasing $D_l'$ once more so that $D_l' \geq K'' l^{10 m_u + 1}$, the conclusion of Proposition \ref{prop:chartDistEst} follows.
\end{proof}

We now complete the proof of Proposition \ref{prop:unstabDistEst}.

\begin{proof}[Proof of Proposition \ref{prop:unstabDistEst}]
Fix $x \in \Gamma_l$. For $y^1, y^2 \in W^u_{loc,x}$, we let $z^i = \exp_{x}^{-1} y^i$,
and  write $y^i_{-k} = f^{-k} y^i$. For objects and quantities in charts, we will use the
same notation as in Proposition \ref{prop:chartDistEst}, so for example, 
$y^i_{-k} = \exp_{f^{-k} x}\big(u^i_{-k} + g_{-k}(u^i_{-k})\big)$ etc.

For part (a), we first consider $y^1, y^2 \in W^u_{loc,x}$ with
$|z^1-z^2| \le \d'_1 D_l^{-1}$, proving that the left side 
of (\ref{distortion1}) is $\le D_l |z^1-z^2|'$ for some $D_l$
that will be enlarged a finite number of times in the course of the proof.
Fixing $n \ge 1$, we compute that
\begin{align*} 
\frac{\det(df^n_{y^1_{-n}} | E^u_{y^1_{-n}} )}{\det(df^n_{y^2_{-n}} | E^u_{y^2_{-n}} )} =  \underbrace{{dm_{E^u_{y^1}} / dm'_{E^u_{y^1}} \over dm_{E^u_{y^2}} / dm'_{E^u_{y^2}} }}_{I} \times 
\underbrace{{dm_{E^u_{y^2_{-n}}} / dm'_{E^u_{ y^2_{-n}}} \over dm_{E^u_{y^1_{-n}}} / dm'_{E^u_{y^1_{-n}}} } }_{II} 
  \times  \underbrace{{ \det' (d\tf^n_{z^1_{-n}} | E^1_{-n})  \over \det' (d\tf^n_{z^2_{-n}} | E^2_{-n} )  } }_{III} \ .
\end{align*}
By Proposition \ref{prop:chartDistEst}, we have 
\[
|\log III| \leq D'_{l} |z^1 - z^2|' \ .
%\leq l  D'_{l} |y^1 - y^2| \ .
\]
It remains to estimate the terms $I$ and $II$. 

For $I$, observe that 
if $L : E^u_x \to E^{cs}_x$ is a linear map with $|L|' \leq 1$, then as 
a consequence of (\ref{max}) in the definition of $|\cdot|'$ norms, 
\begin{align}\label{eq:indVolComp}
\frac{dm_{(\Id + L) E^u_x}}{dm'_{(\Id + L) E^u_x}} = 
\frac{m_{(\Id + L) E^u_x}\big((\Id + L) \Oc \big)}{m_{E^u_x}' \big( \Oc\big)} 
\end{align}
for any Borel subset  $\Oc \subset E^u_x$ of positive Haar measure. 
Since $\lip g \leq 1$ and $E^u_{y^i} = (\Id + (dg)_{u_0^i}) E^u_x$, it follows from \eqref{eq:indVolComp} that
\[
I = \frac{\det (\Id + (dg)_{u^1_0} | E^u_x)}{\det (\Id + (dg)_{u^2_0} | E^u_x)} \ .
\]
Note that all determinants involved are in the natural norm $|\cdot|$, considered as a norm on $\Bc_x \cong \Bc$.

We will estimate this expression for $I$ using Lemma \ref{lem:refineDetReg}, applying that result with $A_i = \Id + (d g)_{u_0^i}$ and $E_1 = E_2 = E^u_x$. First, 
\begin{align*}
|\Id + (d g)_{u^i_0}| & \leq 3 l |\Id + d g|' = 3 l \ , \\
|\big( (\Id + (d g)_{u_0^i})|_{E^u_x} \big)^{-1}| & \leq 3 l |\big( (\Id + (d g)_{u_0^i})|_{E^u_x} \big)^{-1}|' = 3 l \ .
\end{align*}
Here we have used again the fact that $\Id + (d g)_{u_0^i} : E^u_x \to E^u_{y^i}$ is an isometry in $|\cdot|'$. So, for the purpose of bounding $I$, we may take $M$ in
Lemma \ref{lem:refineDetReg} to be $M = 3 l$. The only estimate needed in
the analog of \eqref{eq:condb} is 
\begin{align*}
 |(d g_0)_{u^1_0} - (d g)_{u^2_0}| \leq 3 l |(d g)_{u^1_0} - (d g)_{u^2_0}|' \leq 3 l \cdot C l |u_0^1 - u_0^2|' = 3 C l^2 |z^1 - z^2|' \ ,
\end{align*}
so it suffices to enlarge $D_l$ to $D_l \geq 3^{10 m_u + 1}  \d_1' C C_{m_u} l^{10 m_u + 2}$.
Lemma \ref{lem:refineDetReg} then applies to give
\begin{align}\label{eq:Iestimate}
|\log I | \leq C_{m_u} (3 l)^{10 m_u } \cdot 3 C l^2 |z^1 - z^2|' \leq 
K''' l^{10 m_u + 2} |z^1 - z^2|' \ .
\end{align}

The estimate for $|\log II|$ proceeds similarly, replacing $g$ with $g_{-n}$ and 
$u^i_{0}$ with $u^i_{-n}$. We leave it to the reader to check that it has the 
same bound as $|\log I |$. This completes the proof of part (a) for 
$y^1, y^2 \in W^u_{{\rm loc},x}$ with $|z^1-z^2|' < \d'_1 D_l^{-1}$.

For $y^1, y^2 \in W^u_{{\rm loc},x}$ for which $|z^1-z_2|' = |u^1-u^2|'>
\d'_1 D_l^{-1}$, we insert points $\hat u^1, \dots, \hat u^k$ on the line segment joining 
$u^1$ and $u^2$ so that if $\hat u^0=u^1$ and $\hat u^{k+1}=u^2$, then $|\hat u^i-
\hat u^{i-1}|' \le \d'_1 D_l^{-1}$. Let $\hat z^i = \hat u^i + g(\hat u^i)$ and
$\hat y^i=\exp_x \hat z^i, i=0,1, \dots, k+1$. Then the argument above gives
\begin{eqnarray*}
\left|\log \frac{\det (df^n_{f^{-n} y^1} | E^u_{f^{-n} y^1})  }
{\det (df^n_{f^{-n}y^2} | E^u_{f^{-n} y^2})} \right| & = &
\left|\log \left(\prod_{i=0}^{k} \frac{\det (df^n_{f^{-n} \hat y^i} | E^u_{f^{-n} \hat y^i} ) }
{\det (df^n_{f^{-n} \hat y^{i+1}} | E^u_{f^{-n} \hat y^{i+1}} ) } \right) \right|\\
& \le & D_l (|\hat z^1-\hat z^0|' + \dots + |\hat z^{k+1}-\hat z^k|')\\
& =&  D_l |z_1-z_2|' \ \le \ l D_l |y^1-y^2|\ .
\end{eqnarray*}
This completes the proof of part (a).

\medskip
For part (b), observe that as a consequence of (a), it will suffice to show that 
the sequence $\log \Delta_N$ in \eqref{eq:finiteProducts} is uniformly Cauchy 
over $y \in W^u_{{\rm loc},x}$,
% \cap \exp_x \tilde{B}_x(\d_1 D_{l}^{-1})$,
 the value of $D_l$ having been fixed so that (a) holds. 
 This in turn will follow from (uniform in $y$) bounds on
 
\begin{align}\label{eq:tailProduct}
\log \prod_{k = M + 1}^{N} \frac{\det(df_{f^{-k} x'} | E^u_{f^{-k} x'}) }{\det(df_{f^{-k} y} | E^u_{f^{-k} y}) } = 
\log \frac{\det(df^{N-M}_{f^{-N} x'} | E^u_{f^{-N} x'}) }{\det(df^{N - M}_{f^{-N} y} | E^u_{f^{-N} y}) }
\end{align}
for $M, N$ large, $M < N$. We leave it to the reader to check that the functions in
(\ref{eq:tailProduct}) are bounded by quantities exponentially small in $M$.
\end{proof}

Tracing through the proof of Proposition \ref{prop:unstabDistEst}, one sees that $D_l$ 
can be taken as $C l^q$, where $q \in \N$ and $C$ depend only on $m_u$ and are independent of $l$.

\medskip

Observe that the proof of Proposition \ref{prop:unstabDistEst} used the Lipschitz regularity of the determinant in a crucial way. In some sense, the preceding proof used all possible regularity of the determinant available in our setting.

%\begin{rmk}\label{rmk:stackDistortion} Going forward, it will sometimes be convenient to use unstable stacks
%for which the functions $\Theta(y) = g_y^{x_0}$ are defined on $B^u_{x_0}(r\epsilon_1)$,
%where $r\le 1$ is small enough
%that the conclusions of Proposition \ref{prop:unstabDistEst} hold along individual leaves comprising the stack. We will refer to such a stack as one on which ``distortion estimates
%hold on $W^u$-leaves".
%\end{rmk}
%%%%%%%%%%%%%%%%%%%%%%%%%%%%%%%%%%%%%
\subsection{The nonergodic case}\label{sec:nonergodic}

In Sections 3--5, up until this point, we have operated under the assumption
that $(f,\mu)$ is ergodic. We now discuss the extension of these results to the
nonergodic case.

The nonergodic case of the Multiplicative Ergodic Theorem reads as follows: 
Fix a measurable function $\lambda_\a > l_\a$. Then there is a measurable
integer-valuded function $r$ on $\Gamma$ such that at every $x \in \Gamma$, there are $r(x)$
Lyapunov exponents
$\lambda_1(x) > \lambda_2(x) > \cdots > \lambda_{r(x)}$
with $\lambda_{r(x)} > \lambda_\a(x)$ and an associated splitting
$\Bc = E_1(x) \oplus E_2(x) \oplus \cdots \oplus E_{r(x)}(x) \oplus F(x)$
with respect to which properties (a)--(d) in Theorem \ref{thm:MET} hold. Here
$\dim E_i(x) = m_i(x)$ where $m_i$ are measurable functions 
on $\{x \in \Gamma :  r(x) \ge i\}$. 

Next we define, as in Sect. 4.1, $E^\tau_x$ for $\tau=u,c, s$, and let 
$\lambda^+$ and let $\lambda^-$ be as before, except that they are now
measurable functions that need not be bounded away from $0$.
For $m,n \in \{0,1,2,\dots\}$ and $p,q \in \{ 1,2,\dots\}$, let
$$
\Gamma(m,n;p,q) = \left\{x\in \Gamma : \dim E^u_x=m, \ \dim E^c_x = n; \ 
\lambda^+(x) \ge \frac{1}{p}, \ \lambda^-(x) \le -\frac{1}{q} \right\}\ .
$$
Then each $\Gamma(m,n;p,q)$ is either empty, or it is $f$-invariant, and
$\Gamma = \cup_{m,n,p,q} \Gamma(m,n;p,q)$. For results that concern individual
$E_i$, it will be advantageous to further subdivide $\Gamma(m,n;p,q)$ according 
to the dimensions of these subspaces etc. We will focus here on
 the extension of the results in Sections 4 and 5 to the nonergodic case, for which
 the decomposition into $\Gamma(m,n;p,q)$ suffices. 

We claim -- and leave it for the reader to check -- that for these results, 
the proofs in Sections 4 and 5 go through {\it verbatim} provided that
one restricts to one $\Gamma(m,n;p,q)$ at a time, and allow 
the quantities $\lambda_0, \d_0$, 
hence $\lambda$, and $\d_1, \d_2$, hence the function $l$ and constant $\d'_1$,  
to depend on $(m,n;p,q)$. Notice that when we refer to Corollary \ref{cor:integrable} and Lemma \ref{lem:contSplitting}, 
the subspaces in question are $E^u$, the dimension of which is constant on 
$\Gamma(m,n;p,q)$ and the proofs there go through unchanged as well. Once this is checked, it will follow, for example, that for a fixed $(m,n;p,q)$, 
local unstable manifolds are defined for $\mu$-a.e.
$x \in \Gamma(m,n;p,q)$, stacks of unstable manifolds 
are well defined, and $\mu$-a.e. $x \in \Gamma(m,n;p,q)$ is contained 
in such a stack.

Obviously, the sets $\Gamma(m,n;p,q)$ are not pairwise disjoint.
If one wishes to work with pairwise disjoint $f$-invariant sets, the countable family
$$
\hat \Gamma(m,n;p,q) = \Gamma(m,n;p,q) \setminus \Gamma(m,n;p-1,q-1)
$$
is an alternative to ergodic decompositions.

%%%%%%%%%%%%%%%%%%%%%%%%%%%%%%%%%%%%%%%%%%%%%%%%%%%%%%%%%%%%%%%%%%%%%
%%%%%%%%%%%%%%%%%%%%%%%%%%%%%%%%%%%%%%%%%%%%%%%%%%%%%%%%%%%%%%%%%%%%%
%%%%%%%%%%%%%%%%%%%%%%%%%%%%%%%%%%%%%%%%%%%%%%%%%%%%%%%%%%%%%%%%%%%%%
\section{SRB Measures and the Entropy Formula}
%%%%%%%%%%%%%%%%%%%%%%%%%%%%%%%%%%%%%%%%%%%%%%%%%%%%%%%%%%%%%%%%%%%%%
%%%%%%%%%%%%%%%%%%%%%%%%%%%%%%%%%%%%%%%%%%%%%%%%%%%%%%%%%%%%%%%%%%%%%
%%%%%%%%%%%%%%%%%%%%%%%%%%%%%%%%%%%%%%%%%%%%%%%%%%%%%%%%%%%%%%%%%%%%%

Our proofs of Theorems 1 and 2 follow in outline \cite{ledstrel} and  \cite{led}, which contain analogous results for
diffeomorphisms of finite dimensional manifolds. These proofs are conceptually
as direct as can be: they relate $h_{\mu}(f)$, which measures the rate of information growth with respect to $\mu$, to the rate of volume growth on unstable 
manifolds -- under the assumption that conditional measures of $\mu$ on $W^u$-manifolds are in the same measure class as the induced volumes on these manifolds. 
Other proofs of the entropy formula in finite dimensions, such as \cite{pes, mane}, 
start from
invariant measures with densities on the entire phase space and are less suitable for adaptation
to infinite dimension.

In the last few sections we have laid the groundwork needed to extend the ideas in \cite{ledstrel} and  \cite{led} to Banach space settings. To make feasible the idea of
volume growth, we introduced a $d$-dimensional volume on unstable manifolds. To set the stage for conditional densities of invariant measures,
we proved distortion estimates of iterated determinants. Some technical work remains;
it is carried out in Sect. 6.2.
In Sects. 6.3 and 6.4, we verify carefully that all technical issues have been addressed. 

\medskip
{\it Hypotheses (H1)--(H4) are assumed throughout this section}; they will not 
be repeated in statements of results. Notice the addition
of (H4), the no zero exponents assumption, that was not present in most of the last 
two sections. (H5) will be introduced as needed.

%%%%%%%%%%%%%%%%%%%%%%%%%%%%%%%%%%%%%%%%%%%%%%%%%%%%%%%%%%%%%%%%%%%%%
\subsection{Equivalent definitions of SRB measures}
%%%%%%%%%%%%%%%%%%%%%%%%%%%%%%%%%%%%%%%%%%%%%%%%%%%%%%%%%%%%%%%%%%%%%

We begin with a formal definition of SRB measures for Banach space
mappings. This definition is relatively easy to state, and is equivalent to 
standard definitions used in finite dimensional hyperbolic theory.

Let $\Sc$ be a compact stack of local unstable manifolds as defined in Sect. 5.2, and let
$\xi^\Sc$ be the partition of $\Sc$ into unstable leaves.  By Lemma \ref{lem:homeoStack}, $\xi^\Sc$
is a measurable partition. Assuming $\mu(\Sc)>0$, we let 
$\{\mu_{\xi^{\Sc}(x)}\}_{x \in \Sc}$ denote the canonical disintegration 
of $\mu|_\Sc$ on elements of $\xi^{\Sc}$ (for details on canonical disintegrations, see \cite{rokhlin1949fundamental, rokhlin1967lectures} and \cite{chang1997conditioning}). Recall that $\nu_x$ is the induced volume 
on $W^u_x$. 

\begin{defn}\label{defn:SRB} We say that $\mu$ is an \emph{SRB measure}
of $f$ if
\begin{itemize}
\vspace{-4pt}
\item[(i)] $f$ has a strictly positive Lyapunov exponent $\mu$-a.e., and 
\vspace{-4pt}
\item[(ii)] for any stack $\Sc$ with $\mu(\Sc)>0$, $\mu_{\xi^{\Sc}(x)}$ is absolutely continuous
with respect to $\nu_x$, written $\mu_{\xi^{\Sc}(x)} \ll \nu_x$, for $\mu$-a.e. $x \in \Sc$.
\end{itemize}
\end{defn}

This definition was used in \cite{young1998statistical} (see also \cite{barreira}), and differs {\it a priori} from that used in \cite{ledstrel}, \cite{led}, \cite{ledyou1}, 
which we now recall. 
%We use $\{\mu_{\eta(x)}\}$ to denote the canonical disintegration 
%of $\mu$ on elements of a measurable partition $\eta$.

\begin{defn}
We say that a measurable partition $\eta$ is \emph{subordinate to the unstable foliation}
(abbreviated below as ``subordinate to $W^u$") if for $\mu$-a.e. $x$, we have 
\begin{itemize}
\vspace{-4pt}
\item[(i)] $\eta(x) \subset W^u_x$, 
\vspace{-4pt}
\item[(ii)] $\eta(x)$ contains a neighborhood of $x$ in $W^u_x$, and 
\vspace{-4pt}
\item[(iii)] $\eta(x) \subset f^N(W^u_{{\rm loc}, f^{-N} x})$ for some $N \in \N$ (depending on $x$).
\end{itemize}
\end{defn}

A proof of the existence of measurable partitions subordinate to the $W^u$ 
foliation was given in \cite{ledstrel}; we will provide a sketch in
Sect. 6.3. In \cite{ledstrel,led} and \cite{ledyou1}, SRB measures are defined
in terms of partitions subordinate to $W^u$.

\begin{lem}\label{lem:SRBequiv}
The following are equivalent.
\begin{enumerate}
\vspace{-4pt}
\item $\mu$ is an SRB measure in the sense of Definition \ref{defn:SRB}.
\vspace{-4pt}
\item There exists a partition $\eta$ subordinate to $W^u$ with
the property that $\mu_{\eta(x)} \ll \nu_x$ for $\mu$-a.e. $x$.
\vspace{-4pt}
\item Every partition $\eta$ subordinate to $W^u$ has the property
  $\mu_{\eta(x)} \ll \nu_x$ for $\mu$-a.e. $x$.
\end{enumerate}
\end{lem}

\noindent Lemma \ref{lem:SRBequiv} follows by the uniqueness of the canonical disintegration and the fact that $\mu$-a.e. $x$ is contained in an unstable stack; 
its proof is omitted. 

%%%%%%%%%%%%%%%%%%%%%%%%%%%%%%%%%%%%%%%%%%%%%%%%%%%%%%%%%%%%%%%
\subsection{Technical issues arising from noninvertibility}

%%%%%%%%%%%%%%%%%%%%%%%%%%%%%%%%%%%%%%%%%%%%%%%%%%%%%%%%%%%%%%%
Before proceeding to the proofs of our main results, we wish to dispose of 
some technical issues that do not present themselves in 
the setting of finite dimensional diffeomorphisms. These issues stem from 
the fact that in the course of our proofs, we will need
to deal with dynamics {\it outside of} $\As$. For example, to prove that $\mu$ is
an SRB measure given that the entropy formula holds, we will want to compare
$\mu$ to a measure with conditional densities on $W^u$-leaves, and the
construction of this measure will have to proceed without 
{\it a priori} knowledge that $W^u$-leaves are contained in $\As$. 

The material in this section holds under assumptions (H1)--(H4), with no
additional assumptions on $\mu$. Since the sets under consideration may not be
contained in $\As$, $f^{-1}$ is not necessarily defined, 
and certainly cannot be assumed to be continuous. As a consequence, properties 
that involve backward iterations, such as continuity of $y \mapsto E^u_y$, must be
treated with care, and discussions of $\mu$-typical behavior do not apply.

\medskip
%
%
%\vskip 1in
%The proofs and constructions used in the proofs of our main results heavily depend on the regularity of the map $f$ on stacks $\Sc$ as in Lemma \ref{lem:uStacks}. Unfortunately, as is the case in finite dimensions, it is \emph{a priori} possible for $\Sc \nsubseteq \As$, and so in our infinite-dimensional setting, it is possible that $f^{-1}|_{\Sc}$, although well-defined (being the union of local unstable leaves), may not be continuous. Our primary goal here is to alleviate this concern in Lemma \ref{lem:continuousInvert} below. Having shown this, we will use Lemma \ref{lem:continuousInvert} to handle several technical issues arising later on in Section 6. We can handle these technical issues 

For the rest of this subsection, we restrict ourselves to a component 
$\Gamma(m, n; p, q)$ for some $m, n, p, q \in \N$ (see Section 5.4). We let $l_0 \geq 1$ and fix an unstable stack  
\[
\Sc = \bigcup_{x \in \bar U} \exp_{x_0} (\graph \Theta(x)) \, ,
\]
where for each $x \in U(x_0, \e_0)$, we have $\Theta(x) : B^u_{x_0}(\d l_0^{-3}) \to E^{s}_{x_0}$; all notation is as in Lemma \ref{lem:uStacks}. 

%In what follows, we may need to shrink $\e_0$ further than what is necessary for Lemma \ref{lem:uStacks}; we only do so a finite number of times, and when we do, it will be in such a way that depends only on the uniformity set $\Gamma_{l_0}$ and the value $\d > 0$.

\begin{lem}\label{lem:continuousInvert} 
For all $n \in \N$, $f^{-n}$ is well-defined and continuous on $\Sc$.
\end{lem}
\begin{proof}
That $f^{-n}$ is well-defined on $\Sc$ follows from Theorem \ref{thm:unstabMfld} and the injectivity of $f$ on $\Bc$ (see (H1) in Section 1); the bulk of our work is in showing that $f^{-n}|_{\Sc}$ is continuous. 

\begin{cla}\label{cla:pullbackStack}
There exists $n_0 \in \N$ such that for all $n \geq n_0$ there is, for each $x \in \bar U$, a small neighborhood $V_{n, x}$ of $x$ in $\bar U$ such that (a) the set $f^{-n} \bar V_{n, x}$ obeys the criteria for possessing a compact stack $\Sc_{n, x}$ of unstable leaves as in Remark \ref{rmk:versatileStacks} (here $\bar V_{n, x}$ is the closure of $V_{n, x}$), and (b) we have that
\begin{equation} \label{containment}
f^n(\Sc_{n,x}) \supset \bigcup_{y \in \bar V_{n, x}} \exp_{x_0} (\graph \Theta(y))\ .
\end{equation}
\end{cla}

Assuming Claim \ref{cla:pullbackStack}, we let $n \geq n_0$ and let $\{V_{n, x_i}, i=1,2,\dots, q\}$ be a finite subcover of 
$\{V_{n, x}, x \in \bar U\}$. Then $\cup_{i=1}^q \Sc_{n,x_i}$ is compact, and since
$f^n|_{\cup_{i=1}^q \Sc_{ n, x_i}}$ is continuous, $f^{-n}$ is continuous on
$f^n(\cup_{i=1}^q \Sc_{n, x_i})$, hence $f^{-n}$ is continuous on $\Sc$. For $n < n_0$,
write $f^{-n} = f^{n_0-n} \circ f^{-n_0}$.

%observe first that by standard graph transform arguments
%there exist $\epsilon'>0$ and $\delta'<\d$ such that in the notation of Lemma \ref{lem:graphTprops},
%if $g: \tilde B^u_x(\d'l(x)^{-1}) \to E^{s}_x$ is such that Lip $g \le \frac{1}{10}$ and 
%$|g(0)|<\epsilon'$, then $\Psi_xg$ is well defined as a function on 
%$\tilde B^u_{fx} (\d l(f x)^{-1})$ with Lip$(\Psi_xg) \le \frac{1}{10}$. % and $|\Psi_x g(0)| < \e'$. 

It remains to prove Claim \ref{cla:pullbackStack}. Let us assume for the moment that we can find a neighborhood $V_{n, x}$ satisfying (a). By Remark \ref{rmk:versatileStacks}, the stack
 \[
 \Sc_{n,x} = \bigcup_{z \in f^{-n} \bar V_{n, x}} \exp_{f^{-n} x} \big( \graph g_z^{f^{-n} x}|_{D_{n, x}} \big) \, ,
 \]
where $D_{n, x} = B_{f^{-n} x}^u(\d (e^{n \d_2} l_0)^{-3})$, is well-defined and is comprised of continuously-varying unstable leaves.
%; by Lemma \ref{lem:homeoStack} we have that $\Sc_{n,x}$ is compact. 

To check \eqref{containment}, we relate the leaf through each point $z \in f^{-n} \bar V_{n, x}$ back to the leaf in the chart at $z$, using the considerations at the beginning of the proof of Lemma \ref{lem:uStacks} (see in particular \eqref{eq:graphingCoordinateMap}). To wit, one checks that
\begin{align*}
\exp_{f^{-n} x}  \graph g_z^{f^{-n} x}|_{D_{n,x}} &= \exp_z \graph g_{z}|_{\pi^u_z (\graph g_z^{f^{-n} x}|_{D_{n, x}} + f^{-n} x - z )} \\
& \supset \exp_z \graph g_{z}|_{\tilde B_z^u(\frac \d 6 (e^{n \d_2} l_0)^{-3})} \, ;
\end{align*}
this may require shrinking $V_{n, x}$ so that the diameter of $f^{-n} V_{n, x}$ is sufficiently small. Likewise, we relate the leaf in $\Sc$ through $f^n z$ with the corresponding unstable manifold leaf in the chart at $f^n z$: inspecting the proof of Lemma \ref{lem:uStacks}, we see that
\[
\exp_{x_0} \graph \Theta(f^n z)  \subset \exp_{f^n z} \graph g_{f^n z}|_{\tilde B^u_{f^n z}(4 \d l_0^{-2})} \, .
\]
%similarly, this may require shrinking $\e_0$ as in the definition of $\Sc$ in a way depending only on $\d$ and $l_0$.
So, to check \eqref{containment}, it suffices to show that for each $z \in f^{-n} \bar V_{n, x}$,
$$
\tilde f_{z}^n \graph \big(g_{z}|_{\tilde B_z^u(\frac \d 6 (e^{n \d_2} l_0)^{-3})} \big)
\supset \graph (g_{f^n z}|_{\tilde B^u_{f^n z}(4 \d l_0^{-2})}) \, .
$$
%
%\[
%f^n \circ \exp_{f^{-n} x} \big( \graph g_z^{f^{-n} x}|_{B^u_{f^{-n} x} (\d (e^{n \d_2} l_0)^{-3})} \big) \supset W^u_{\d, f^n z} \, ;
%%\exp_{x_0} \graph \Theta(f^n z) \, ;
%\]
This follows from a graph transform argument (see Section 5.1) for all $n \geq n_0$, where $n_0$ depends on $l_0$.
%, and may require taking $n$ sufficiently large (depending only on $l_0$). 

We now set about finding a neighborhood $V_{n, x}$ satisfying (a) in 
Claim \ref{cla:pullbackStack}. We will show 
that there exists $n_0 \in \N$ such that for all $n \geq n_0$ and any 
$x \in \bar U$, there is $V_{n, x}$ such that $|\pi^s_z|_{E^u_{z'}}| < \frac12$, and
$|\pi^s_z|_{E^s_{z'}}| <2$ for all $z, z' \in f^{-n} \bar V_{n, x}$; see (\ref{changechart}).
%
%
%we have
%\begin{align}\label{eq:pastSubspacesCloseEnough}
%|\pi^{s}_{z}|_{E^u_{z'}}| + \frac{3 l_0 e^{n \d_2} }{10} \, |\pi^u_{z}|_{E^{s}_{z'}}| < \frac12
%\end{align}
%(compare with \eqref{eq:subspacesCloseEnough} in Lemma \ref{lem:uStacks}). 
Having done so, and perhaps on shrinking $V_{n, x}$ further, it will follow from Remark \ref{rmk:versatileStacks} that the stack 
$\Sc_{n, x}$ as above satisfies the conclusions of Lemma \ref{lem:uStacks}.

To control $|\pi^s_z|_{E^u_{z'}}|$, observe that  $|\pi^s_z|_{E^u_{z'}}| \leq 
|\pi^s_z| d_H(E^u_z, E^u_{z'}) \leq 3 e^{n \d_2} l_0 d_H(E^u_z, E^u_{z'})$.
Since $f^{-n} \bar U \subset \Gamma_{e^{n \d_2} l_0}$ (Lemma \ref{lem:slowvary}), it follows from Lemma \ref{lem:contSplitting} that $z \mapsto E^u_z$ is 
continuous on $f^{-n} \bar U$. Thus we obtain the desired bound by choosing
$V_{n,x}$ sufficiently small. 
%
%
%
%
%\eqref{eq:pastSubspacesCloseEnough}, note that since $\bar U \subset \Gamma_{l_0}$, 
%we have $f^{-n} \bar U \subset \Gamma_{e^{n \d_2} l_0}$ (Lemma \ref{lem:slowvary}), and so it follows from Lemma \ref{lem:contSplitting} that $z \mapsto E^u_z$ is continuous across $z \in f^{-n} \bar U$. Thus for any $n \geq 1$,
%$x \in \bar U$, we can find a small enough neighborhood $V_{n, x} \subset \bar U$ of $x$ on which
%$d_H(E^u_z, E^u_{z'})$ is as small as we like for any $z, z' \in V_{n, x}$. As $|\pi^s_z|_{E^u_{z'}}| \leq 
%|\pi^s_z| d_H(E^u_z, E^u_{z'}) \leq 3 e^{n \d_2} l_0 d_H(E^u_z, E^u_{z'})$, this takes care of the first
%term in \eqref{eq:pastSubspacesCloseEnough}. Observe that this can be done for any $n \geq 1$.
To control $|\pi^s_z|_{E^s_{z'}}|$, we let
$v \in E^s_{z'}$ be a unit vector, and write $v=v^u + v^s \in E^u_{z} \oplus E^s_{z}$.
Since $|v^s| \le |v^u|+1$, it suffices to bound $|v^u|$. Now
\begin{equation} \label{Es}
|df^n_{z}(v) - df^n_{z'}(v)| \ge |df^n_{z}(v^u)| - |df^n_{z}(v^s) - df^n_{z'}(v)|\ .
\end{equation}
Choose $V_{n,x}$ small enough  that $\sup_{z, z' \in f^{-n} \bar V_{n, x} }|df^n_{z}-df^n_{z'}| < 1$, and $n_0$ large enough (depending only on $l_0$) 
so that for all $n \geq n_0$,
 $m(df^n_{f^{-n}y}|_{E^u_{f^{-n}y}}) \gtrsim e^{n \l}$ and $|df^n_{f^{-n}y}|_{E^s_{f^{-n}y}}| \lesssim e^{- n \l}$ for all $y \in \Gamma_{l_0}$. Then it follows from
(\ref{Es}) and $|v^s| \le |v^u|+1$ that $|v^u| \lesssim e^{-\l n}$.
\end{proof}

We now apply Lemma \ref{lem:continuousInvert} to obtain various facts about
$\Sc$.

\begin{lem}\label{lem:stackContUnstable}
The mappings $y \mapsto E^u_y$ are continuous on $f^{-q}\Sc$ for $q=0,1,2,
\dots$.
\end{lem}

Continuity of $E^u$ along individual unstable leaves follows from Theorem \ref{thm:unstabMfld},
but Lemma \ref{lem:stackContUnstable} asserts more than that: 
it asserts continuity across all the 
{\it different} leaves that comprise $f^{-q}\Sc$.
In the case that $\Sc \subset \Gamma$, this result follows from Lemma \ref{lem:contSplitting}, but we do not assume that. Still, we will follow the proof of Lemma \ref{lem:contSplitting} closely, supplying additional justification where needed. 

\begin{proof} We give the proof for $q=0$; it will be clear that this argument
will also prove the assertion for all $q \ge 1$.

Let $x, x^n \in \Sc$ be such that 
$x^n \to x$. Then $x_{-k}$ and 
$x^n_{-k}$ are defined for all $k \in \N$, and for each $k$, $x^n_{-k} \to x_{-k}$
as $n \to \infty$ by Lemma \ref{lem:continuousInvert}. We let $E(z)$ in Lemma \ref{lem:contSplitting}
be $E^u_z$.
Since every $z \in \Sc$ lies in $W^u_{\d, y}$ for some $y \in \bar U \subset \Gamma_{l_0}$, $(df^k_{z_{-k}}|_{E^u_{z_{-k}}})^{-1}, k=1,2,\dots$, 
have the uniform estimates required for $G^i_L$ in Lemma \ref{lem:contSplitting},
and Lemma \ref{char}(b) shows that these estimates uniquely characterize $E^u_z$.

To carry out the argument in Lemma \ref{lem:contSplitting}, we need to
show that along the backward orbit of $x$, there are closed subspaces
$\bar F(x)$ and $\bar F(x_{-k})$ such that (i) $\Bc_x = E^u_{x} \oplus \bar F(x)$
and $\Bc_{x_{-k}} = E^u_{x_{-k}} \oplus \bar F(x_{-k})$, (ii) $df_{x_{-k}} \bar F(x_{-k}) \subset \bar F(x_{-(k-1)})$, and
(iii) there exist arbitrarily large $k$ for which
$|df^k_{x_{-k}}|_{\bar F(x_{-k})}| \le e^{-\frac34 k\la}$.
Notice that this is needed for $x$ only, not for $x^n$.

Here is where the situations differ: For $z \in \Sc \setminus \Gamma$, 
there is no intrinsically defined $E^s_z$, hence we will have to construct a 
surrogate sequence of subspaces $\bar F(x)$ and
 $\bar F(x_{-k})$. Identifying
the tangent space $\Bc_x$ with $\Bc_y$ where $y$ is a point in 
$\bar U$ with the property that 
$x \in W^u_{\d, y}$, we let $\bar F(x) = E^s_y$, and claim that
 $\bar F(x_{-k})$ for $k=1,2,\dots$ are determined
by property (ii) in the last paragraph. To justify this claim, it is convenient to
work in the charts 
associated with the backwards orbit of $y$. Let  $\tilde \ $ denote corresponding 
objects in charts, so that $\tilde x_{-k} \in \tilde B_{f^{-k}y}( \d (e^{\d_2 k} l_0)^{-1})$, and
$\tilde F(\tilde x_{-k})$ is the subspace we seek etc.
For $\d$ small enough, we may assume, by Lemma \ref{prop:chartProps}, that 
$d(\tilde f_{f^{-1}y})_{\tilde x_{-1}}$ is sufficiently close to $d(\tilde f_{f^{-1}y})_0$
that the backward graph transform argument (Lemma \ref{lem:existsBackwardsTform}) can be applied to
 give a subspace $\tilde F(\tilde x_{-1})$ such that  
$d(\tilde f_{f^{-1}y})_{\tilde x_{-1}}(\tilde F(\tilde x_{-1})) \subset \tilde F(\tilde x)$.
Another application of the backward graph transform gives
$\tilde F(\tilde x_{-2})$ such that $d(\tilde f_{f^{-2}y})_{\tilde x_{-2}}(\tilde F(\tilde x_{-2})) \subset \tilde F(\tilde x_{-1})$, and so on. The situation only improves as we go
along, as $\tilde x_{-k} \to 0$ as $k \to \infty$. Moreover, since
$|d(\tilde f_{f^{-k}y})_{\tilde x_{-k}}|_{\tilde F(x_{-k})}|'_{f^{-k}y}
\leq e^{-\la} + \d$, property (iii) in the last paragraph follows.

We have now fully duplicated the conditions in Lemma \ref{lem:contSplitting},
and the arguments there carry over verbatim to give the continuity of $E^u$ on $\Sc$.
\end{proof}

\begin{rmk}
It is curious to compare Lemma \ref{lem:stackContUnstable} to the case of Anosov diffeomorphisms
of compact (finite dimensional) manifolds, where $y \mapsto E^u_y$ is known to be 
H\"older continuous (and generally not more than that) as $y$ moves in stable 
directions, with a H\"older exponent depending on the magnitude of greatest contraction
along stable directions \cite{hirsch1977invariant}.
As our map $f$ can and does contract arbitrarily
strongly in stable directions, one cannot expect any control on the modulus 
of continuity of $y \mapsto E^u_y$.
\end{rmk}

We conclude Sect. 6.2 with two applications of Lemmas \ref{lem:continuousInvert} 
and \ref{lem:stackContUnstable} to resolve some measurability issues we will 
encounter later on. Let $\Sc$ be as at the beginning of this section, and let 
$\xi = \xi^\Sc$ denote the (measurable) partition of $\Sc$ into unstable leaves.

\begin{lem}\label{lem:nuMeasurable}
For any Borel measurable set $B \subset \Sc$, the map
$x \mapsto \nu_x(B \cap \xi(x)) $ is measurable.
\end{lem}

\begin{proof} Following Lemma \ref{lem:homeoStack}, we have a homeomorphism 
$\Psi : B^u_{x_0}(\d l_0^{-3}) \times \Sigma \to \Sc$ where 
$\Sigma = \exp^{-1}_{x_0}\big(\Sc\big) \cap E^{s}_{x_0}$, so 
it suffices  to prove the corresponding result on 
$B^u_{x_0}(\d l_0^{-3}) \times \Sigma$ for the family of measures
$\{\hat \nu_\sigma, \sigma \in \Sigma\}$ that are carried by $\Psi$ to $\nu_x$.
More precisely, if $\Psi( B^u_{x_0}(\d l_0^{-3}) \times \{\sigma\}) = \xi(x)$, 
then $\hat \nu_\sigma$ is the measure defined by
$\Psi_*(\hat \nu_\sigma) = \nu_x$. Let $m$ denote
the volume induced on the finite dimensional space $B^u_{x_0}$, and on
$ B^u_{x_0} \times \{\sigma\}$ via its natural identification with $ B^u_{x_0}$. Then 
$\frac{d \hat \nu_\sigma}{dm}(u) = \det(\Id + dg_\sigma(u))$ where $\det$ is 
with respect to $m$ and the induced volume on $E^u_{\Psi(u, \sigma)}$.
Observe that $(u,\sigma) \mapsto E^u_{\Psi(u, \sigma)}$
is continuous by Lemma \ref{lem:stackContUnstable}.
From this and from properties of the $\det$-function
(Proposition \ref{prop:detReg}), we deduce that the mapping $(u,\sigma) \mapsto 
\tau(u,\sigma) := \frac{d \hat \nu_\sigma}{dm}(u)$ is continuous. The desired result follows from 
the continuity of $\tau$.
\end{proof}

Let $J^u(x) := \det(df_x|E^u_x)$.
It follows from Proposition \ref{prop:unstabDistEst} (b) that on each leaf $\xi (x)$,
the function
\[
z \mapsto \Delta(x,z) := \prod_{i = 1}^\infty \frac{ J_u(f^{-i} x)}{  J_u(f^{-i }z)} \ ,
\qquad z \in \xi(x)\ ,
\]
is Lipchitz-continuous and bounded from above and below. Let $q: \Sc \to \mathbb R$
be given by
\[
q(z) := \frac{\Delta(x, z)}{\int_{\xi(x)} \Delta(x,z) d \nu_x(z)} \, 
\]
where $z \in \xi(x)$. Observe that $q$ is well defined and independent of the choice
of $x$. 

\begin{lem}\label{lem:qCont}
The function $q$ is continuous.
\end{lem}

\begin{proof} First we claim that for every $n \in \N$, the function
$z \mapsto \det(df^n_{f^{-n}z} |E^u_{f^{-n}z})$ is continuous. That is true 
because
(i) $f^{-n}|_\Sc$ is continuous (Lemma \ref{lem:continuousInvert}), 
(ii) $y \mapsto E^u_y$ is continuous on $f^{-n}\Sc$ (Lemma \ref{lem:stackContUnstable}),
and (iii) $y \mapsto \det(df^n_y |E^u_y)$ is continuous on $f^{-n}\Sc$ 
(Proposition \ref{prop:detReg} together with (ii)). 

To prove the continuity of $q$, we fix, for each $z \in \Sc$, a reference point $\bar \sigma (z)$ defined
to be the unique point in $\xi(z) \cap \exp_{x_0}(\Sigma)$ where $\Sigma$ is
as in the proof of Lemma \ref{lem:nuMeasurable}. Define 
$$z \mapsto \bar \Delta_n(z) : = 
\Delta_n (\bar \sigma (z), z) = 
\prod_{i = 1}^n  \frac{J_u(f^{-i} \bar \sigma (z))}{J_u(f^{-i }z)}\ .
$$
Then $z \mapsto \bar \Delta_n(z)$ is continuous by the argument above and the continuity of $z \mapsto \bar \sigma (z)$. 
By Proposition \ref{prop:unstabDistEst}, the sequence $\bar \Delta_n$ 
converges uniformly to $\bar \Delta$ where $\bar \Delta (z) = \Delta (\bar \sigma (z),z)$. Thus $\bar \Delta$ is continuous on $\Sc$.

It remains to show that $z \mapsto \int \bar \Delta d\nu_z$ is continuous, 
and that follows from the continuity of $\bar \Delta$ and arguments given in the proof of Lemma \ref{lem:nuMeasurable}.
\end{proof}

%%%%%%%%%%%%%%%%%%%%%%%%%%%%%%%%%%%%%%%%%%%%%%%%%%%%%%%%%%%%%%%%%%%%%
\subsection{Proof of entropy formula for maps with SRB measures}
%%%%%%%%%%%%%%%%%%%%%%%%%%%%%%%%%%%%%%%%%%%%%%%%%%%%%%%%%%%%%%%%%%%%%

Recall that the distinct Lyapunov exponents of $(f,\mu)$ are denoted by $\lambda_i$ with
multiplicity $m_i$. Let $h_\mu(f)$ denote the entropy of $f$ with respect to
$\mu$, and let $a^+ := \max\{a,0\}$. This section contains the proof of Theorem 1, which
we state for the reader's convenience:

\bigskip \noindent
{\bf Theorem 1.} {\it Assume that $\mu$ is an SRB measure of $f$. Then
$$
h_{\mu}(f) = \int \sum_i m_i \l_i^+ \ d\mu\ .$$
}

Below we recall in outline the proof in \cite{ledstrel} (referring the reader to \cite{ledstrel} for detail), and point out the
modifications needed to make the argument in \cite{ledstrel} work in the present Banach space
setting. We divide the proof into
two parts:

\bigskip \noindent
{\bf (A) Construction of partitions subordinate to $W^u$.} First, some notation:
For a partition $\a$, we let $\a(x)$ denote the atom of $\a$ containing $x$. For two partitions $\a, \beta$, we write $\a \leq \beta$ if $\beta$ is a refinement of $\a$, 
and let
$\a \vee \b= \{A \cap B : A \in \a, B \in \b\}$, and $f^{-1} \a = \{f^{-1}A, A \in \a\}$. 
Finally, we say a partition $\a$ is \emph{decreasing} if $\a \leq f^{-1} \a$.

The following is the analog of Proposition 3.1 in \cite{ledstrel}.

\begin{prop} \label{eta} Assuming that $(f,\mu)$ has a positive Lyapunov
exponent $\mu$-a.e., there is a stack of unstable manifolds $\Sc$ with $\mu(\Sc)>0$
and a measurable partition $\eta$ on $\tilde \Sc := \cup_{n \ge 0} f^n \Sc$
with the following properties:
\begin{itemize}
\vspace{-4pt}
\item[(a)] $\eta$ is subordinate to $W^u$,
\vspace{-4pt}
\item[(b)] it is decreasing, and
\vspace{-4pt}
\item[(c)] for any Borel set $B \subset \As$, the function $x \mapsto \nu_x(\eta(x) \cap B)$
is finite-valued and measurable. 
\end{itemize}
%Let $\eta$ be the partition obtained by restricting $\eta$ to $\As$. Then $\eta$ has
%all the properties above except for condition (ii) in Definition 6.2: one does not
%necessarily have that for $\mu$-a.e. $x$, $\eta(x)$ contains a neighborhood of $x$ in $W^u_x$.
\end{prop}

We remark that if $\eta(x)$ is to contain a neighborhood of $x$ in $W^u_x$
for $\mu$-a.e. $x$ (part of the definition of being subordinate to $W^u$), 
then we cannot assume $\eta(x) \subset \As$. This condition is used in the proof 
of Lemma \ref{lem:partialconverse}, where we have to work with the set $\tilde \Sc$.

%
%
%and $\tilde \eta(x)$ can be different {\it a priori}:
%For $\mu$-a.e. $x$, $\tilde \eta(x) \supset W^u_{\epsilon, x}$ for some 
%$\epsilon>0$ and it need not be contained in $\As$, while $\eta(x) \subset \As$
%and may not contain $W^u_{\epsilon, x}$ for any $\epsilon>0$.
%To properly carry out the proofs of Theorem 2, 
%we have found it necessary to work with both 
%$\tilde \eta$ and $\eta$, a distinction not made in \cite{ledstrel} or  \cite{led}. 

\begin{proof}[Sketch of proof] Using the notation in Sect. 5.2, we first
construct a stack. The stack in the statement of 
this proposition will be of the form $\Sc = \Sc_r$ where
\begin{equation} \label{stack}
\Sc_r := \bigcup_{y \in \bar U} \exp_{x_0} (\graph \big( \Theta(y)|_{B_{x_0}^u(r \d l_0^{-3})} \big)) \, ,
%\pd_r := \bigcup_{y \in U} \{\Theta(y)(u) : u \in E^u_{x_0}, |u| = r \e_1 \} \; .
\end{equation}
$x_0 \in \As$ is a point, and $r \in (0,1)$ is a number to be determined. We choose these so that
$\mu(\Sc_r)>0$ for any choice of $r$, and let $\tilde \Sc = \cup_{n \ge 0} f^n \Sc$. 

The partition $\eta$ on $\tilde \Sc$ is constructed as follows:
Let $\xi$ be the partition of $\Sc$ into unstable leaves. For $k=0,1,2, \dots$, 
let $\xi_k = \{f^k(W), W \in \xi\} \cup \{\tilde \Sc \setminus f^k(\Sc)\}$, and let 
$\eta = \bigvee_{k=0}^\infty \xi_k$. Since $x \in f^k(W)$ if and only if $fx \in f^{k+1}(W)$,
it follows immediately that $\eta$ is decreasing. 

Item (a) also follows automatically from this construction except for the requirement that
$\eta(x)$ contains a neighborhood of $x$ in $W^u_x$ for $\mu$-a.e. $x$.
This is done by choosing $r$ judiciously. Let 
$$
\partial \Sc_r := \bigcup_{y \in \bar U} \exp_{x_0} \big( \{\Theta(y)(u) : u \in E^u_{x_0}, |u| = r \d l_0^{-3} \} \big) \, .
$$
Since $\partial (\eta(x)) \subset \cup_{k=0}^\infty f^k(\partial \Sc)$, to 
guarantee $W^u_{\epsilon(x),x} \subset \eta(x)$, it suffices to have
$f^{-k}W^u_{\epsilon(x),x} \cap \partial \Sc_r = \emptyset$ for all $k \ge 0$. 
We choose $r$ so that
$\mu(\partial \Sc_r)=0$ and $\epsilon(x)>0$ for $\mu$-a.e. $x$ by a Borel-Cantelli type argument. See \cite{ledstrel} for details.

Item (c) follows from Lemma \ref{lem:nuMeasurable} together with standard approximation arguments (to go from $\xi$ to $\eta$). 
\end{proof}

If $(f, \mu)$ is nonergodic, it may happen that $\mu(\tilde \Sc) <1$.
It can be shown that at most a countable number of (disjoint) sets of the type
$\tilde \Sc$ will cover a full $\mu$-measure set.

\medskip \noindent
{\bf (B) Entropy computation.} The following are the main points in the rest of 
the proof in \cite{ledstrel}. We first list them (as they appear in \cite{ledstrel}) before commenting on modifications needed:

\smallskip \noindent
1. Since Ruelle's Inequality \cite{ruelle} states that 
\begin{equation} \label{ruelle}
h_{\mu'}(f) \le \int \sum_i m_i \lambda_i^+ d\mu'
\end{equation}
for any Borel probability invariant measure $\mu'$, to prove the entropy formula
 it suffices to show that the 
reverse inequality holds for SRB measures.
In particular, it suffices to show that if $\mu$ is an SRB measure,
then 
\begin{equation} \label{relation}
H(f^{-1}\eta| \eta) = h_\mu(f,\eta) = \int \sum_i m_i \lambda_i^+ d\mu
= \int \log J^u d\mu  
\end{equation}
where $\eta$ is the partition constructed in Part (A) and $J^u(x) := 
\det(df_x|E^u(x))$.

\smallskip \noindent
2. Let $\nu$ be the $\sigma$-finite measure with the property that
for any Borel subset $K$, 
$$
\nu (K) = \int \nu_x(\eta(x) \cap K) d \mu (x)\ .
$$
By assumption, $\mu \ll \nu$. Let $\rho = \frac{d\mu}{d\nu}$, and let $\mu_{\eta(x)}$ 
denote the disintegration of $\mu$ on partition elements of $\eta$. Then
$$
\rho = \frac{d\mu_{\eta(x)}}{d\nu_x} \qquad \nu_x-a.e. \mbox{ on } 
\eta(x) \ \ \mbox{ for } \mu - a.e.\  x\ .
$$

\smallskip \noindent
3. The main computation is the following transformation rule for $\rho$:
\begin{lem}\label{lem:transRule} For $\mu$-a.e. $x$ and $\nu_x$-a.e. $z \in \eta(x)$,
$$
\rho(z) = \rho(fz) \cdot J_u(z) \cdot \mu_{\eta(x)} ((f^{-1} \eta)(x))\ .
$$
\end{lem}

\smallskip \noindent
4. From Lemma \ref{lem:transRule}, one deduces easily that the information function $I(f^{-1}\eta|\eta)$
satisfies
\begin{equation} \label{cocycle}
I(f^{-1}\eta|\eta)(x) = \log J^u(x) + \log \frac{\rho(fx)}{\rho(x)}
\end{equation}
for $\mu$-almost every $x$. As $I(f^{-1}\eta|\eta) \ge 0$ and $\log J^u \in L^1(\mu)$, it follows that
$\log^- \frac{\rho \circ f}{\rho} \in L^1(\mu)$. A general measure-theoretic
lemma then gives $\int \log \frac{\rho \circ f}{\rho} d\mu =0$. Integrating (\ref{cocycle})
gives (\ref{relation}).

\medskip
We now comment on the modifications needed for the arguments above to carry
over to Banach space mappings. For Banach space mappings, the
analog of \eqref{ruelle} is proved in \cite{thieu}, so here as well, it suffices to prove
(\ref{relation}). But to make sense of the last equality in (\ref{relation}),
one needs to first introduce a notion of volume on finite dimensional subspaces,
so that $\det(\cdot | \cdot)$ is defined; this is done in Section 2, and the last equality 
in (\ref{relation}), which relates Lyapunov exponents to volume growth, is proved 
both in \cite{blumenthal} and in Proposition \ref{prop:volGrowthMET}.

With regard to Item 2, that $\nu$ so defined is a measure follows from part (c)
of Proposition \ref{eta}. The characterization of $\rho$ given is a purely 
measure-theoretic fact. 

The proof of Lemma \ref{lem:transRule} uses only (i) the change of variables formula
for induced measures on finite dimensional manifolds and (ii) the invariance of $\mu$,
more precisely that $\mu (f^{-1}K)=\mu(K)$ for a countable sequence of Borel sets $K$. 
For the measures $\nu_x$, (i) is proved in (\ref{changevariable}). (ii) is not an issue 
for us as $\mu$ is supported on the compact metric space $\As$.

The last item is also a purely measure-theoretic fact. 

This completes the proof of Theorem 1.

%%%%%%%%%%%%%%%%%%%%%%%%%%%%%%%%%%%%%%%
\subsection{Entropy formula implies SRB measure}

We first prove Theorem 2 under the assumption that $h_\mu(f)=
h_\mu(f, \eta)$ for a partition $\eta$ of the type constructed in Proposition \ref{eta}, leaving
the justification of this assumption for later. 

\begin{lem}\label{lem:partialconverse} Given $(f, \mu)$ with a positive Lyapunov
exponent $\mu$-a.e., let $\eta$ be as in 
Proposition \ref{eta}. If 
$$
h_\mu (f,\eta) = \int \log J^u d\mu\ ,
$$
 then $\mu$ is an SRB measure whose
densities on unstable manifolds satisfy (\ref{density}).
\end{lem}

\begin{proof} Our proof follows \cite{led} in outline. Here it is essential
that the elements of $\eta$ contain open
subsets of $W^u_x$ for $\mu$-a.e. $x$.
Let $\Sc, \xi$ and $\tilde \Sc$ be as in Proposition \ref{eta}. We discuss 
the ergodic case, dividing it into two main steps.

\medskip \noindent
{\bf (A) Construction of a candidate SRB measure $\vartheta$.} 
As noted in Sect. 6.2, the function $z \mapsto \Delta(x,z)$ is 
%
%It follows from Proposition \ref{prop:unstabDistEst} (b) that for each leaf $\xi (x)$
%in $\Sc = \Sc_r$, the stack from which $\eta$ is constructed, the function
%\[
%z \mapsto \Delta(x,z) = \prod_{i = 1}^\infty \frac{ J_u(f^{-i} x)}{  J_u(f^{-i }z)} \ ,
%\qquad z \in \xi(x)\ ,
%\]
Lipchitz-continuous and bounded from above and below on $\xi(x)$. Since for $\mu$-a.e. $x$, there exists $n \ge 0$
such that $f^{-n}(\eta(x))$ is contained in a leaf of $\Sc$, and $f^n$ 
restricted to each leaf is a $C^2$ embedding
(by Hypothesis (H1)), the statement above  holds (with nonuniform Lipschitz bounds)
for all $z \in \eta(x)$ for $\mu$-a.e. $x$. This together with the fact that
$\nu_x(\eta(x))>0$ for $\mu$-a.e. $x$ implies that
\[
p(z) := \frac{\Delta(x, z)}{\int_{\eta(x)} \Delta(x,z) d \nu_x(z)}\ , \qquad z \in \eta(x)\ ,
\]
is well defined for $\mu$-a.e. $x$. 

We seek to define a probability measure $\vartheta$ on Borel subsets of 
$\tilde \Sc = \cup_{n \geq 0} f^n \Sc$ by letting
\[
\vartheta(K) = \int \left( \int_{\eta(x) \cap K} p(z) d\nu_x(z) \right) d \mu(x) 
\]
for $K \subset \tilde \Sc$. 
%Notice that {\it a priori} for $x \in \Gamma$, 
%$W^u_{{\rm loc}, x}$ need not be contained in $\As$, and if we want to have
%$\int_{\eta(x)} p(z) d\nu_x(z)=1$, then we must allow $\hat \nu$ to be supported
%not on $\As$ but on $\tilde \Sc$, which is a countable union of 
%compact sets.
%
That is to say, we want $\vartheta$ and $\mu$ to project to the same measure on 
the quotient space $\tilde \Sc/\eta$, and we want the conditional measures of
$\vartheta$ on elements of $\eta$ to be given
by the (normalized) densities $p(\cdot)$ (We note that
$\tilde \Sc$ is a complete, separable metric space,
 being a countable union of compact sets in $\Bc$; the existence of canonical disintegrations in this setting is proved in, e.g., \cite{chang1997conditioning}). To prove that $\vartheta$ is a {\it bona fide} measure, we need to prove the measurability of $p$,
which can be deduced, via standard arguments, from the measurability of the function
$$
q(z) := \frac{\Delta(x, z)}{\int_{\xi(x)} \Delta(x,z) d \nu_x(z)}\ , \qquad z \in \Sc\ .
$$
This involves studying backward iterates not just along individual $W^u_{\rm loc}$-leaves
but {\it across} the leaves that comprise $\Sc$, at points that are not necessarily
$\mu$-typical. We have treated these issues in Sect. 6.2; the continuity of $q$ is
proved in Lemma \ref{lem:qCont}.

\medskip \noindent
{\bf (B) Proof of $\vartheta = \mu$.}
This part of the argument is identical to that in \cite{led}; we recall it for completeness.
 For any $n \in \N$, we have
\begin{eqnarray*}\label{eq:nuInfoEst}
\int - \log \vartheta_{\eta(x)} ((f^{-n} \eta)(x)) d \mu(x) & = & \int \log \det(df^n_x|_{E^u(x)}) d \mu\\
&=& H(f^{-n} \eta | \eta) \ = \ \int -  \log \mu_{\eta(x)} ((f^{-n} \eta)(x)) d \mu(x)\ .
\end{eqnarray*}
The first equality is by the change of variables formula (the same computation
as in Items 3 and 4 in Part (B) of Sect. 6.3), the second is
by the main assumption in Lemma \ref{lem:partialconverse}, iterated $n$ times, and the third is
the definition of entropy. 

We introduce the $\tilde \Sc/(f^{-n}\eta)$-measurable
function  
$$
\phi_n (x) = \frac{\vartheta_{\eta(x)} ((f^{-n} \eta)(x))}{\mu_{\eta(x)} ((f^{-n} \eta)(x))} \, ,
$$
which is well-defined $\mu$-almost surely. Let $\vartheta^{(n)}$ and $\mu^{(n)}$ 
denote the restriction of $\vartheta$ and $\mu$ to $\Bs_{f^{-n} \eta}$, the $\sigma$-algebra of measurable subsets that are unions of atoms in $f^{-n} \eta$, and
decompose $\vartheta^{(n)} = \vartheta^{(n)}_\mu + \vartheta^{(n)}_{\perp}$, where 
$\vartheta^{(n)}_\mu \ll \mu^{(n)}$ and $\vartheta^{(n)}_\perp$ is mutually singular with 
$\mu^{(n)}$. Observe that $\vartheta^{(n)}_{\perp}$ can be strictly positive
(this happens if a positive $\mu$-measure set of $x$ has the property that $\eta(x)$ contains one or more elements of $f^{-n}\eta$ with $\mu_{\eta(x)}$-measure zero), 
while $\phi_n = d\vartheta^{(n)}_\mu/d\mu^{(n)}$. Thus $\int \phi_n d\mu \leq 1$.

At the same time, it follows from the string of equalities at the beginning of the proof that
 $\int \log \phi_n d\mu =0$. Now Jensen's inequality tells us that
$$
\int \log \phi_n d \mu \leq \log \int \phi_n d \mu  \, ,
$$
with equality holding iff $\phi_n \equiv $ constant $\mu$-a.e. So $\phi_n \equiv 1$, 
from which it  follows that $\mu$ and $\vartheta$ coincide on $\Bs_{f^{-n} \eta}$. The conclusion now follows from the fact that $ f^{-n} \eta \nearrow \varepsilon$, the partition into points.
\end{proof}

\smallskip
It remains to address the issue of whether or not the (uncountable) partition $\eta$  
captures all the entropy of $f$, i.e., whether $h_\mu (f) = h_\mu(f,\eta)$. To get a
handle on this, we seek a measurable partition $\mathcal P$ with the properties that
$H_\mu(\mathcal P)<\infty$ and $\eta \le \mathcal P^+ := 
\vee_{n = 0}^{\infty} f^j \mathcal P$. 
Since elements of $\eta$ are contained $W^u$-leaves, and 
$\mathcal P^+(x) = \{y \in \Bc : f^{-n}y \in \mathcal P(f^{-n}x)$ for all $n \ge 0\}$,
the next lemma, which follows immediately
from the characterization of local  unstable manifolds in Lemma \ref{char}, is relevant:

\begin{lem}\label{lem:partitionInsideWuloc} Suppose $\mathcal P$ is a partition with the property that for $\mu$-a.e. $x$ and
every $n \in \N$, 
$$
|f^{-n}y - f^{-n}x|'_{f^{-n}x} \le \d l(f^{-n}x)^{-1} \qquad \mbox{ for all } \ y \in \mathcal P^+(x)\ .
$$
Then $ \mathcal P^+(x) \subset W^u_{\d, x}$.
\end{lem}

Thus the problem is reduced to finding a finite entropy partition $\mathcal P$ with the
property in the lemma above. 
In \cite{led}, such a partition was constructed by appealing to a lemma due to Ma\~n\'e
\cite{mane}, and here lies another difference between finite and infinite dimensions: 
the lemma in \cite{mane} uses the finite dimensionality of the ambient manifold. 
Our next lemma contains a slight strengthening of this result that is adequate for
our purposes; see \cite{linshu1} for a similar result.% for Hilbert space maps. 

\begin{lem}[following Lemma 2 in \cite{mane}]\label{lem:Mane}
Let $Z$ be a compact metric space with box-counting dimension $=\d < \infty$.
Let $T : Z \to Z$ be a homeomorphism, let $m$ be a Borel invariant probability on $Z$,
and let $\rho : Z \to (0,1)$ be a measurable function for which $\log \rho \in L^1(m)$. Then, there exists a countable partition $\mathcal P$ of finite entropy such that for $m$-almost all $x \in Z$, 
\[
\mathcal P(x) \subset B(x, \rho(x)) \, ,
\]
where $B(x, r)$ is the ball of radius $r$ centered at $x$.
\end{lem}

\begin{proof}[Proof Sketch for Lemma \ref{lem:Mane}]
Define $U_n = \{x \in Z : e^{- (n + 1)} < \rho(x) \leq e^{- n}\}$ for $n \geq 1$. Since $\log \rho$ is integrable, we have that $\sum_{n = 1}^{\infty} n \, m(U_n) < \infty$. This implies (see Lemma 1 of \cite{mane}) that 
\[
\sum_{n = 1}^{\infty} - m(U_n) \log m(U_n) < \infty \, .
\]
By the definition of box-counting dimension, there exists $C>0$ be such that for any $r>0$, there is a finite cover of $Z$ by balls of radius $r$
with cardinality $\leq C r^{- (\d + 1)}$. It follows that there is a finite partition $\mathcal P_r$ of $Z$ of cardinality $\leq C r^{- (\d + 1)}$, each element of which is contained in one of these balls. Writing $r_n = e^{-(n + 1)}$ for each $n \in \N$, we fix such a partition $\mathcal P_{r_n}$. 
The desired partition $\mathcal P$ is now defined as follows: Elements of $\mathcal P$ are of the form $A \cap U_n$, $n \geq 0$, with $A \in \mathcal P_{r_n}$. One then estimates
\[
\sum_{P \in \mathcal P, P \subset U_n} - m (P) \log m(P) \leq m(U_n) \big( \log |\mathcal P_{r_n}| - \log m(U_n) \big) 
\]
for each $n \geq 0$, where the cardinality $|\mathcal P_{r_n}|$ is $ \leq C r_n^{- (\d + 1)}$; this decay is sufficient to show that $H_m(\mathcal P) < \infty$.
\end{proof}

The discussion above serves to motivate
 the finite box-counting assumption in Theorem 2, a complete
statement of which is as follows:

\medskip \noindent
{\bf Theorem 2.} {\it In addition to (H1)--(H4), 
we assume (H5), i.e., that the set $\As$ has finite box-counting dimension. If $\lambda_1>0$
and $(f,\mu)$ satisfies the Entropy Formula 
\[
h_{\mu}(f) = \int \sum_i m_i \l_i^+ \ d\mu\ ,
\]
then $\mu$ is an SRB measure.}

\medskip
Theorem 2 is implied by Lemma \ref{lem:entropySubsumed} below and Lemma \ref{lem:partialconverse}.

\begin{lem}\label{lem:entropySubsumed} Assume (H1)--(H5), and that $\lambda_1>0$ $\mu$-a.e. Let
$\eta$ be the partition  in Lemma 6.9. Then
$$
h_\mu(f) = h_\mu(f,\eta)\ .
$$
\end{lem}

\begin{proof}[Sketch of proof.] This part of our proof is identical to that in \cite{led} (see also \cite{ledyou1}). 
For completeness we outline the proof of the ergodic case. 

\medskip \noindent
{\bf (A) Construction of an auxiliary partition $\mathcal P$.} The aim of this step is
to produce a partition
$\mathcal P$ of $\As$ with the property in Lemma \ref{lem:partitionInsideWuloc} and with $H_\mu(\mathcal P)<\infty$.
Fix $l_0 > 1$ such that $\mu (\Gamma_{l_0}) > 0$, and define $N : \Gamma_{l_0} \to \N$ to be the first return time to $\Gamma_{l_0}$. Extend $N$ to all of $\As$ by setting $N|_{\As \setminus \Gamma_{l_0}} = 0$. We apply Lemma \ref{lem:Mane} to the function $\rho(x) = \d l_0^{-2} (e^{2 \d_2} K)^{- N(x)} $ for any $\d \le \d'_1$,  (where $K$ is an upper bound on $|df|$ for $\{y \in \Bc : d(y, \As) \leq r_0\}$; see the discussion preceding Lemma \ref{prop:chartProps}), for which $\log \rho$ is in $L^1(\mu)$ since $\int_{\As} N(x) d \mu(x) = \big( \mu(\Gamma_{l_0} ) \big)^{-1} < \infty$. It is straightforward to check that the $\mathcal P$ given by 
Lemma \ref{lem:Mane} has the desired properties.

\medskip \noindent
{\bf (B) Proof of $h_\mu(f) = h_\mu(f, \eta)$.} Given a small $\varepsilon>0$, we let 
$\mathcal Q = \mathcal P \vee \{\Sc \cap \As, \As \setminus \Sc\} \vee \mathcal P_0$
where $\Sc$ is the stack used in the construction of $\eta$ and $\mathcal P_0$
is a finite partition chosen so that $h_\mu(f, \mathcal Q)>h_\mu(f) - \varepsilon$. Let
$\mathcal Q^+ = \vee_{n = 0}^{\infty} f^j \mathcal Q$, and check that by construction,
$\eta(x) \supset \mathcal Q^+(x)$ for $\mu$-a.e. $x$. Then
\begin{equation}  \label{converse}
h_{\mu}(f, \eta)  = \frac1n H_{\mu}(f^{-n} \eta | \eta)  \geq \frac{1}{n} H_{\mu}(\vee_{j = 0}^n f^{-j} \mathcal Q | \eta) - \frac1n H_{\mu}(\vee_{j = 0}^n f^{-j} \mathcal Q | f^{-n} \eta) \, .
\end{equation}
Since $\mathcal Q^+ \geq \eta$, the first term on the right side of (\ref{converse}) is
$$\geq \frac1n H_{\mu}(\vee_{j = 0}^n f^{-j} \mathcal Q | \mathcal Q^+) \geq h_{\mu}(f, \mathcal Q) \geq h_{\mu}(f) - \varepsilon\ .
$$
while the second term in \eqref{converse} can be shown to be $<\varepsilon$ for large $n$ since 
modulo sets of $\mu$-measure $0$, $\vee_{j = 0}^{\infty} f^{-j} \eta$ partitions
$\As$ into points. 
\end{proof}

%%%%%%%%%%%%%%%%%%%%%%%%%%%%%%%%%%%%
\subsection{Corollaries}

We finish with the following corollaries to our main results.

\begin{thm}\label{thm:density} Let $\mu$ be an SRB measure of $f$. Let $\eta$
be given by Proposition \ref{eta}, and let $\rho$ be the densities with respect to 
$\nu_x$ of the conditional measures 
of $\mu$ on elements of $\eta$. Then $y \mapsto \rho(y)$ is Lipschitz on
each element of $\eta$, 
and  for $\mu$-a.e. $x$
\begin{equation} \label{density}
\frac{\rho(y)}{\rho(z)} =  \prod_{i=1}^\infty \frac{ \det(df_{f^{-i}z}|E^u_{f^{-i}z})}
{ \det(df_{f^{-i}y}|E^u_{f^{-i}y})} \quad \mbox{ for } 
\nu_x {\rm -a.e. } \ y,z \in \eta(x)\ .
\end{equation}
The conclusion above holds also if $\Sc$ is a stack of local unstable manifolds with 
$\mu(\Sc)>0$ and $\eta$ is replaced by $\xi$, the partition of $\Sc$ into unstable leaves. 
\end{thm}

\begin{proof} In the proof of Theorem 1, we showed that
$h_\mu (f,\eta) = \int \log J^u d\mu$. The form of the densities comes
from the proof of Lemma \ref{lem:partialconverse}.
\end{proof}

\begin{cor} Let $\mu$ be an SRB measure of $f$. Then $W^u_x \subset \operatorname{supp}(\mu)
\subset \As$ for $\mu$-a.e. $x$.
\end{cor}

\begin{proof} This is because $\rho >0$ $\nu_x$-a.e. on $W^u_x$ for $\mu$-a.e. $x$
by Theorem \ref{thm:density}.
\end{proof}

%%%%%%%%%%%%%%%%%%%%%%%%%%%%%%%%%%%%%%%%%%%%%%%
%%%%%%%%%%%%%%%%%%%%%%%%%%%%%%%%%%%
%%%%%%%%%%%%%%%%%%%%%%%%%%%%%%%%
\section*{Appendix}

Recall that the hypothesis of Lemma \ref{lem:openCond} are
$$
(*) \qquad E, E', F \in \Gc(\Bc)\ , \quad \Bc = E \oplus F\ , \quad \mbox{ and } 
\quad d_H(E, E') < |\pi_{E \ds F}|^{-1}\ .
$$
Let us write $d(\cdot, \cdot)$ instead of $d_H(\cdot, \cdot)$ for simplicity.

\begin{proof}[Proof of Lemma \ref{lem:openCond}] To prove $\Bc = E' \oplus F $,
we first show $E' \cap F = \{0\}$. If not, pick $e' \in E' \cap F$ with $|e'|=1$.
Since $d(e', S_E) \le d(E, E') < |\pi_{E \ds F}|^{-1}$, 
there exists $e \in E$ with $|e|=1$ such that $|e-e'| < |\pi_{E \ds F}|^{-1}$.
This is incompatible with
$$
1 = |e| = |\pi_{E \ds F} e|  = |\pi_{E \ds F} (e - e')| \leq |\pi_{E \ds F}| \cdot |e - e'| < 1\ .
$$

Next, we claim that $E' \oplus F$ is closed. It will suffice to show that there exists $A > 0$ such that for any $e' \in E', f \in F$, we have
\begin{align} \label{eq:verifyKober}
|e'| \leq A |e' + f|\ .
\end{align}
This is known as the Kober criterion \cite{kober}. Indeed, if \eqref{eq:verifyKober} holds and $x_n = e'_n + f_n$ is Cauchy, then $e_n'$ and $f_n$ individually are Cauchy, and thus converge to some $e' \in E', f \in F$ respectively, hence $x_n \to x := e' + f \in E' + F$.
To prove \eqref{eq:verifyKober}, pick arbitrary $e' \in E'$ and $f \in F$, and fix $c>1$ with
$c d(E,E')< |\pi_{E \ds F}|^{-1}$. As before, let
$e \in E$ be such that $|e|  = |e'|$ and  $|e - e'| \leq |e'| c d(E, E')$. 
 Then
$$
|e' + f| \geq |e + f| - |e - e'| \geq |e'| (|\pi_{E \ds F}|^{-1} - c d(E, E')) =: A^{-1} |e'| \ .
$$

To finish, assume for the sake of contradiction that $E' \oplus F \ne \Bc$.
By Assumption (ii), there exist $c_1<1<c_2 $ such that $c_2 d(E,E') |\pi_{E \ds F}| < c_1$.
Since $E' \oplus F$ is closed, the Riesz Lemma \cite{sch} asserts that there exists $x \in \Bc$ with $|x|=1$ such that $|x -  (e' + f)| \geq c_1$ for all $e' \in E', f \in F$. 
On the other hand, since $\Bc = E \oplus F$, we have that $x = e + f$ for some $e \in E, f \in F$;
notice that $|e| \le |\pi_{E \ds F}|$. But there exists $e' \in E$ with $|e'|=|e|$ and $|e-e'| \le c_2 
|e| d(E,E')$, and for such an $e'$, 
$$
|x - (e' + f)| = |e - e'| \leq |e| c_2 d(E, E') < c_1 \ ,
$$
contradicting our choice of $x$.
\end{proof}

\smallskip \noindent
{\bf Lemma A.1.} {\it Assume (*). Then} \vspace{-4pt}
\begin{itemize} 
\item[(i)] $$ |\pi_{E' \ds F}| \leq \frac{|\pi_{E \ds F}|}{1 -  |\pi_{E \ds F}| d(E, E')} \, , $$
\item[(ii)] $$ |\pi_{F \ds E'}|_E| \leq 2 |\pi_{E' \ds F}| d(E, E') \, . $$
\end{itemize}

%$$
%(i) \hskip 1in  |\pi_{E' \ds F}| \leq \frac{|\pi_{E \ds F}|}{1 -  |\pi_{E \ds F}| d(E, E')}
%\hskip 3in
%$$
%$$
%(ii) \hskip 1in  |\pi_{F \ds E'}|_E| \leq 2 |\pi_{E' \ds F}| d(E, E')
%\hskip 3in
%$$

\medskip
\begin{proof} Since $\Bc = E\oplus F = E'\oplus F$, (i) above is equivalent to
\begin{equation} \label{bound1}
\a(E, F) \leq d(E, E') + \a(E', F)
\end{equation}
by the formula $\a(E, F) = |\pi_{E \ds F}|^{-1}$ from Sect. 2.1.2. To
estimate $\a(E', F)$ from below, we let $e' \in E'$ with $|e'| = 1$ and $f \in F$ be arbitrary. 
For $c > 1$, we let $e \in E, |e| = 1$ be such that $|e - e'| \leq c d(E, E')$. Then,
$$
|e'  - f| \geq |e - f| - |e' - e| \geq \a(E, F) - c d(E, E')\ .
$$
But $e', f$ were arbitrary and so our formula follows on taking $c \to 1$.

\medskip
To prove (ii), fix $e \in E, |e| = 1$. Then for $c > 1$ arbitrarily close to $1$, let $e' \in E', |e'| = 1$ be such that $|e - e'| \leq c d(E, E')$. Then
$$
|\pi_{F \ds E'} e| = |\pi_{F \ds E'} (e - e')| \leq |\pi_{F \ds E'}| \cdot |e - e'| \leq 2 |\pi_{E' \ds F}| \cdot c d(E, E')\ .
$$
\end{proof}

\begin{proof}[Proof of Lemma \ref{lem:graNormEst}.] That $\Bc = E'\oplus F$ follows from Lemma \ref{lem:openCond}.
The bounds in (a) are given by Lemma A.1, and (b) follows from (a). 
\end{proof}

\bibliography{bibliography}

\begin{thebibliography}{10}

\bibitem{akh}
N.~I. Akhiezer and I.~M. Glazman.
\newblock {\em Theory of linear operators in Hilbert space}.
\newblock Courier Corporation, 1993.

\bibitem{barreira}
Luis Barreira and Yakov Pesin.
\newblock Smooth ergodic theory and nonuniformly hyperbolic dynamics.
\newblock {\em Handbook of dynamical systems}, 1:57--263, 2006.

\bibitem{blumenthal}
Alex Blumenthal.
\newblock A volume-based approach to the multiplicative ergodic theorem on
  banach spaces.
\newblock {\em arXiv preprint arXiv:1502.06554}, 2015.

\bibitem{bollo}
Bela Bollobas.
\newblock {\em Linear analysis, an introductory course}.
\newblock Cambridge University Press, 1999.

\bibitem{buse}
Herbert Busemann.
\newblock The geometry of finsler spaces.
\newblock {\em Bulletin of the American Mathematical Society}, 56(1):5--16,
  1950.

\bibitem{castaing1977convex}
Charles Castaing and Michel Valadier.
\newblock {\em Convex analysis and measurable multifunctions}.
\newblock Springer, 1977.

\bibitem{chang1997conditioning}
Joseph~T Chang and David Pollard.
\newblock Conditioning as disintegration.
\newblock {\em Statistica Neerlandica}, 51(3):287--317, 1997.

\bibitem{eckmannRuelle}
J-P Eckmann and David Ruelle.
\newblock Ergodic theory of chaos and strange attractors.
\newblock {\em Reviews of modern physics}, 57(3):617, 1985.

\bibitem{QFmet}
Gary Froyland, Simon Lloyd, and Anthony Quas.
\newblock A semi-invertible oseledets theorem with applications to transfer
  operator cocycles.
\newblock {\em arXiv preprint arXiv:1001.5313}, 2010.

\bibitem{QGTmet}
Cecilia Gonz{\'a}lez-Tokman and Anthony Quas.
\newblock A semi-invertible operator oseledets theorem.
\newblock {\em Ergodic Theory and Dynamical Systems}, pages 1--43, 2011.

\bibitem{quasConcise}
Cecilia Gonz{\'a}lez-Tokman and Anthony Quas.
\newblock A concise proof of the multiplicative ergodic theorem on banach
  spaces.
\newblock {\em arXiv preprint arXiv:1406.1955}, 2014.

\bibitem{henry}
Daniel Henry.
\newblock {\em Geometric theory of semilinear parabolic equations}, volume 840.
\newblock Springer-Verlag Berlin, 1981.

\bibitem{hirsch1977invariant}
Morris~W Hirsch, Michael Shub, and Charles~C Pugh.
\newblock {\em Invariant manifolds}.
\newblock Springer, 1977.

\bibitem{kato}
Tosio Kato.
\newblock {\em Perturbation theory for linear operators}.
\newblock Springer, 1995.

\bibitem{kober}
H~Kober.
\newblock A theorem on banach spaces.
\newblock {\em Compositio Mathematica}, 7:135--140, 1940.

\bibitem{kupka}
Joseph Kupka and Karel Prikry.
\newblock The measurability of uncountable unions.
\newblock {\em American Mathematical Monthly}, pages 85--97, 1984.

\bibitem{led}
Fran{\c{c}}ois Ledrappier.
\newblock Propri{\'e}t{\'e}s ergodiques des mesures de sinai.
\newblock {\em Publications Math{\'e}matiques de l'IH{\'E}S}, 59(1):163--188,
  1984.

\bibitem{ledstrel}
Fran{\c{c}}ois Ledrappier and Jean-Marie Strelcyn.
\newblock A proof of the estimation from below in pesin's entropy formula.
\newblock {\em Ergodic Theory and Dynamical Systems}, 2(02):203--219, 1982.

\bibitem{ledyou1}
Fran{\c{c}}ois Ledrappier and L-S Young.
\newblock The metric entropy of diffeomorphisms, part i: Characterization of
  measures satisfying pesin's entropy formula.
\newblock {\em Annals of Mathematics}, pages 509--539, 1985.

\bibitem{ledyou2}
Fran{\c{c}}ois Ledrappier and L-S Young.
\newblock The metric entropy of diffeomorphisms, part ii: Relations between
  entropy, exponents and dimension.
\newblock {\em Annals of Mathematics}, 122(3):540--574, 1985.

\bibitem{linshu1}
Zhiming Li and Lin Shu.
\newblock The metric entropy of random dynamical systems in a hilbert space:
  Characterization of invariant measures satisfying pesin's entropy formula.
\newblock {\em Discrete and Continuous Dynamical Systems}, 33(9):4123--4155,
  2013.

\bibitem{linshu2}
Zhiming Li and Lin Shu.
\newblock The metric entropy of random dynamical systems in a banach space:
  Ruelle inequality.
\newblock {\em Ergodic Theory and Dynamical Systems}, 34:1--22, 2014.

\bibitem{lianlu}
Zeng Lian and Kening Lu.
\newblock {\em Lyapunov exponents and invariant manifolds for random dynamical
  systems in a Banach space}.
\newblock American Mathematical Society, 2010.

\bibitem{lianyou1}
Zeng Lian and L-S Young.
\newblock Lyapunov exponents, periodic orbits and horseshoes for mappings of
  hilbert spaces.
\newblock {\em Annales Henri Poincar\'{e}}, 12(6):1081--1108, 2011.

\bibitem{lianyou2}
Zeng Lian and L-S Young.
\newblock Lyapunov exponents, periodic orbits, and horseshoes for semiflows on
  hilbert spaces.
\newblock {\em Journal of the American Mathematical Society}, 25(3):637--665,
  2012.

\bibitem{lianyou3}
Zeng Lian, L-S Young, and Chongchun Zeng.
\newblock Absolute continuity of stable foliations for systems on banach
  spaces.
\newblock {\em Journal of Differential Equations}, 254(1):283--308, 2013.

\bibitem{luwangyoung}
Kening Lu, Qiudong Wang, and Lai-Sang Young.
\newblock {\em Strange attractors for periodically forced parabolic equations},
  volume 224.
\newblock American Mathematical Soc., 2013.

\bibitem{maneDimension}
Ricardo Ma{\~n}{\'e}.
\newblock On the dimension of the compact invariant sets of certain non-linear
  maps.
\newblock In {\em Dynamical systems and turbulence, Warwick 1980}, pages
  230--242. Springer, 1981.

\bibitem{mane}
Ricardo Ma{\~n}{\'e}.
\newblock A proof of pesinÕs formula.
\newblock {\em Ergodic Theory Dynamical Systems}, 1(1):95--102, 1981.

\bibitem{Mmet}
Richardo Ma{\~n}{\'e}.
\newblock Lyapounov exponents and stable manifolds for compact transformations.
\newblock In {\em Geometric dynamics}, pages 522--577. Springer, 1983.

\bibitem{nuss}
Roger~D Nussbaum.
\newblock The radius of the essential spectrum.
\newblock {\em Duke Mathematical Journal}, 37(3):473--478, 1970.

\bibitem{pes}
Yakov~Borisovich Pesin.
\newblock Characteristic lyapunov exponents and smooth ergodic theory.
\newblock {\em Russian Mathematical Surveys}, 32(4):55--114, 1977.

\bibitem{rokhlin1949fundamental}
Vladimir~Abramovich Rokhlin.
\newblock On the fundamental ideas of measure theory.
\newblock {\em Matematicheskii Sbornik}, 67(1):107--150, 1949.

\bibitem{rokhlin1967lectures}
Vladimir~Abramovich Rokhlin.
\newblock Lectures on the entropy theory of measure-preserving transformations.
\newblock {\em Russian Mathematical Surveys}, 22(5):1--52, 1967.

\bibitem{ruelle}
David Ruelle.
\newblock An inequality for the entropy of differentiable maps.
\newblock {\em Bulletin of the Brazilian Mathematical Society}, 9(1):83--87,
  1978.

\bibitem{Rmet}
David Ruelle.
\newblock Characteristic exponents and invariant manifolds in hilbert space.
\newblock {\em Annals of Mathematics}, pages 243--290, 1982.

\bibitem{ruelle2}
David Ruelle.
\newblock The thermodynamic formalism for expanding maps.
\newblock {\em Communications in Mathematical Physics}, 125(2):239--262, 1989.

\bibitem{rund}
Hanno Rund.
\newblock {\em The differential geometry of Finsler spaces}.
\newblock Springer, 1959.

\bibitem{sch}
Martin Schechter.
\newblock {\em Principles of functional analysis}.
\newblock American Mathematical Soc., 1973.

\bibitem{thieu}
Philippe Thieullen.
\newblock Fibr{\'e}s dynamiques asymptotiquement compacts exposants de
  lyapounov. entropie. dimension.
\newblock In {\em Annales de l'institut Henri Poincar{\'e} (C) Analyse non
  lin{\'e}aire}, volume~4, pages 49--97. Gauthier-Villars, 1987.

\bibitem{walMET}
Peter Walters.
\newblock A dynamical proof of the multiplicative ergodic theorem.
\newblock {\em Transactions of the American Mathematical Society},
  335(1):245--257, 1993.

\bibitem{woj}
P~Wojtaszczyk.
\newblock {\em Banach Spaces for Analysts}.
\newblock Cambridge University Press, 1991.

\bibitem{young1998statistical}
Lai-Sang Young.
\newblock Statistical properties of dynamical systems with some hyperbolicity.
\newblock {\em Annals of Mathematics}, pages 585--650, 1998.

\bibitem{youngMathTheoryLE}
Lai-Sang Young.
\newblock Mathematical theory of lyapunov exponents.
\newblock {\em Journal of Physics A: Mathematical and Theoretical},
  46(25):254001, 2013.

\end{thebibliography}
\bibliographystyle{plain}
%\nocite{*}

\end{document}